\documentclass{amsart}
\usepackage[utf8]{inputenc}
\usepackage{amsmath}
\usepackage{amsfonts}
\usepackage{amsthm}
\usepackage{amssymb}
\usepackage{stmaryrd}
\usepackage{mathtools}
\usepackage{tikz}
\usepackage{tikz-cd}
\usepackage{float}
\usepackage{import}
\usepackage{longtable}
\usepackage{makecell}
\usepackage{geometry}
\restylefloat{table}

\theoremstyle{plain}
 \newtheorem{theorem}{Theorem}[section]
 \newtheorem{lemma}[theorem]{Lemma}
\newtheorem{corollary}[theorem]{Corollary}
\newtheorem{proposition}[theorem]{Proposition}
\newtheorem{definition}[theorem]{Definition}

\theoremstyle{definition} 
 \newtheorem{remark}[theorem]{Remark}

\usepackage{mathtools}
\usepackage{tikz}
\usepackage{tikz-cd}
\usepackage{longtable}
\usepackage{makecell}
\usepackage{float}
\makeatletter
\newcommand*{\da@rightarrow}{\mathchar"0\hexnumber@\symAMSa 4B }
\newcommand*{\da@leftarrow}{\mathchar"0\hexnumber@\symAMSa 4C }
\newcommand*{\xdashrightarrow}[2][]{%
  \mathrel{%
    \mathpalette{\da@xarrow{#1}{#2}{}\da@rightarrow{\,}{}}{}%
  }%
}
\newcommand{\xdashleftarrow}[2][]{%
  \mathrel{%
    \mathpalette{\da@xarrow{#1}{#2}\da@leftarrow{}{}{\,}}{}%
  }%
}
\newcommand*{\da@xarrow}[7]{%
  \sbox0{$\ifx#7\scriptstyle\scriptscriptstyle\else\scriptstyle\fi#5#1#6\m@th$}%
  \sbox2{$\ifx#7\scriptstyle\scriptscriptstyle\else\scriptstyle\fi#5#2#6\m@th$}%
  \sbox4{$#7\dabar@\m@th$}%
  \dimen@=\wd0 %
  \ifdim\wd2 >\dimen@
    \dimen@=\wd2 %   
  \fi
  \count@=2 %
  \def\da@bars{\dabar@\dabar@}%
  \@whiledim\count@\wd4<\dimen@\do{%
    \advance\count@\@ne
    \expandafter\def\expandafter\da@bars\expandafter{%
      \da@bars
      \dabar@ 
    }%
  }%  
  \mathrel{#3}%
  \mathrel{%   
    \mathop{\da@bars}\limits
    \ifx\\#1\\%
    \else
      _{\copy0}%
    \fi
    \ifx\\#2\\%
    \else
      ^{\copy2}%
    \fi
  }%   
  \mathrel{#4}%
}
\makeatother

\usepackage{hyperref}
  \usepackage[backend=bibtex, style=alphabetic,url=false,doi=false,isbn=false,sorting=nyt,maxbibnames=9,giveninits=true]{biblatex}
 \AtEveryBibitem{\clearfield{month}}
\AtEveryBibitem{%
  \clearlist{language}%
}
\newcommand\blfootnote[1]{%
  \begingroup
  \renewcommand\thefootnote{}\footnote{#1}%
  \addtocounter{footnote}{-1}%
  \endgroup
}

\title{K-moduli of log Fano complete intersections}
\author{Theodoros Stylianos Papazachariou}
\address{Isaac Newton Institute, University of Cambridge, 20 Clarkson rd, Cambridge, CB3 0EH, United Kingdom}
\email{tsp35@cam.ac.uk}

\begin{document}
% \fancyfoot{}
% \fancyfoot[L]{MSC2020 Primary: 14J10, 14J45, 14L24. Secondary: 14D23, 	14Qxx}

% \fancyfoot[L]{MSC2020 Primary: 14J10, 14J45, 14L24. Secondary: 14D23, 	14Qxx}

\begin{abstract}
We explicitly describe the K-moduli compactifications and wall crossings of log pairs formed by a Fano complete intersection of two quadric threefolds and a hyperplane, by constructing an isomorphism with the VGIT quotient of such complete intersections and a hyperplane. We further characterize all possible such GIT quotients based on singularities. 
We also explicitly describe the K-moduli of the deformation family of Fano $3$-folds $2.25$ in the Mori--Mukai classification, which can be viewed as blow ups of complete intersections of two quadrics in dimension three, by showing there exists an isomorphism to a GIT quotient which we also explicitly describe. 
Furthermore, we also construct computational algorithmic methods to study VGIT quotients of complete intersections and hyperplanes, which we use to obtain the explicit descriptions detailed above. We also introduce the reverse moduli continuity method, which allows us to relate canonical GIT compactifications to K-moduli of Fano varieties.

% Our work uses the moduli continuity method for log pairs by relating the K-moduli to certain GIT compactifications. In addition, we introduce the reverse moduli continuity method, which allows us to relate canonical GIT compactifications to K-moduli of Fano varieties. We also compute the CM line bundle for complete intersections of hypersurfaces of fixed degree with a hyperplane section, and we show it isomorphic to an ample line bundle in the Picard group of the canonical GIT quotient of complete intersections and a hyperplane. We use explicit GIT methods to classify in detail the GIT stability of a complete intersection of two quadrics in dimension four, i.e. a del Pezzo surface of degree $4$, together with a hyperplane section. We also classify the GIT stability of a complete intersection of two quadrics in dimension three. Furthermore, we explicitly compactify the moduli space of log Fano pairs of complete intersections and hyperplane sections, by establishing a direct link with GIT. 

%We use the above methods to prove our main result, namely the explicit description of the first wall crossing for the K-moduli of log Fano pairs, of complete intersections of two quadrics in dimension four, and a hyperplane section. The reverse moduli continuity method, alongside our computational GIT techniques, allow us to compactify the K-moduli of the deformation family of Fano $3$-folds $2.25$, which is one of the only three known explicit compactifications of K-moduli of Fano threefolds.

\end{abstract}
\maketitle
 \tableofcontents %% Just for papers exceeding 50 pages.

%TODO re-read and redo intro and abstract
\section{Introduction}\label{intro}

% Classification problems form a large part of the study of modern mathematics, and especially so in algebraic geometry. In the study of algebraic varieties, the construction of moduli spaces is a key step in classification problems. 

\blfootnote{MSC2020 Primary: 14J10, 14J45, 14L24. Secondary: 14D23, 	14Qxx} Moduli problems have been a key part of modern mathematics, with a number of different moduli constructions emerging in the past years, used to describe a wide variety of classification problems. For Fano varieties in particular, due to tremendous progress in the past decade (see e.g. \cite{odaka-moduli, Li-Wang-Xu, blum-xu, alper_reductivity,liu2021finite}), it has been shown that K-stability constructs a projective moduli space, named the K-moduli space. The K-moduli space parametrises K-polystable Fano varieties, and its construction has also been extended to the case of log Fano pairs \cite{ascher2019wall}. In this particular case, one obtains an array of K-moduli spaces and a wall-chamber decomposition, along with wall crossings, depending on the coefficient of the divisor. 

Although K-moduli are the right description of moduli spaces for Fano varieties, explicit descriptions are hard to obtain. As such, constructing explicit compactifications of moduli spaces of Fano varieties and determining when Fano varieties $S$ or log Fano pairs $(S,(1-\beta)D)$ are K-(poly)stable remain open questions. Up to now, explicit compactifications have been given for del Pezzo surfaces \cite{mabuchi_mukai_1990, odaka_spotti_sun_2016} and for few Fano threefolds \cite{liu-xu, abban2023onedimensional}, Fano fourfolds \cite{liu2020kstability} and Fano $n$-folds \cite{spotti_sun_2017} respectively. The case of log Fano pairs has seen extensive study recently; \cite{Gallardo_2020} studied the K-moduli spaces of log pairs consisting of cubic surfaces and anticanonical divisors in $\mathbb{P}^3$, while \cite{ascher2019wall, ascher2020kmoduli} studied the K-moduli of log pairs $(\mathbb{P}^n, cD)$ and $(\mathbb{P}^1\times\mathbb{P}^1, cD)$, where $D$ is an anticanonical divisor. Furthermore, \cite{fujita2017kstability} studied the K-stability of log Fano hyperplane arrangements, and more recently \cite{zhao_walls_2022} studied k-moduli wall crossings for quintic del Pezzo pairs $(X,cD)$, where $D\sim_{\mathbb{Q}}-2K_X$. The above papers rely on determining the walls and chambers manually, and in general cannot give a complete description of K-moduli walls and chambers.

In this paper, we study the K-moduli spaces $\overline{M}^K_{4,2,2}(\beta)$ parametrising K-polystable log Fano pairs $(S, (1-\beta)D)$, where $S$ is a complete intersection of two quadrics in $\mathbb{P}^4$ and $D = S\cap H$ is a hyperplane section, which are the good moduli space of the Artin stack $\mathcal{M}^K_{4,2,2}(\beta)$. We prove that for $\beta > \frac{3}{4}$ they are isomorphic to the VGIT quotients $\overline{M}^{GIT}_{4,2,2}(t(\beta))$ parametrising GIT-polystable pairs $(S,H)$. We show this result at the level of stacks.

\begin{theorem}[see Theorem \ref{iso of stacks}]\label{wall crossing thm}
Let $\beta > \frac{3}{4}$. There exists a natural isomorphism between the K-moduli stack $\mathcal{M}^K_{4,2,2}(\beta)$ of K-semistable log Fano pairs $(S, (1-\beta)D)$ and the GIT$_t(\beta)$-stack $\mathcal{M}^{GIT}_{4,2,2}(t(\beta))$ where 
$$t(\beta) = \frac{6(1-\beta)}{6-\beta}.$$
In particular, this isomorphism restricts to the good moduli spaces $\overline{M}^K_{4,2,2}(\beta)$ and
$\overline{M}^{GIT}_{4,2,2}(t(\beta))$.
\end{theorem}
As a corollary, we also obtain an explicit description of the first wall crossing.

\begin{corollary}[See Corollary \ref{Wall crossing}]\label{main corollary_intro}
    The first wall crossing occurs at $\beta = \frac{6}{7}$, $t = \frac{1}{6}$. Both the varieties and the divisors are deformed before and after the first wall.
\end{corollary}

\addtocontents{toc}{\protect\setcounter{tocdepth}{0}}
%\subsection*{Structure of the introduction}

In this paper, we further introduce a compactification $\overline{M}^{GIT}_{n,k,d}(t)$ depending on a parameter $t$ controlling the polarisation, for the moduli space of log Fano pairs formed by a complete intersection of hypersurfaces of the same degree and a hyperplane section.

% We will give some background on wall crossings for GIT and K-moduli. We will then proceed to describe how these compactifications are obtained via GIT, by detailing one of our results from our computational methods to study VGIT quotients. 

% We then detail the main results of this paper (Theorem \ref{wall crossing thm} and Corollary \ref{main corollary_intro}), where we explicitly describe of the wall crossing of the K-moduli space of a log Fano pair $(S, (1-\beta)D)$, where $S$ is a complete intersection of two quadrics in $\mathbb{P}^4$ and $D$ is an anticanonical divisor. This is the first example of K-moduli wall crossing, where both the variety and divisor admit deformations before and after the wall crossing. 

% We conclude by mentioning our compactification of the K-moduli of the deformation family of Fano $3$-folds $2.25$ in the Mori--Mukai classification, which are given as blow ups of $\mathbb{P}^3$ along a complete intersection of two quadrics, paralleling our GIT techniques. We note that this explicit description is achieved by relating GIT compactifications to K-moduli via the \emph{reverse moduli continuity method}, which we introduce in this paper. This is one of the three known explicit examples of K-moduli of Fano threefolds, after Spotti-Sun \cite{spotti_sun_2017} and Liu--Xu \cite{liu-xu}. 

\subsection*{Wall crossings in VGIT and K-moduli and computational VGIT}
%\textbf{Structure of the introduction}

% In our application, we study log Fano pairs $(S, (1-\beta)D)$ where $S$ is a del Pezzo surface of degree $4$ and $D$ is an anticanonical divisor. 

Of particular interest to us are GIT quotients parametrising pairs formed by complete intersections, $S$, of fixed degree and a hyperplane $H$. In this case, the GIT stability conditions tessellate into a wall-chamber decomposition \cite{thaddeus_1996, dolgachev_1994}. Hence, each GIT moduli space $\overline{M}^{GIT}_{n,d,k}(t)$ parametrising GIT$_t$-polystable pairs $(S,H)$ depends on a parameter $t$ controlling the polarisation. Two moduli spaces $\overline{M}^{GIT}_{n,d,k}(t)$ and $\overline{M}^{GIT}_{n,d,k}(t')$
%and each moduli space $\overline{M}^{GIT}(t)_{n,d,k}$ which parametrises GIT$_t$-polystable pairs $(S,H)$ is 
are isomorphic if and only if $t,t'$ belong to the same wall/chamber. This approach to GIT is known as Variations of GIT (VGIT). Furthermore, it is known that the number of walls is finite \cite{dolgachev_1994, thaddeus_1996}, and that there exist birational morphisms (Thaddeus flips) between different walls and chambers. Due to \cite{ascher2019wall}, similar wall crossing phenomena appear in the study of K-moduli of log Fano pairs $(S, (1-\beta)D)$. Hence, each K-moduli space $\overline{M}^{K}(\beta)$ parametrising K-polystable pairs depends on the parameter $\beta$ perturbing the divisor. 
% As before, two moduli spaces $\overline{M}^{K}(\beta)$ and $\overline{M}^{K}(\beta')$ are isomorphic if and only if $\beta$, $\beta'$ belong to the same wall/chamber.

% Explicit descriptions of wall crossing phenomena for K-moduli are rare. One of the few results on K-moduli wall crossing can be found in Ascher--DeVleming--Liu \cite{ascher2019wall}, where moduli stacks are used to study the K-moduli of log pairs of the form $(\mathbb{P}^n,cD)$, with a detailed analysis of plane curves. In \cite{ascher2020kmoduli}, the authors further study K-moduli of log Fano pairs $(\mathbb{P}^1\times\mathbb{P}^1,cC)$ where $C$ is a $(4,4)$-curve. Furthermore, in \cite{zhao_walls_2022}, the K-moduli for quintic del Pezzo pairs $(X,cD)$ for $0<c<1/2$ and $D\sim_{\mathbb{Q}}-2K_X$ are studied.

% In this paper, we consider $\overline{M}^K_{n,d,k}(\beta)$ which is the good moduli space of the Artin stack $\mathcal{M}^K_{n,d,k}(\beta)$ that parametrises K-polystable log Fano pairs $(S, (1-\beta)D)$ where $S$ is a complete intersection of $k$ hypersurfaces of degree $d$ in $\mathbb{P}^n$ and $D= S\cap H$ is a hyperplane section. As in Ascher--DeVleming--Liu \cite{ascher2019wall} a wall-chamber decomposition also occurs for the K-moduli stacks $\mathcal{M}^K_{n,d,k}(\beta)$. 

% \subsection*{Computational VGIT}

In studying these wall crossing phenomena in K-moduli, a natural method is to relate the K-moduli compactifications to specific VGIT quotients via the \emph{moduli continuity method for log Fano pairs} \cite{ascher2019wall, Gallardo_2020, ascher2020kmoduli}. In order to do so in our specific examples, our first step is to construct an  algorithmic approach to study the GIT of complete intersections of $k$ hypersurfaces in $\mathbb{P}^n$ of arbitrary same degree $d$, with $kd\leq n$, and a hyperplane section $D$. 

\begin{theorem}[See Theorem \ref{theorem-H not in supp S}]\label{parametrising GIT}
Let $S$ be a Fano complete intersection of $k$ hypersurfaces of degree $d$ in $\mathbb{P}^n$ and $H$ be a hyperplane in $\mathbb{P}^n$. Then $\overline{M}^{GIT}_{n,d,k}(t)$ parametrises closed orbits associated to pairs $(S, D = S\cap H)$, where $D$ is a divisor on $S$.
\end{theorem}

The main novelty of the above Theorem is that for the case of Fano varieties, the VGIT quotient $\overline{M}^{GIT}_{n,d,k}(t)$ which parametrises closed orbits associated to pairs $(S,H)$, in fact parametrises closed orbits associated to pairs $(S,D=S\cap H)$, where $\operatorname{Supp}(H)\not\subset \operatorname{Supp}(S)$. Hence, $D$ is itself a complete intersection. In particular, Theorem \ref{parametrising GIT} shows that in the above setting, the VGIT quotient $\overline{M}^{GIT}_{n,d,k}(t)$ provides a compact space parametrising pairs $(S,D)$, where $D$ is a divisor of the complete intersection $S$. Hence, $\overline{M}^{GIT}_{n,d,k}(t)$ becomes a natural quotient for the study of log Fano pairs $(S,cD)$, which is of particular interest in this paper. We note, that in the case of Calabi-Yau varieties, or varieties of general type, this is not true in general.

Extending our setting to tuples $(S, H_1,\dots, H_m)$ of complete intersections and hyperplanes, we extend our computational setting by constructing explicit computational methods that determine unstable and non-semistable tuples. Fixing coordinates, a tuple $(S,H_1,\dots, H_m)$ can be determined by homogeneous polynomials $f_1,\dots, f_k$ and $h_1,\dots, h_m$ of degrees $d$ and $1$, respectively. These define tuples of monomials, namely those which appear with non-zero coefficients in $f_1,\dots, f_k$ and $h_1,\dots, h_m$. Suppose $(S,H_1,\dots, H_m)$ is not $\vec{t}$-stable. Here the vector $\vec{t} = (t_1,\dots, t_m)$ determines the polarisation (see e.g \cite{gallardo_martinez-garcia_2019}). In Section \ref{sec:semi-dist fams} we define sets of monomials $N_{\vec{t}}^{\ominus}(\lambda, x^{J_1},\dots, x^{J_{k-1}}, x_{j_1},\dots,x_{j_m})$ explicitly, such that in some coordinate system the equations of $f_i$ and $h_1,\dots, h_m$ are given by monomials in $N_{\vec{t}}^{\ominus}(\lambda, x^{J_1},\dots, x^{J_{k-1}}, x_{j_1},\dots,x_{j_m})$. A similar procedure follows for $t$-unstable tuples, where the relevant sets of monomials are $N_{\vec{t}}^{-}$.
% $(\lambda, x^{J_1},\dots, x^{J_{k-1}}, x_{j_1},\dots,x_{j_m})$. 

\begin{theorem}[See Theorem \ref{strictly_semistable_k-case-vgit}]\label{thm 1.4}
    A tuple $(S,H_1,\dots,H_m)$ is not $\vec{t}$-stable ($\vec{t}$-unstable, respectively) if and only if there exists $g \in SL(n+1)$ such that the set of monomials associated to $g \cdot \big((\operatorname{Supp}(f_1)\times\dots\times\operatorname{Supp}(f_k) ), \operatorname{Supp}(h_1),\dots, \operatorname{Supp}(h_m)\big)$ is contained in a tuple $N^{\ominus}$ ($N^-$, respectively) defined in Lemma \ref{nminus_k-case-vgit}. The sets $N^{\ominus}$ and $N^-$ are maximal with respect to the containment order of sets and define families of non-$\vec{t}$-stable pairs ($\vec{t}$-unstable pairs, respectively). Any not $\vec{t}$-stable (respectively $\vec{t}$-unstable) tuple belongs to one of these families for some group element $g$.
    \end{theorem}

Although the existence of the sets satisfying the conditions of Theorem \ref{thm 1.4} is known to experts and has been explored before in different examples (c.f. \cite[\S 7]{mukai_2003}, \cite{gallardo_martinez-garcia_2018}), the main novelty of Theorem \ref{thm 1.4} is the explicit definition of the sets $N^\ominus$ and $N^-$ in terms of monomials (see Proposition \ref{annihilator_k-case-vgit}), which in turn allow us to completely determine GIT$_{\vec{t}}$ unstable elements via software \cite{theodoros_stylianos_papazachariou_2022}.

The above ideas allow us to construct a computational approach and algorithm (see Section \ref{sec:how to study VGIT}) to study VGIT quotients of complete intersections of hypersurfaces of the same degree, and hyperplanes. For the case of pairs, this algorithm has been implemented in a computer software package \cite{theodoros_stylianos_papazachariou_2022}. This algorithm receives as input the numbers $n$, $k$, $d$ and computes the values of walls and chambers. It then proceeds to compute the sets $N^\ominus$ and $N^-$, using their explicit Definitions in Section \ref{sec:semi-dist fams}, and keep only the maximal ones (see Section \ref{sec:semi-dist fams} for more details) for each wall and chamber. By Theorem \ref{thm 1.4}, each maximal set $N^\ominus$ and $N^-$ corresponds to either a non $t$-stable or $t$-unstable pair $(S,H)$, whose monomials are given, up to projective equivalence by an action of $\operatorname{SL}(n+1)$, by the monomials of the sets $N^\ominus$ and $N^-$. Hence, the algorithm generates all maximal unstable and non-stable families for each VGIT quotient. The detailed explanation of this algorithm is provided in Section \ref{sec:how to study VGIT}. When $kd\leq n$, due to with Theorem \ref{parametrising GIT}, this algorithm generates all non $t$-stable and $t$-unstable pairs $(S,D)$ for each wall and chamber, up to projective equivalence. Hence, this computational approach allows the explicit GIT classification of specific log Fano pairs.

In this paper we employ the computational methods of Section \ref{VGITsection} to classify in detail the  GIT compactifications of log pairs, for all walls and chambers of pairs $(S,D)$, where $S$ is a complete intersection of two quadrics in $\mathbb{P}^4$ and $D$ is a hyperplane section, and for varieties $X$, where $X$ is a complete intersection of two quadrics in $\mathbb{P}^3$. This is integral in proving our main results (Theorem \ref{wall crossing thm} and Corollary \ref{main corollary_intro} and Theorem \ref{fano 3fold thm}).

\subsection*{Wall crossing of the K-moduli of log Fano pairs}

%, or the machinery presented in Ascher-DeVlemming-Liu \cite{ascher2019wall}. 

%Using GIT, we construct a projective variety $\overline{M}^{GIT}_{n,d,k}(t)$ which parametrizes equivalence classes of semistable complete intersections of $k$ hypersurfaces of degree $d$. 

The CM line bundle, originally introduced by Paul and Tian in the absolute case \cite{paultian2006}, 
% and extended for log Fano pairs by Gallardo--Martinez-Garcia--Spotti \cite{Gallardo_2020}, i
is an invariant which plays an important role in the link between K-stability and GIT-stability. We make such an explicit link between the K-stability of log Fano pairs $(S,(1-\beta)D)$, for $\beta \in (0,1)\cap \mathbb{Q}$ and the GIT$_{t}$ stability of pairs $(S,D = S\cap H)$ by computing the log CM line bundle in Section \ref{cm_linebundle section}. This allows us to get an explicit relation between $t$ and $\beta$, and is the first step to relating K-moduli and VGIT quotients. In particular, by explicitly constructing a family $f\colon \mathcal{X}\rightarrow \mathcal{T}$ parametrising complete intersections and hyperplane sections, we show that the CM line bundle is ample, and we prove that $t$ can be expressed as a function of $\beta$ (Lemma
\ref{calculating the line bundle}). We focus here on the case where $S$ is a complete intersection of two quadrics in $\mathbb{P}^4$ and $D = S\cap H$ is an anticanonical divisor.

\begin{corollary}[See Corollary \ref{cm of ci in p4}]\label{CM line bundle}
    For any rational number $0<\beta<1$, we have $$\Lambda_{\operatorname{CM},\beta}\simeq \mathcal{O}(2(6-\beta),12(1-\beta))$$ as $\mathbb{Q}$-line bundles, which is ample. In particular,  $\Lambda_{\operatorname{CM},\beta}$ is ample for all $c\in (0,1)\cap \mathbb Q$ and $t(\beta)=\frac{6(1-\beta)}{6-\beta}$.
\end{corollary}

The description for the general case is more complicated so we direct the reader to Section \ref{cm_linebundle section} for a more detailed description.

Focusing further on the case where $S$ is a complete intersection of two quadrics in $\mathbb{P}^4$ and $D = S\cap H$ is an anticanonical divisor, the K-moduli parametrising the K-polystable log Fano pairs $(S, (1-\beta)D$ is $\overline{M}^K_{4,2,2}(\beta)$, and the GIT moduli space parametrising GIT$_t$ polystable pairs $(S,D)$ is $\overline{M}^{GIT}_{4,2,2}(t(\beta))$. We show the following theorem.

\begin{theorem}[See Theorem \ref{k-stab implies git stab}]
    Let $(S,D)$ be a log Fano pair, where $S$ is a complete intersection of two quadrics in $\mathbb{P}^4$ and $D$ an anticanonical section. Suppose $(S, (1-\beta)D)$ is log $K$- (semi/poly)stable. Then, $(S,D)$ is GIT$_{t(\beta)}$-(semi/poly)stable, with slope $t(\beta) = \frac{b(\beta)}{a(\beta)} = \frac{6(1-\beta)}{6-\beta}$.
\end{theorem}

We should note that without explicit knowledge of GIT unstable, semistable and stable pairs given by the computational methods detailed above, described in more detail in Section \ref{p4 VGIT section}, the above Theorem would be unattainable in complete generality, and could only be proven for specific cases depending on singularities. This implies that, at least in the case of complete intersections, the interaction between GIT and K-stability has to be studied to some extent via explicit computational methods, which are particularly aided by the results of Section \ref{VGITsection}. Using our classification of $t$-(poly)/stable pairs via isolated singularities in Section \ref{p4 VGIT section}, we prove our main result, Theorem \ref{wall crossing thm}, as well as the explicit description of the first wall crossing in Corollary \ref{main corollary_intro}. We should also note, that after the first appearance of this paper, we were able to extend the results of Theorem \ref{iso of stacks} to all walls and chambers in \cite{martinezgarcia2024kmodulilogdelpezzo}, by making extensive use of the main results of this paper. In particular, Theorem \ref{k-stab implies git stab} and Corollary \ref{cm of ci in p4} are necessary for extending the results of Theorem \ref{iso of stacks} in \cite{martinezgarcia2024kmodulilogdelpezzo}.

\subsection*{Compactification of the K-moduli space of Family 2.25}
To conclude our analysis into K-moduli spaces of complete intersections, we explicitly compactify the K-moduli space of family $2.25$ in the Mori--Mukai classification. The general element of family $2.25$, $X$, is given as the blow-up of $\mathbb{P}^3$ along a complete intersection of quadrics, i.e. $X = \operatorname{Bl}_{C_1\cap C_2} \mathbb{P}^3$. We denote the K-moduli space parametrising K-semistable varieties in this family as $\mathcal{M}^{K}_{2.25}$, with corresponding moduli space $\overline{M}^K_{2.25}$. Using the computational methods of Section \ref{VGITsection} and the computational programme \cite{theodoros_stylianos_papazachariou_2022}, we also explicitly describe the GIT moduli stack $\mathcal{M}^{GIT}_{3,2,2}$ parametrising GIT-semistable complete intersections of quadrics in $\mathbb{P}^3$ algorithmically. We also describe the good moduli space $\overline{M}^{GIT}_{3,2,2}$ explicitly. Then, we prove the following.

\begin{theorem}[see Theorem \ref{2.25 compactification}]\label{fano 3fold thm}
There exists an isomorphism between $\mathcal{M}^{K}_{2.25}$  and $\mathcal{M}^{GIT}_{3,2,2}$. In particular, this restricts to an isomorphism of good moduli spaces $\overline{M}^{K}_{2.25}\cong \overline{M}^{GIT}_{3,2,2}\cong \mathbb{P}^1$.
\end{theorem}

The main novelty in proving the above Theorem is the introduction and use of the \emph{reverse moduli continuity method}, which adapts the moduli continuity method, which was introduced for del Pezzo surfaces in \cite{odaka_spotti_sun_2016}, and applied for quartic del Pezzo varieties in \cite{spotti_sun_2017} and for cubic threefolds in \cite{liu-xu}. The first instance of a similar method was presented by Mabuchi-Mukai \cite{mabuchi_mukai_1990}, which studied the link between GIT stability and the existence of K{\"a}hler-Einstein metrics on del Pezzo surfaces of degree $4$.

The novelty of the reverse moduli continuity method is that we obtain a map from the GIT quotient stack to the K-moduli stack, which is the reverse direction to the moduli continuity method. The benefits of this method are that if we can explicitly describe the GIT quotients and show that GIT polystability implies K-polystability, then we can immediately obtain isomorphisms on the level of stacks as in Theorem \ref{fano 3fold thm}. We note that although this method is applied to this specific K-moduli problem, it works in complete generality when one can show that all GIT polystable orbits are K-polystable, which is feasible due to recent computational methods in studying K-stability. We expect this technique to be useful in future work. Since one needs to have an explicit description of the GIT quotient in order to apply the reverse moduli continuity method, we expect that a combination of the above computational techniques and methods in K-moduli theory will have broad application, especially in obtaining explicit descriptions of low-dimensional K-moduli spaces, where the GIT quotients can be studied computationally.

\subsection*{Outline of the paper}

The outline of the paper is the following: In Section \ref{section_prelims} we make a brief introduction to the K-stability preliminaries used throughout the paper. In section \ref{VGITsection} we study the Variational GIT of tuples $(S,H_1,\dots, H_m)$ of complete intersections $S$ and hyperplanes $H_i$. We will do so algorithmically, expanding upon the theory presented in \cite{gallardo_martinez-garcia_2018} and \cite{zhang_2018}. This algorithm allows us to computationally find GIT walls and chambers (Theorem \ref{git walls}), and specific unstable families (Theorem \ref{unstable families vgit}). In this section, we also prove Theorem \ref{parametrising GIT}. 

In Section \ref{p3 VGIT section} we introduce some necessary preliminaries in the study of singularities of the complete intersection of two quadrics in $\mathbb{P}^n$ and we classify all complete intersections of two quadrics in $\mathbb{P}^3$ based on singularities, by giving explicit equations. We use this in conjunction with the computational methods of Section \ref{VGITsection} and the computational program \cite{theodoros_stylianos_papazachariou_2022} to classify in completeness the GIT (semi/poly-)stable points for the complete intersection of two quadrics in $\mathbb{P}^3$. Then, in Section \ref{fano 3fold section} we prove that GIT polystable points correspond to K-polystable points, and we introduce the reverse moduli continuity method. We then proceed to use the reverse moduli continuity method to prove Theorem \ref{fano 3fold thm}. In section \ref{p4 VGIT section} we classify all pairs $(S,D)$ where $S$ is a complete intersection of two quadrics in $\mathbb{P}^4$ and $D$ is a hyperplane section based on their singularities, and we provide explicit equations. We use our VGIT algorithm to classify all GIT$_t$ stable and polystable log pairs $(S,D)$ for each wall and chamber. These serve to obtain the necessary tools to prove Theorem \ref{wall crossing thm} and Corollary \ref{main corollary_intro}. In Section \ref{cm_linebundle section} we calculate the log CM line bundle of log Fano pairs $(S,(1-\beta)D)$ which gives us the explicit link between VGIT and K-moduli, and we show that for the above setting K-(semi/poly-)stable implies VGIT (semi/poly-)stable. Using the above, we prove our main Theorem, Theorem \ref{wall crossing thm} in Section \ref{main theorem section}. 

\renewcommand{\abstractname}{Acknowledgements}
\begin{abstract}
I would like to thank my PhD supervisor, Jesus Martinez-Garcia for his useful comments and suggestions in improving this draft. I would also like to thank Kenneth Ascher, Ivan Cheltsov, Ruadha{'i} Dervan, Patricio Gallardo, Yuji Odaka and Junyan Zhao for useful discussions and suggestions. This paper is part of my PhD thesis at the University of Essex, funded by a Mathematical Sciences Scholarship. Parts of this paper were finished during my postdoctoral position at the University of Glasgow, funded by Ruadha{\'i} Dervan's Royal Society University Research Fellowship. I would also like to thank my PhD defence examiners, Chenyang Xu and Gerald Williams, for many useful comments and suggestions in improving this draft.  
\end{abstract}

\addtocontents{toc}{\protect\setcounter{tocdepth}{2}}

\section{Preliminaries}\label{section_prelims}

We recall some definitions and theorems on K-stability that will be useful throughout the paper.
\begin{definition}
Let $X$ be a normal variety and $D$ an effective $\mathbb{Q}$-Weil divisor on $X$. The pair $(X,D)$ is called a \emph{log pair}. If, further, $X$ is projective and $(-K_X-D)$ is $\mathbb{Q}$-Cartier and ample the pair is called a \emph{log Fano pair}. If $(X,D)$ is klt, then we say that it is klt log Fano. 
$X$ is a \emph{$\mathbb{Q}$-Fano variety} if $X$ is a klt Fano variety.  
\end{definition}

%A pair $(X,L)$ where $X$ is a projective variety and $L$ is an ample line bundle is called a polarized pair. 
Below is the definition for the test configurations of polarised varieties due to Tian \cite{tian_test} and Donaldson \cite{Donaldson_test}. 

\begin{definition}
Let $(X,L)$ be a polarized variety of dimension $n$. A \emph{test configuration} $(\mathcal{X},\mathcal{L})/\mathbb{A}^1$ for $(X,L)$ consists of the following data:
\begin{enumerate}
    \item a normal variety $\mathcal{X}$ together with a surjective projective morphism $\pi\colon \mathcal{X}\rightarrow \mathbb{A}^1$;
    \item a $\pi$-ample line bundle $\mathcal{L}$;
    \item a $\mathbb{G}_m$-action on $(\mathcal{X},\mathcal{L})$ such that $\pi$ is $\mathbb{G}_m$-equivariant with respect to the natural action of $\mathbb{G}_m$ on $\mathbb{A}^1$;
    \item $(\mathcal{X}\setminus \mathcal{X}_0, \mathcal{L}|_{\mathcal{X}\setminus \mathcal{X}_0})$ is $\mathbb{G}_m$-equivariantly isomorphic to $(X,L)\times (\mathbb{A}^1\setminus \{0\})$.
\end{enumerate}

\end{definition}

We may write 
$$w(k) = b_0k^{n+1}+b_1k^n+\dots$$
for the weight of the determinant vector space $\det H^0(\mathcal{X}_0,\mathcal{L}^{\otimes m}_0)$, and  
$$N_m\vcentcolon=h^0(X,L^{\otimes m}) = a_0m^n+a_1m^{n-1}+\dots$$ 
for the Hilbert polynomial. We then can define the Donaldson Futaki invariant of the test configuration.

\begin{definition}\label{df invariant def}
The \emph{Donaldson-Futaki invariant} of a test configuration $(\mathcal{X};\mathcal{L})$ is 
$$\operatorname{DF}(\mathcal{X},\mathcal{L})\vcentcolon=2\frac{b_0a_1-a_0b_1}{a_0}.$$
\end{definition}

Using the Donaldson-Futaki invariant, we can define when a polarised variety is K-(semi)stable.
 
\begin{definition}\label{k-stability def}\cite{tian_test}, \cite{Donaldson_test}
We say that:  
\begin{enumerate}
    \item $X$ is \emph{K-semistable} if $\operatorname{DF}(\mathcal{X},\mathcal{L})\geq0$ for all non-trivial test configurations; 
    \item $X$ is \emph{K-stable} if it is K-semistable and  $\operatorname{DF}(\mathcal{X},\mathcal{L})=0$ only for product test configurations;
    \item $X$ is \emph{K-polystable} if it is K-semistable and  $\operatorname{DF}(\mathcal{X},\mathcal{L})=0$ only for trivial test configurations
%    \item If $X$ is Fano, the polarized pair $(X,-K_X)$ admits a K{\"a}hler-Einstein metric if and only if it is K-polystable. \cite{chen_donaldson_sun_2013}
\end{enumerate}

\end{definition}

Writing further $\tilde w(k) = \tilde b_0k^n+\dots$ for the weight of the determinant vector space  $\det H^0(\mathcal{D}_0,\mathcal{L}^{\otimes m}|_{\mathcal{D}_0})$,  we can define the $\beta$-Donaldson Futaki invariant of  $(\mathcal{X},\mathcal{D},\mathcal{L})$ \cite{donaldson_problem}.

\begin{definition}[{\cite[Definition 3.2]{odaka_sun_2015}}]\label{beta-df invariant def}
For $\beta \in (0,1]\cap \mathbb{Q}$ the  $\beta$-\emph{Donaldson-Futaki invariant} for the test configuration $(\mathcal{X},\mathcal{D},\mathcal{L})$ is  a numerical invariant 
$$\operatorname{DF}_{\beta}(\mathcal{X},\mathcal{L})\vcentcolon=2\frac{b_0a_1-a_0b_1}{a_0}+(1-\beta)\frac{\tilde b_0a_1-a_0\tilde b_1}{a_0}$$
\end{definition}

We can now define K-stability for pairs $(X,(1-\beta)D, L)$.

\begin{definition}\label{beta k-stability def}
An pair $(X,(1-\beta)D, L)$ is
\begin{enumerate}
    \item \emph{K-semistable} if $\operatorname{DF}_{\beta}(\mathcal{X},\mathcal{L})\geq0$ for all non-trivial test configurations; 
    \item \emph{K-stable} if it is K-semistable and  $\operatorname{DF}_{\beta}(\mathcal{X},\mathcal{L})=0$ only for trivial test configurations;
    \item $X$ is \emph{K-polystable} if it is K-semistable and  $\operatorname{DF}_{\beta}(\mathcal{X},\mathcal{L})=0$ only for test configurations equivariantly isomorphic to the trivial test configuration.
\end{enumerate}
\end{definition}

Detecting K-stability can be a challenging process, except in specific cases, such as toric Fano varieties. In particular, we will make use of the following Theorem in Section \ref{fano 3fold section}, when we prove that specific singular toric Fano varieties are K-polystable.

\begin{theorem}[\cite{batyrev}, {\cite[Theorem 1.2]{fujita_volume}}, {\cite[Corollary 1.2.]{Berman}}]\label{toric batyrev}
Let $X$ be a normal toric Fano variety, and let $P$ be its associated
anticanonical polytope in $M \otimes_{\mathbb{Z}} \mathbb{R}$, where $M$ be the character lattice of the torus. Then $X$ is K-polystable if and only if the barycentre of $P$ is the origin.
\end{theorem}

Now let $\pi \colon \mathcal{X}\rightarrow T$ be a flat proper morphism of relative dimension $n$, $\mathcal{D}\subseteq \mathcal{X}$ an effective $\mathbb{Q}$-Weil divisor, with a restriction $\pi_{\mathcal{D}}$ also proper and of relative dimenion $n-1$. 

For sufficiently large $k>0$ the Knudsen-Mumford theorem \cite{knudsen} says that there exist functorially defined line bundles $\lambda_j \vcentcolon=\lambda_j(\mathcal{X},\mathcal{B},\mathcal{L})$, $\tilde\lambda_j \vcentcolon=\lambda_j(\mathcal{X},\mathcal{B},\mathcal{L}|_{\mathcal{D}})$ such that 

\begin{equation*}
    \begin{split}
     \det(\pi!_*(\mathcal{L}^k)) &= \bigotimes^{n+1}_{i=1}\lambda_i^{\binom{k}{i}}\\
     \det\big(\pi!_*((\mathcal{L}|_{\mathcal{D}})^k)\big) &= \bigotimes^{n}_{i=1}\tilde{\lambda}_i^{\binom{k}{i}}.
    \end{split}
\end{equation*}

\begin{definition}\label{log cm line bundle}\cite{paultian2006, Gallardo_2020}
Given a tuple $(\mathcal{X},\mathcal{D},T,\mathcal{L})$ its \emph{log CM line bundle} with angle $\beta \in \mathbb{Q}_{>0}$ on $T$ is 
$$\Lambda_{CM,\beta}(\mathcal{X},\mathcal{D},\mathcal{L}) = \lambda_{n+1}^{\otimes \Big(n(n+1) +\frac{2a_1-(1-\beta)\tilde a_0}{\alpha_0}\Big)}\otimes\lambda_n^{\otimes (-2(n+1))}\otimes \tilde \lambda_n^{\otimes (1-\beta)(n+1)} .$$

\end{definition}
We then have the following two results due to 
Phong--Ross--Sturm \cite{phong_ross_sturm_2008} and
Gallardo--Martinez-Garcia--Spotti \cite{Gallardo_2020} which extended to the log Fano pair setting:
\begin{theorem}[\cite{phong_ross_sturm_2008}, {\cite[Theorem 2.6]{Gallardo_2020}}]\label{gmgs thm 2.6}
Let $T= \mathbb{C}$ and suppose that $(\mathcal{X},\mathcal{D},\mathbb{C},\mathcal{L})$ is a test configuration for an $L$-polarised pair $(X,D)$. Then:
$$w(\Lambda_{CM,\beta}(\mathcal{X},\mathcal{D},\mathcal{L}^r)) = (n+1)!\operatorname{DF}_{\beta}(\mathcal{X},\mathcal{D},\mathcal{L})$$
\end{theorem}

\begin{theorem}[{\cite[Theorem 2.7]{Gallardo_2020}}]\label{gmgs thm 2.7}
Let $(X,D,L)$ be the restriction of a family  $(\mathcal{X},\mathcal{D},\mathcal{L})$ to a general $t \in T$. If $\mathcal{L} = -K_{\mathcal{X}/T}$ and $\mathcal{D}|_{\mathcal{X}_t}\in |-K_{\mathcal{X}_t}|$ for all $t \in T$ then 
\begin{equation*}
    \begin{split}
        \deg(\Lambda_{CM,\beta}) = &-(1+n(1-\beta))\pi_* \Big(c_1(-K_{\mathcal{X}/T})^{n+1} \Big)\\
        &+ (1-\beta)(n+1)\pi_* \Big(c_1(-K_{\mathcal{X}/T})^{n}\cdot \mathcal{D} \Big)
    \end{split}
\end{equation*}
%$$\deg(\Lambda_{CM,\beta}) = -(1+n(1-\beta))\pi_* \Big(c_1(-K_{\mathcal{S}/\mathbb{P}^1})^{n+1} \Big) + (1-\beta)(n+1)\pi_* \Big(c_1(-K_{\mathcal{S}/\mathbb{P}^1})^{n}\cdot \mathcal{D} \Big).$$
\end{theorem}

\section{Computational Variation of GIT}\label{VGITsection}
In this section, we will study how to generalise results in \cite{laza-cubics,gallardo_martinez-garcia_2018, zhang_2018}, in order to obtain a computational toolkit that will allow us to study specific GIT problems algorithmically. In more detail, we will consider GIT quotients of tuples $(X,H_1,\dots,H_m)$ where $X$ is the complete intersection of $k$ hypersurfaces of degree $d$ in $\mathbb{P}^n$, and the $H_i$ are hyperplanes. We will explain what makes these types of quotients variational, and we will explain how one can define a Hilbert-Mumford numerical criterion in order to study these quotients. 

We will then proceed to construct the computational setting necessary for our analysis. We will introduce a finite \textit{fundamental set of one-parameter subgroups}, which, roughly, determines which of these tuples are not-stable/unstable with respect to a specific polarisation. Using this, we will demonstrate, with extra care in the case of pairs $(X,H)$ how one can computationally obtain the walls and chambers of this VGIT problem.

We will also use a polyhedral criterion, the \textit{Centroid Criterion}, following \cite[Theorem 9.2]{dolgachev_1994} and \cite{gallardo_martinez-garcia_2018}, which will allow us to distinguish between stable, and strictly semistable tuples. We will demonstrate how this, in addition with the extra condition that $X$ is Fano, shows that in the case of pairs, the GIT quotient parametrises log pairs $(X, D= S\cap H)$. We also compute the dimension of the VGIT quotient.

We conclude the computational side of VGIT by introducing the concept of  semi\--de\-stab\-ilising families, following \cite{gallardo_martinez-garcia_2018}, which can roughly be thought as sets of weights which in turn will define the polynomials in the support of the pair of complete intersection and hyperplane such that these are unstable/non-stable for a specific parameter $t$. We will show that these families are maximal, in the sense that a pair $(X,H)$ which is unstable/non-stable for some $t$ must have their weights in these families.

%TODO fill in this intro here

\subsection{Preliminaries}\label{VGIT_prelims}
Throughout this Section, we will work over an algebraically closed field $k$. 
Let $G\coloneqq \operatorname{SL}(n+1)$. Consider a variety $S$ which is the complete intersection of $k$ hypersurfaces of degree $d$ in $\mathbb{P}^n$, i.e, $S = \{f_1 = f_2 = \dots = f_k=0\}$, where each $f_i(x) = \sum f_{I_i}x^{I_i} $ with $I_i = \{d_{i,0}, \dots, d_{i,n} \}$, $\sum_{j=0}^n d_{i,j}=d$ for all $i$. Here, $x^{I_i} = x_0^{d_{i,0}} x_1^{d_{i,1}}\dots x_n^{d_{i,n}}$. Also consider $H_1,\dots, H_m$, $m$ distinct hyperplanes with defining polynomials $h_i(x) = \sum h_{i,j}x_j$.  Let $\Xi_d$ be the set of monomials of degree $d$ in variables $x_0, \dots, x_n$, written in the vector notation $I_i =(d_{i,0},\dots, d_{i,n})$. As in Gallardo --Martinez-Garcia \cite{gallardo_martinez-garcia_2018, zhang_2018} we define the associated set of monomials $$\operatorname{Supp}(f_i)= \{x^{I_i} \in \Xi_d|f_{I_i} \neq 0\},\quad \operatorname{Supp}(h_i)= \{x_j \in \Xi_1|h_{i,j} \neq 0\}.$$

Let $V \vcentcolon= \operatorname{H}^0(\mathbb{P}^n, \mathcal{O}_{\mathbb{P}^n}(1))$ and $W \vcentcolon= \operatorname{H}^0(\mathbb{P}^n, \mathcal{O}_{\mathbb{P}^n}(d))  \simeq \operatorname{Sym}^d V $ be the vector space of degree $d$ forms.  For an embedded variety $S = (f_1,\dots, f_k) \subseteq \mathbb{P}^n $ we associate its Hilbert point

$$[S] = [f_1 \wedge \dots \wedge f_k] \in \operatorname{Gr}(k, W) \subset \mathbb{P} \bigwedge^k W.$$
Note that $\operatorname{Gr}(k, W)$ is embedded in $\mathbb{P} \bigwedge^k W$ via the Pl{\"u}cker embedding $(w_1,\dots, w_r)\rightarrow$ \\ $[w_1\wedge\dots\wedge w_r]$, where the $w_r$ are the basis vectors of $W$.  We denote by $[\overline{S}] \vcentcolon= f_1 \wedge \dots \wedge f_k$ some lift in $\bigwedge^k W$. We will consider the natural $G$ action, given by $A\cdot f(x) = f(Ax)$ for $A\in G$.

For simplicity, we will denote $\operatorname{Gr}(k,W)$ by $\mathcal{R}_{n,d,k}$, and we let $\mathcal{R}_m \vcentcolon= \mathcal{R}_{n,d,k}\times \Big(\mathcal{R}_{n,1,1}\Big)^m$ be the parameter scheme of tuples $(f_1,\dots f_k, h_1, \dots, h_m)$,  where:
$$\mathcal{R}_{m} \cong \operatorname{Gr}\Big(k,{\binom{n+d}{d}}\Big) \times \Big(\mathbb{P}(H^0(\mathbb{P}^n, \mathcal{O}_{\mathbb{P}^n}(1)))\Big)^m \hookrightarrow \mathbb{P}\bigwedge^k W \times (\mathbb{P}^n)^m.$$
%$$\mathcal{R} \vcentcolon= \operatorname{Gr}\Big(k,{{n+d}\choose{d}}\Big) \times \mathbb{P}(H^0(\mathbb{P}^n, \mathcal{O}_{\mathbb{P}^n}(1))) \hookrightarrow \mathbb{P}\bigwedge^k W \times \mathbb{P}^n.$$
In the case where $m = 1$, we will just write $\mathcal{R} = \mathcal{R}_1$. There is a natural $G$ action on $V$ and $W$ given by the action of $G$ on $\mathbb{P}^n$. This action induces an action of $G$ on $\mathcal{R}_{n,d,k}$ via the natural maps, and by the inclusion map to the Pl{\"u}cker embedding $\mathbb{P}\bigwedge^k W$. By extension, we also obtain an induced action of $G$ to $\mathcal{R}$. We aim to study the GIT quotients $\mathcal{R}_m\sslash G$.

%In classical GIT, the variety $S$ is \textit{semi-stable}, under the $G$-action on $\mathcal{R}_{n,d,k}$, induced by $W$, if and only if $0 \not \in \operatorname{SL(n+1)}\cdot [\overline{S}] $ and $S$ is \textit{stable}, under the $G$-action, if and only if $\operatorname{SL(n+1)}\cdot [\overline{S}]$ is closed and the stabilizer $\operatorname{SL(n+1)_{[\overline{S}]}}$ is finite (see \cite[Chapter 7]{mukai_2003}). Furthermore, $S$ is called \textit{unstable} if it is not semi-stable \cite{fedorchuk}. 

Let $\mathcal{C}\coloneqq\mathbb{P}(W)$. We will begin our analysis by first studying the GIT quotients $\mathcal{C}\sslash G$. $\mathcal{C}$ parametrises hypersurfaces $X = \{f=0\}$ of degree $d$ in $\mathbb{P}^n$, where $f = \sum f_Ix^I$ is a polynomial of degree $d$.  We must study the Hilbert-Mumford numerical criterion in order to study this quotient. In order to do so, we fix a maximal torus $T\cong (\mathbb{G}_m)^n \subset G$ which in turn induces lattices of characters $M = \operatorname{Hom}_{\mathbb{Z}}(T, \mathbb{G}_m) \cong \mathbb{Z}^{n+1}$ and one-parameter subgroups $N = \operatorname{Hom}_{\mathbb{Z}}( \mathbb{G}_m, T) \cong \mathbb{Z}^{n+1}$ with natural pairing 
$$\langle - , -\rangle \colon M\times N \rightarrow \operatorname{Hom}_{\mathbb{Z}}( \mathbb{G}_m, \mathbb{G}_m) \cong \mathbb{Z}$$
%Working in an algebraically closed field $\mathbb{K}$, where $\operatorname{char}(\mathbb{K}) = 0$ (we will work over $\mathbb{C}$ unless stated otherwise), w
%We fix a maximal torus $T\cong \mathbb{G}_m^n \subset G$ which in turn induces lattices of characters $M = \operatorname{Hom}_{\mathbb{Z}}(T, \mathbb{G}_m) \cong \mathbb{Z}^{n+1}$ and one-parameter subgroups $N = \operatorname{Hom}_{\mathbb{Z}}( \mathbb{G}_m, T) \cong \mathbb{Z}^{n+1}$ with natural pairing 
%$$\langle - , -\rangle \colon M\times N \rightarrow \operatorname{Hom}_{\mathbb{Z}}( \mathbb{G}_m, \mathbb{G}_m) \cong \mathbb{Z}.$$
given by the composition $(\chi,\lambda) \mapsto \chi \circ \lambda$. We also choose projective coordinates $(x_0 \colon \dots \colon x_n)$ such that the maximal torus $T$ is diagonal in $G$. Given a one-parameter subgroup $\lambda \colon$\\ $ \mathbb{G}_m \rightarrow T \subset SL(n+1)$ we say $\lambda$ is \textit{normalised} if 
$$\lambda(s) = \operatorname{Diag} (s^{\mu_0}, \dots, s^{\mu_n})$$ where $\mu_0 \geq \dots \geq \mu_n$ with $\sum \mu_i = 0$ (implying $\mu_0 >0, \mu_n <0$ if $\lambda$ is not trivial).

From Lemma \ref{dimension of picard group} with $m=1$, we can choose an ample $G$-linearisation $\mathcal{L} = \mathcal{O}_{\mathcal{C}}(1)$. Hence, the set of characters, with respect to this $G$-linearisation corresponds to degree $d$ polynomials $f$, as sections $s\in H^0(\mathcal{C}, \mathcal{O}_{\mathcal{C}}(1))$ are polynomials of degree $d$. In particular, for a monomial $x^I$ of degree $d$ and a normalised one-parameter subgroup $\lambda$ as above, the natural pairing is given by $\langle x^I, \lambda \rangle = \sum_{i=0}^n d_i\mu_i$. Note, that in many cases we will abuse the notation and write $\langle I, \lambda \rangle$ instead of $\langle x^I, \lambda \rangle$. Then, the $G$-action induced by $\lambda$ on a monomial $x^I$ is given by $\lambda(s)\cdot x^I = s^{\langle I, \lambda \rangle}x^I$. Naturally, the $G$-action induced by $\lambda$ on the polynomial $f$ is given by 
$$\lambda(s)\cdot f = \sum_{I\in \operatorname{Supp}(f)} s^{\langle I, \lambda \rangle}f_I x^I.$$
In addition, notice that the action of $\lambda$ on a fiber is equivalent to the action of $\lambda$ on the polynomial $f$, and by the above discussion, we have that $\operatorname{weight}(f,\lambda) = \min_{x^I\in \operatorname{Supp}(f)}\{\langle I, \lambda \rangle\}$. Thus the Hilbert-Mumford function reads
$$\mu(f,\lambda) = - \min_{x^I\in \operatorname{Supp}(f)}\{\langle I, \lambda \rangle\},$$
and the Hilbert-Mumford numerical criterion is:
\begin{lemma}\label{HM criterion for hypersurfaces}
With respect to a maximal torus $T$:
\begin{enumerate}
    \item $X = \{f=0\}$ is \emph{semi-stable} if and only if $\mu(f,\lambda) \geq 0$, i.e. if $\min_{x^I\in \operatorname{Supp}(f)}\{\langle I, \lambda \rangle\}\leq 0$, for all non-trivial one-parameter subgroups  $\lambda$ of $T$.
    \item $X = \{f=0\}$ is \emph{stable} if and only if $\mu(f, \lambda) > 0$, i.e. if $\min_{x^I\in \operatorname{Supp}(f)}\{\langle I, \lambda \rangle\}< 0$, for all non-trivial one-parameter subgroups  $\lambda$ of $T$.
\end{enumerate}
\end{lemma}

\begin{remark}\label{equivalence of min and max}
Fix a torus $T$, coordinates $(x_0\colon\dots\colon x_n)$ such that the torus is diagonal, and a normalised one-parameter subgroup $\lambda(s) = \operatorname{Diag} (s^{\mu_0}, \dots, s^{\mu_n})$. Let $f$ be a degree $d$ polynomial, and assume that for some $I = (d_0,\dots,d_n)$, for $x^I\in \operatorname{Supp}(f)$, $\langle I,\lambda \rangle$ is minimal, i.e.
$$\mu(f,\lambda) = - \langle I,\lambda \rangle = -\sum_{i=0}^nd_i\mu_i.$$
Notice that $\overline{\lambda(s)} = \operatorname{Diag} (s^{-\mu_n}, \dots, s^{-\mu_0})$ defines another normalised one-parameter subgroup, since \\ $\sum_{j=0}^n(-\mu_{n-j}) = -\sum_{i=0}^n\mu_i =0$, and $-\mu_{n-i}\geq -\mu_{n-i-1}$ by the fact that $\mu_{n-i-1}\geq \mu_{n-i}$. Let also $\overline{I} = ({d_n}, {d_{n-1}}, \dots, {d_0})$ be another monomial vector. Then, we have $\langle\overline{I}, \overline{\lambda} \rangle= -\langle I,\lambda \rangle$. We can think of obtaining this new vector $\overline{I}$ by making the change of coordinates $x_{n-i} \leftrightarrow x_i$; this is a projective change of coordinates, such that $f$ is projectively equivalent and isomorphic to $\overline{f}$, where $\overline{f} = \sum f_{\overline{I}}x^{\overline{I}}$. %Notice, that in order to be careful, we need to introduce a new torus, $\overline{T}$, which is diagonalisable under coordinates $(x_n\colon\dots\colon x_0)$, and is conjugate to $T$. 
We define $$\overline{\mu(\overline{f}, \overline{\lambda})}\coloneqq  \max_{x^{\overline{J}}\in \operatorname{Supp}(\overline{f})}\{\langle \overline{J}, \overline{\lambda} \rangle\}$$
and we will show that $\overline{\mu(\overline{f}, \overline{\lambda})} = \langle\overline{I}, \overline{\lambda} \rangle$, i.e. that $\langle\overline{I}, \overline{\lambda} \rangle$ is maximal for $x^{\overline{I}}\in \operatorname{Supp}(\overline{f})$. Suppose it is not. Then, there exists $\overline{I}' = ({d_n}', {d_{n-1}}', \dots, {d_0}')$, with $x^{\overline{I'}}\in \operatorname{Supp}(\overline{f})$, such that $\langle\overline{I}, \overline{\lambda} \rangle <\langle\overline{I}', \overline{\lambda} \rangle$. But, this in turn would imply that $\langle I',\lambda \rangle < \langle I,\lambda \rangle$, where $I' = (d_0',\dots,d_n')$, and $x^{I'} \in \operatorname{Supp}(f)$ since $f\cong \overline{f}$. This contradicts the original assumption that $\langle I,\lambda \rangle$ is minimal. Thus, we have shown that $$\mu(f,\lambda) = \overline{\mu(\overline{f}, \overline{\lambda})}.$$ 
The implication of the above is that we may reformulate Lemma \ref{HM criterion for hypersurfaces} as follows:
\end{remark}

\begin{lemma}\label{HM crit reformulated}
With respect to a maximal torus $T$:
\begin{enumerate}
    \item $X = \{f=0\}$ is \emph{semi-stable} if and only  if $\max_{x^I\in \operatorname{Supp}(f)}\{\langle I, \lambda \rangle\}\geq 0$, for all non-trivial one-parameter subgroups  $\lambda$ of $T$.
    \item $X = \{f=0\}$ is \emph{stable} if and only if $\max_{x^I\in \operatorname{Supp}(f)}\{\langle I, \lambda \rangle\}> 0$, for all non-trivial one-parameter subgroups  $\lambda$ of $T$.
\end{enumerate}

\end{lemma}

\begin{remark}
Throughout Sections \ref{VGIT_prelims}, \ref{section:stability conditions} and \ref{section: centroid crit}, as well as Sections \ref{p3 VGIT section} and \ref{p4 VGIT section} we will use the description of the Hilbert-Mumford numerical criterion in Lemma \ref{HM crit reformulated} (adapted for our VGIT problem of complete intersections and hyperplanes). Throughout these sections, to ease notation, we will denote the Hilbert-Mumford function by $\mu$ and not $\overline{\mu}$. %In Section \ref{thresholds section}, in order to aide the reader, we will use the convention of Lemma \ref{HM criterion for hypersurfaces} (adapted for our VGIT problem of complete intersections and hyperplanes), in keeping with the convention of the notation in \cite{zanardini2021stability}.

It is also noteworthy to point out that we would have obtained an identical Hilbert-Mumford function if we considered the conjugate action $A\cdot f(x) = f(A^{-1}x)$ for $A\in G$, as in \cite{fedorchuk, avritzer-lange}. This should not be a surprise, as the actions are conjugate to each other. We are choosing the natural $G$-action to be consistent with the moduli descriptions we are looking for.
\end{remark}

%Under the above conditions, all one-parameter subgroups are conjugate to a normalised one, so in most situations it suffices to consider only normalised one-parameter subgroups. %The Hilbert-Mumford numerical criterion for GIT- stability of $S$ can then be stated as if for all normalised or not one-parameter subgroups $\lambda$:
%$S$ is semi-stable (resp. stable) if and only if for all one-parameter subgroups $\lambda$, $\lim_{s \to 0} \lambda(s) \cdot [\overline{S}] \neq 0 $ (resp. $\lim_{s \to 0} \lambda(s) \cdot [\overline{S}]$ does not exist) \cite{mukai_2003}.

%It is useful to note that it suffices to take only normalised one-parameter subgroups as every one-parameter subgroup is conjugate to a normalised one.

%We let the embedded variety $S$ be defined by $k$ degree $d$ polynomials %$f_i$, where:
%$$f_i = \sum {(f_i)}_{I_i}x^{I_i}$$

%with $I_i = (d_{i,0}, \dots, d_{i,n})$, with $d_{i,j} \neq 0$, for some $j$, and $\sum_{j} d_{i,j} = d$. For simplicity we will denote by $\Xi_d$ the set of all monomials of degree $d$ in variables $x_0, \dots, x_n$. 

We will proceed to analyse the GIT quotient $\mathcal{R}_{n,d,k}\sslash G$. We will keep the notation as before, and we will fix a maximal torus $T$, with coordinates $(x_0\colon \dots\colon x_n)$, such that $T$ is diagonal. Recall that $I_i =(d_{i,0},\dots, d_{i,n})$ is a monomial vector such that $x^{I_i}\in \operatorname{Supp}(f_i)$. Let $\lambda$ be a normalised one-parameter subgroup. In this case, the natural pairing gives us $\langle I_i, \lambda \rangle = \sum_{j=0}^n d_{i,j}\mu_j$, and more specifically $\lambda(s) \cdot x^{I_i} = s^{\langle I_i, \lambda \rangle} x^{I_i}$. Similarly, the action of $\lambda$ on a Pl{\"u}cker coordinate $\bigwedge_{i=1}^k x^{I_i}$ is induced as:
$$\lambda(s) \cdot \bigwedge_{i=1}^k x^{I_i}= s^{\sum_{i=1}^k \langle I_i, \lambda \rangle} \bigwedge_{i=1}^k x^{I_i},$$
hence the $\lambda$-action on $[\overline{S}] \vcentcolon= \bigwedge_{i=1}^k f_{i} =  \bigwedge_{i=1}^k \sum f_{I_i}x^{I_i}$ is:

$$\lambda(s) \cdot [\overline{S}]=  \sum_{x^{I_i}\in \Xi_d} s^{\sum_{i=1}^k \langle I_i, \lambda \rangle} \bigwedge_{i=1}^k  f_{I_i}x^{I_i}$$
for $\lambda$ a normalised one-parameter subgroup.

%$$\mu(Q\wedge R, \lambda) \vcentcolon= \max \{\langle I, \lambda \rangle +\langle J, \lambda \rangle | q_I \neq 0 \quad \text{and}\quad r_J \neq 0\} .$$
%Given the above polynomials $Q, R$ we can define their associated sets of monomials $\operatorname{Supp}(Q) = \{x^I \in \Xi_2|q_I \neq 0 \}$ and $\operatorname{Supp}(R) = \{x^J \in \Xi_2|r_J \neq 0 \}$, and we can re-write the Hilbert- Mumford function for normalised $\lambda$ as
%$$\mu(Q\wedge R, \lambda) = \max \{\langle I, \lambda \rangle +\langle J, \lambda \rangle | x^I\in \operatorname{Supp}(Q) \quad \text{and}\quad x^J\in \operatorname{Supp}(R), I\neq J\} .$$
%Notice that the limit $\lim_{s \to 0} \lambda(s) \cdot [\overline{S}]$ is not zero if and only if $\mu(Q\wedge R, \lambda) \geq 0$ and does not exist if and only if $\mu(Q\wedge R, \lambda) > 0$.Hence, from the above discussion we have shown the lemma below.

%We can thus generalise the results of Section \ref{GIT} as follows:

From our discussion above, we have that 
%$$\operatorname{weight}(f_1\wedge\dots \wedge f_k,\lambda) = \min_{x^{I_i}\in \operatorname{Supp}(f_i)}\left\{\sum_{i=1}^k\langle I_i, \lambda \rangle\big| \text{ for all } i, I_i \neq I_j  \text{ for all } i,j \right\}.$$
%This implies that the Hilbert-Mumford function for a normalised one-parameter subgroup $\lambda$ is

%\begin{equation*}
%        \mu(f_1\wedge\dots \wedge f_k, \lambda)\vcentcolon= -\min_{x^{I_i}\in \operatorname{Supp}(f_i)} \left\{\sum_{i=1}^k\langle I_i, \lambda \rangle \text{ }|\ \text{ for all } i, I_i \neq I_j  \text{ for all } i,j\right\}.
%\end{equation*}

%Notice, that the discussion in Remark \ref{equivalence of min and max} still applies to this case, as for a minimum $\sum_{i=1}^k\langle I_i, \lambda \rangle$, with respect to $\lambda$, we have that $-\sum_{i=1}^k\langle I_i, \lambda \rangle $ is a maximum with respect to $\overline{\lambda}$. In addition, for the change of coordinates presented in Remark \ref{equivalence of min and max}, we have that $f_1\wedge \dots f_k \cong \overline{f}_1\wedge \dots \overline{f}_k$. Hence, abusing notation, we write the Hilbert-Mumford function for a normalised one-parameter subgroup $\lambda$ as:
\begin{equation*}
        \mu(f_1\wedge\dots \wedge f_k, \lambda)\vcentcolon= \max \left\{\sum_{i=1}^k\langle I_i,\lambda\rangle  \text{ }\Big| (I_1,\dots,I_k)\in (\Xi_d)^k, I_i \neq I_j  \text{ if } i\neq j \text{ and }x^{I_i}\in \operatorname{Supp}(f_i)\right\}.
\end{equation*}

%$$\mu(f_1\wedge\dots \wedge f_k, \lambda) \vcentcolon= \max \{\langle I_1, \lambda \rangle +\langle I_2, \lambda \rangle +\dots+\langle I_k, \lambda \rangle \text{ }|\text{ } (f_{i})_{I_i} \neq 0 \quad \\ \text{for all}\quad i, I_i \neq I_j \quad \text{for all}\quad i,j\} .$$
%Given the above polynomials $f_1, \dots, f_k$ we can %define their associated sets of monomials $\operatorname{Supp}({f}_i) = \{x^{I_i} \in \Xi_d|(f_i)_{I_i} \neq 0 \}$, and we can 
%re-write the Hilbert-Mumford function in terms of their associated sets of monomials: 
%$$\mu(f_1\wedge\dots \wedge f_k, \lambda) = \max \left\{\sum_{i=1}^k\langle I_i, \lambda \rangle\text{ } |\text{ } x^{I_i}\in \operatorname{Supp}({f}_i) \text{ for all } i, I_i \neq I_j \text{ for all } i,j\right\}.$$
%Notice that the limit $\lim_{s \to 0} \lambda(s) \cdot [\overline{S}]$ is not zero if and only if $\mu(Q\wedge R, \lambda) \geq 0$ and does not exist if and only if $\mu(Q\wedge R, \lambda) > 0$.Hence, from the above discussion we have shown the lemma below.

\begin{lemma}\label{mumford_map_GIT}
With respect to a maximal torus $T$:
\begin{enumerate}
%    \item $S$ is semi-stable if and only if $\mu(f_1\wedge\dots \wedge f_k, \lambda) \geq 0$ for all normalised one-parameter subgroups $\lambda$ of $T$; 
%    \item $S$ is stable if and only if $\mu(f_1\wedge\dots \wedge f_k, \lambda) > 0$ for all normalised one-parameter subgroups $\lambda$ of $T$;
    \item $[f_1\wedge\dots \wedge f_k]\in \operatorname{Gr}(k,W)$ is \emph{semi-stable} if and only if $\mu(f_1\wedge\dots \wedge f_k, \lambda) \geq 0$ for all non-trivial one-parameter subgroups  $\lambda$ of $T$. %(NOTE: MAYBE HERE CAN WORK FOR SEMISTABLE IF WE MENTION FOR ALL NON TRIVIAL 1-PS AND SAY IN THE PROOF WE ONLY NEED TO CONSIDER NORMAL ONES DUE TO THEIR PROPERTY)
    \item $[f_1\wedge\dots \wedge f_k]\in \operatorname{Gr}(k,W)$ is \emph{stable} if and only if $\mu(f_1\wedge\dots \wedge f_k, \lambda) > 0$ for all non-trivial one-parameter subgroups  $\lambda$ of $T$.
\end{enumerate}
\end{lemma}
\begin{proof}

From the above discussion, and \cite[Theorem 2.1]{mumford_fogarty_kirwan_1994}, we are left to show that\\ $\operatorname{sign}(\mu(f_1\wedge\dots \wedge f_k, \lambda)) = \operatorname{sign}(\mu([f_1\wedge\dots \wedge f_k], \lambda))$. A general element of $[f_1\wedge\dots \wedge f_k]$ is of the form $$ \alpha\bigwedge_{i=1}^k f_i,$$ 
and hence we have:
\begin{equation*}
    \begin{split}
        \operatorname{sign} (\mu([f_1\wedge\dots \wedge f_k], \lambda))& = \\
        \operatorname{sign}\Bigg(\max \bigg\{\sum_{i=1}^k\langle I_i, \lambda \rangle \text{ }\Big| (I_1,\dots,I_k)&\in (\Xi_d)^k, \text{ } I_i \neq I_j  \text{ if } i\neq j \text{, }x^{I_i}\in \operatorname{Supp}(f_i)\bigg\}\Bigg)\\
        & = \operatorname{sign}(\mu(f_1\wedge\dots \wedge f_k, \lambda)).
    \end{split}
\end{equation*}
%Finally, for $S, S'$ where $S\cong S'$ induced by an action of $\operatorname{Aut}(\mathbb{P}^n) = \operatorname{PGL}(n+1)$ and not $G$, we have by Koll{\'a}r \cite{kollar_survey} that $\mu(S,\lambda) = \mu(S',\lambda)$. 
\end{proof}

\begin{remark}
If the choice of the $f_i$ is clear, we will write $\mu(S,\lambda)$ instead of $\mu(f_1\wedge\dots \wedge f_k, \lambda)$
\end{remark}

\begin{remark}
The complete intersection $S$ induces an element of the Grassmanian \\$\Phi\vcentcolon= \big\{\sum \alpha_if_i|(a_i:\dots:a_k) \in \mathbb{P}^{k-1}\big\}$, which is a linear system of dimension $k-1$, whose base locus is $S$. As in \cite{mabuchi_mukai_1990}, when $\Phi$ is (semi-)stable we will also say that $S$ is (semi-)stable.

%When we refer to the (semi-)stability of $S$ the Hilbert-Mumford numerical criterion shows that both $S$ and $\Phi$ are (semi-)stable. From here on out, for simplicity we say that $S$ is (semi-)stable, as in Mabuchi--Mukai \cite{mabuchi_mukai_1990}.  
\end{remark}

Extending to the VGIT case, i.e to the quotient $\mathcal{R}_m\sslash G$, as in \cite{gallardo_martinez-garcia_2018, zhang_2018} we have the following Lemmas.

\begin{lemma}\label{dimension of picard group}
Let $G$ be an algebraic group such that $\operatorname{Pic}(G) = \{1\}$ and let $X$ be a normal projective $k$-scheme such that $X = X_1\times \dots X_m$ with $\dim(\operatorname{Pic})(X_i) = 1$, such that the action of $G$ on $X$ extends to an action of $G$ on each $X_i$. Let also every line bundle of $X$ have at most one linearisation class. Then the set of $G$-linearisable line bundles $\operatorname{Pic}^G(X)$ is isomorphic to $\mathbb{Z}^m$. A line bundle $\mathcal{L}\in\operatorname{Pic}^G(X)$ is ample if and only if 
$$\mathcal{L} = \mathcal{O}(a_1,\dots, a_{m})\vcentcolon= \pi_1^* \mathcal{O}_{X_1}(a_1)\otimes  \pi_m^*\mathcal{O}_{X_m}(a_m),$$
where the $p_i$  are the natural projections on $X_i$,  and $a_i > 0$, and we denote by $\mathcal{O}_{X_i}(a_i)$ an ample generator (over $\mathbb{Z}$) of the Picard group.
\end{lemma}
\begin{proof}
The proof follows \cite[Section 2.1]{zhang_2018} which comes as a generalisation of \cite[Lemma 2.1]{gallardo_martinez-garcia_2018}.

Let $p_i$ be the projections. Since the action of $G$ on $X$ extends to an action of $G$ on each $X_i$, the $p_i$ are morphisms of $G$-varieties. Recall there is an exact sequence (see \cite[Theorem 7.2]{dolgachev_2003})
\begin{center}
    \begin{tikzcd}
    0 \arrow[r] &\mathcal{X}(G) \arrow[r]& \operatorname{Pic}^G(X)\arrow[r]&\operatorname{Pic}(X) \arrow[r]&\operatorname{Pic}(G) 
    \end{tikzcd}
\end{center}
where $\mathcal{X}(G)$ is the kernel of the forgetful morphism $\operatorname{Pic}^G(X) \rightarrow \operatorname{Pic}(X)$. By assumption $\mathcal{X}(G)=\{1\}$, hence we have an isomorphism $\operatorname{Pic}^G(X)\cong \operatorname{Pic}(X)$. Moreover, given that $\operatorname{Pic}^G(X) \subseteq \operatorname{Pic}(X)^G\subseteq \operatorname{Pic}(X)$, where $\operatorname{Pic}(X)^G$ is the group of $G$-invariant line bundles, there result follows from 
$$\operatorname{Pic}^G(X)\cong\operatorname{Pic}(X)^G\cong p_1^*(\operatorname{Pic}(X_1))\otimes\dots\otimes p_m^*(\operatorname{Pic}(X_m))\cong \mathbb{Z}^m.$$
\end{proof}

\begin{remark}\label{remark on why pic of grp g is trivial}
The condition that the algebraic group $G$ has a trivial Picard group, i.e. $\operatorname{Pic}(G) = \{1\}$, in the above Lemma may seem restrictive, but for us, where $G = \operatorname{SL}(n+1)$, this condition will always be satisfied (see \cite[Chapter 7.2]{dolgachev_2003}).
\end{remark}

We are now in a position to explicitly define the Hilbert-Mumford function for VGIT.

\begin{lemma}\label{picard rank 2}
 The set of $G$-linearisable line bundles $\operatorname{Pic}^G(\mathcal{R}_m)$ is isomorphic to $\mathbb{Z}^{m+1}$. A line bundle $\mathcal{L} \in \operatorname{Pic}^G(\mathcal{R}_m)$, is ample if and only if
$$\mathcal{L} = \mathcal{O}(a,b_1,\dots, b_{m})\vcentcolon= \mathcal{O}(a,\vec{b})\vcentcolon= \pi_1^* \mathcal{O}_{\operatorname{Gr}}(a)\otimes \pi_{2_1}^*\mathcal{O}_{\mathbb{P}^n}(b) \otimes\dots\otimes \pi_{2_m}^*\mathcal{O}_{\mathbb{P}^n}(b_m)$$ 
where $\pi_1$ and $\pi_{2_i}$ are the natural projections and $ a, b_i > 0$.
\end{lemma}
\begin{proof}
The proof follows the argument of the proof of \cite[Lemma 3.2]{gallardo_martinez-garcia_2018}, and Lemma \ref{dimension of picard group}, noting that all the assumptions of Lemma \ref{dimension of picard group} are satisfied.
%Let $\pi_1\colon \mathcal{R} \rightarrow \mathbb{P}\bigwedge^k W$ , $pi_2\colon \mathcal{R} \rightarrow \mathbb{P}^n$ be the natural projections. The action of $G$ on $\Xi_d$ and
%$\Xi_1$ induces a natural action on $\mathcal{R}$ which preserves the fibers. Hence $G$ acts on both components of $\mathcal{R}$ and the $\pi_i$ are morphisms of $G$-varieties. Recall there is an exact sequence (ref Dolgachev):
%$$0 \rightarrow X(G) \rightarrow \operatorname{Pic}^G(\mathcal{R} \rightarrow \operatorname{Pic}(\mathcal{R}) \rightarrow \operatorname{G}$$
%where $X(G)$ is the kernel of the forgetful morphism $\operatorname{Pic}^G(\mathcal{R}) \rightarrow \operatorname{Pic}(\mathcal{R})$. Since $X (G) = {1}$ and $\operatorname{Pic}(G) \{1\}$ by Dolgachev then $\operatorname{Pic}^G(\mathcal{R}) \simeq \operatorname{Pic}(\mathcal{R})$. Moreover, given that $\operatorname{Pic}^G(\mathcal{R}) \subseteq \operatorname{Pic}(\mathcal{R})^G \subset \operatorname{Pic}(\mathcal{R})$, were $\operatorname{Pic}(\mathcal{R})^G$ is the group of $G$-invariant line bundles, there result follows from
%$$\operatorname{Pic}^G(\mathcal{R})\simeq \operatorname{Pic}(\mathcal{R})^G \simeq \operatorname{Pic}(\mathcal{R}) \simeq \pi_1^*(\operatorname{Pic}(\operatorname{Gr}))\otimes \pi_2^*(\operatorname{Pic}(\mathbb{P}^n)) \simeq \mathbb{Z}\times \mathbb{Z},$$
%as $\operatorname{Pic}(\operatorname{Gr}) = \mathbb{Z}.$
\end{proof}

\begin{lemma}\label{decomposition of H-M function}
Let $\mathcal{L}\cong \mathcal{O}(a_1,\dots,a_m)\in \operatorname{Pic}^G(X)$ be ample. Then 
$$\mu^{\mathcal{L}}(X, \lambda)= \sum_{i=1}^ma_i\mu(X_i, \lambda).$$
\end{lemma}
\begin{proof}
By the properties of the Hilbert-Mumford function \cite[Definition 2.2]{mumford_fogarty_kirwan_1994}, we have that $\mu^{\mathcal{L}}(-,\lambda)\colon M\rightarrow \mathbb{Z}$ is a group homomorphism, and that given any $G$-equivariant morphism of $G$-varieties $f\colon X \rightarrow Y$ and ample $\mathcal{M}\in \operatorname{Pic}^G(Y)$ , we have that $\mu ^{f^*\mathcal{M}}(X,\lambda) =\mu ^{\mathcal{M}}(f(X),\lambda)$. Hence, we have:
\begin{equation*}
    \begin{split}
      \mu^{\mathcal{O}(a_1,\dots,a_m)}(X,\lambda)&=\mu^{\pi_1^* \mathcal{O}_{X_1}(a_1)\otimes  \pi_m^*\mathcal{O}_{X_m}(a_m)}(X,\lambda)\\
      &= \sum_{i=1}^m \mu^{\pi_i^* \mathcal{O}_{X_i}(a_i)}(X_i,\lambda)\\
      &=\sum_{i=1}^ma_i\mu(X_i, \lambda).
    \end{split}
\end{equation*}
\end{proof}

For ample $\mathcal{L} = \mathcal{O}(a,\vec{b})$ the GIT quotient $\mathcal{R}\sslash G$ is 

$$\overline{M}^{GIT}_{n,d,k,m,\vec{t}}\vcentcolon= \overline{M}^{GIT}(\vec{t})_{n,d,k,m}\vcentcolon=\operatorname{Proj}\bigoplus_j H^0(\mathcal{R}, \mathcal{L}^{\otimes j})^G,$$
where $\vec{t} = (t_1,\dots,t_m)$ with $t_i= \frac{b_i}{a} \in \mathbb{Q}_{>0}$. By Lemma \ref{decomposition of H-M function}, for $\mathcal{L}$ as above, and by the functoriality of the Hilbert-Mumford function, we have $\mu^{\mathcal{L}}((S,H_1,\dots, H_m), \lambda) = a \mu_{\vec{t}}(S,H_1,\dots, H_m,\lambda)$  where 
\begin{equation*}
    \begin{split}
        \mu_{\vec{t}}(S, H_1,\dots, H_m, \lambda) &\vcentcolon= \mu_{\vec{t}}(f_1\wedge\dots \wedge f_k, h_1,\dots, h_m, \lambda) \\
        &\coloneqq  \mu(f_1\wedge\dots \wedge f_k, \lambda) + \sum_{p=1}^mt_p\mu(H_p, \lambda)      \\
        &= \max \left\{\sum_{i=1}^k\langle I_i, \lambda\rangle \Big| (I_1,\dots,I_k)\in (\Xi_d)^k \text{, } I_i \neq I_j  \text{ if } i\neq j \text{, }x^{I_i}\in \operatorname{Supp}(f_i) \right\}\\
        &+ \sum_{p=1}^mt_p\max\{r_j| h_{j,p}\neq0, h_{j,p} \in \operatorname{Supp}(h_p)\}.
    \end{split}
\end{equation*}

\begin{definition}\label{HM for VGIT}
Let $\vec{t}$ be such that all $t_i\in \mathbb{Q}_{>0}$. The tuple $(f_1,\dots, f_k, h_1,\dots, h_m)$ is \emph{$\vec{t}$-stable} (resp. \emph{$\vec{t}$-semistable}) if $\mu_{\vec{t}}( f_1\wedge\dots\wedge f_k, h_1,\dots, h_m, \lambda) > 0$
(respectively, $\mu_{\vec{t}}(f_1\wedge\dots\wedge f_k, h_1,\dots, h_m, \lambda)  \geq 0$) for all non-trivial normalised one-parameter subgroups $\lambda$ of $G$. A tuple $(f_1,\dots, f_k, h_1,\dots, h_m)$ is
\emph{$\vec{t}$-unstable} if it is not $\vec{t}$-semistable. A tuple $(f_1,\dots, f_k, h_1,\dots, h_m)$ is \emph{strictly $\vec{t}$-semistable} if it is $\vec{t}$-semistable but not $\vec{t}$-stable.
\end{definition}
\begin{remark}
We will often write $(S, H_1,\dots, H_m)$ instead of
the tuple $(f_1,\dots, f_k, h_1,\dots, h_m)$ for ease of notation.
\end{remark}

\subsection{Stability conditions}\label{section:stability conditions}
As before, we fix a maximal torus $T \subset G$ and coordinates such that $T$ is diagonal. We define a partial order on $\Xi_d$ (following Mukai \cite[\S 7]{mukai_2003}): given $v, v' \in \Xi_d $ we have
$$v \leq v' \quad \text{if and only if} \quad \langle v, \lambda  \rangle \leq \langle v', \lambda  \rangle$$
for all normalised one-parameter subgroups $\lambda \colon \mathbb{G}_m \rightarrow T \subset G$. It is also useful to note there exists a lexicographic order $<_{lex}$, e.g. $x_0x_1 <_{lex} x_0^2$ in $\Xi_d$, which is a total order \cite[Theorem 1.8]{harzheim}, i.e. it is  antisymmetric, transitive, and equipped with a connex relation, where if $x \neq y$ we have either $x<_{lex}y$ or $y<_{lex}x$. This gives rise to the order $\leq_{\lambda}$ for a normalised one-parameter subgroup $\lambda$ which we call the $\lambda$-order: $v \leq_{\lambda} v'$ if and only if:

\begin{enumerate}
    \item $\langle v, \lambda  \rangle < \langle v', \lambda  \rangle$ or
    \item $\langle v, \lambda  \rangle = \langle v', \lambda  \rangle$ and $v \leq_{lex} v'$.
\end{enumerate}

\begin{lemma}
The $\lambda$-order is a total order .
\end{lemma}
\begin{proof}
This follows, since the lexicographic order is a total order, \cite[Theorem 1.8]{harzheim}.%, hence 
%\begin{enumerate}
%    \item it is reflexive since $v\leq_{\lambda} v$, as $v\leq_{lex} v$
%    \item it is transitive as 
%\end{enumerate}
\end{proof}

As in Gallardo--Martinez- Garcia \cite[Section 5]{gallardo_martinez-garcia_2018} we define a fundamental set of one-parameter subgroups $\lambda \in T$.

\begin{definition}\label{finite_set_def_k_case}
The \emph{fundamental set of one-parameter subgroups} $\lambda \in T$ for $k$ hypersurfaces of degree $d$ and $m$ hyperplanes, $P_{n,k,d,m}$  consists of all non-trivial elements $\lambda(s) = \operatorname{Diag} (s^{\mu_0}, \dots, s^{\mu_n})$ where 
$$({\mu_0}, \dots, {\mu_n}) = c({\gamma_0}, \dots, {\gamma_n}) \in \mathbb{Z}^{n+1}$$
satisfying the following:

\begin{enumerate}
    \item $\gamma_i = \frac{a_i}{b_i}\in \mathbb{Q}$ such that $\gcd(a_i, b_i) = 1$ for all $i = 0, \dots n$ and $c = \operatorname{lcm}(b_0, \dots, b_n)$,
    \item $1=\gamma_0 \geq \gamma_1 \geq \dots \geq \gamma_n = -1 - \sum^{n-1}_{i=1} \gamma_i$,
    \item $(\gamma_0, \dots, \gamma_n)$ is the unique solution of a consistent linear system given by $n-1$ equations chosen from the union of the following sets:
    $$\operatorname{Eq}(n-1,k,d) \vcentcolon= \{\gamma_i - \gamma_{i+1} = 0\} \cup$$ 
    $$ \bigg\{ \sum^{n}_{j=0}\big[\sum_{i=1}^k d_{i,j} -\sum_{i=1}^k \overline{d}_{i,j}\big]\gamma_j = 0\text{ }|\text{ } d_{i,j},\text{ }\overline{d}_{i,j}\in \mathbb{Z}_{\geq 0}, \forall i,j, \sum^{n}_{j=0} d_{i,j} = \sum^{n}_{j=0} \overline{d}_{i,j}= d \bigg\}.$$
\end{enumerate}
Since there is a finite number of monomials in $\Xi_d$ and equations in $\operatorname{Eq}(n-1,k,d)$, the set $P_{n,k,d,m}$ is finite.
\end{definition}

\begin{lemma}\label{unstablelemma-vgit}
A tuple $(S, H_1,\dots, H_m)$ is not $\vec{t}$-stable (respectively not $\vec{t}$-semistable) 
%A pair $(S,H)$ is not $t-$stable (respectively not $t$-semistable) 
if and only if there is $g \in G$ such that
$$\mu_{\vec{t}}(S, H_1,\dots,H_m) \coloneqq \max_{\lambda \in P_{n,d,k,m}}\{\mu_t(g\cdot S, g\cdot H_1, \dots, g\cdot H_m, \lambda)\}\leq 0 \quad (\text{respectively}<0)$$
%$$\mu_{t}(S, H) \vcentcolon= \max_{\lambda \in P_{n,d,k}}\{\mu_t(g\cdot S, g\cdot H, \lambda)\}\geq 0 \quad (\text{respectively}>0)$$
where $\mu_{\vec{t}}$ is the Hilbert-Mumford function defined above and $P_{n,d,k,m}$ is the fundamental set of Definition \ref{finite_set_def_k_case}.
\end{lemma}
\begin{proof}
The proof follows the argument in the proof of \cite[Lemma 3.2]{gallardo_martinez-garcia_2018}.
Let $R^{n}_{m,T}$ be the non-stable loci of $\mathcal{R}_m$ with respect to a maximal torus $T<G$; and let $\mathcal{R}^n_m$ be
the non-stable loci of $\mathcal{R}_{m}$. By Mumford \cite[p.137]{mumford_fogarty_kirwan_1994}, the following holds 
$$\mathcal{R}^n_{m} = \bigcup_{T_i \subset G, \text{ maximal}} R^{n}_{m, T_i}.$$
Let $(f_1',\dots, f_k',h_1,\dots, h_m)$ be a tuple which is not $\vec{t}$-stable. Then, $\mu_{\vec{t}}(f_1' \wedge \dots \wedge f_k',h'_1,\dots, h'_m, \rho) \\\leq 0$ for some $\rho \in T'$ in a maximal torus $T'$. All the maximal tori are conjugate to each other in $G$, and by  \cite[Definition 2.2]{mumford_fogarty_kirwan_1994} the following holds:
$$\mu_{\vec{t}}(f_1' \wedge \dots \wedge f_k', h'_1,\dots, h'_m, \rho)  = \mu_{\vec{t}}(g \cdot (f_1' \wedge \dots \wedge f_k', h'_1,\dots, h'_m), g\rho g^{-1}).$$

Then, there is $g_0 \in  G$ such that $\lambda \vcentcolon= g_0 \rho {g_0}^{-1}$ is normalised with respect to a torus $T$ whose generators define the variables for the monomials of $f_i$ and $h_j$ and $(f_1,\dots, f_k, h_1,\dots, h_m)\\\vcentcolon= g_0 \cdot (f_1',\dots,  f_k',h'_1,\dots, h'_m)$ has coordinates in the coordinate system which diagonalises $\lambda$ such that $\mu_{\vec{t}}(f_1 \wedge \dots \wedge f_k,h_1,\dots,h_m, \lambda) \leq 0$. In this coordinate system normalised one-parameter subgroups \\$\lambda = \operatorname{Diag}(s^{\mu_0}, \dots, s^{\mu_n})$, with fixed $\mu_0 >0$ form a closed convex polyhedral subset $\Delta$ of dimension $n$ in $N \otimes \mathbb{Q} \cong \mathbb{Q}^{n+1}$ (in fact $\Delta$ is a standard simplex). Indeed, this is the case since for any normalised one-parameter subgroup, $\lambda(s) = \operatorname{Diag} (s^{\mu_0}, \dots, s^{\mu_n})$ where $\mu_0 \geq \dots \geq \mu_n$ with $\sum \mu_i = 0$ and we may assume without loss of generality that $\mu_0 = 1$.

For any fixed $f_1,\dots, f_k,h_1,\dots,h_m$, the function $\mu_{\vec{t}} (f_1 \wedge \dots \wedge f_k,h_1,\dots,h_m, -)\colon N \otimes\\ \mathbb{Q} \rightarrow \mathbb{Q} $ is piecewise linear, as for a fixed maximal torus $T$ and normalised one-parameter subgroup $\lambda$ the function is the maximum of a finite number of linear forms. The critical points of $\mu_{\vec{t}}$ (i.e. the points where $\mu_{\vec{t}}$ fails to be linear) correspond to those points in $N \otimes \mathbb{Q}$ where $\sum_{i=1}^k \langle I_i, \lambda \rangle +\sum_{p=1}^mt_p \langle x_{l_p},\lambda \rangle   = \sum_{i=1}^k\langle \overline{I}_i, \lambda \rangle +\sum_{p=1}^mt_p \langle \overline{x}_{l_p},\lambda \rangle$, for  $I_i = (d_{i,0}, \dots, d_{i,n})$, $\overline{I}_i = (\overline{d}_{i,0}, \dots, \overline{d}_{i,n})$ representing monomials of degree $d$, with $I_i\neq I_j$ for all $i\neq j$, %of the form $f_1\wedge \dots \wedge f_k = \bigwedge_{i=1}^k\sum \prod{(f_i)}_{I_i}   x^{I_i}$ defining $f_1 \wedge \dots \wedge f_k.$
where $f_i = \sum f_{I_i}x^{I_i}$, and $x_{l_p}$ are  monomials of degree $1$, such that $x_{l_p}\coloneqq \max \operatorname{Supp}(h_p)$. Notice that $x_{l_p}$
is unique for each of the $H_p = \{h_p=0\}$, hence $\sum_{p=1}^mt_p \langle x_{l_p},\lambda \rangle = \sum_{p=1}^mt_p \langle \overline{x}_{l_p},\lambda \rangle$ implies that $x_{l_p} = \overline{x}_{l_p}$ for each $p$. Hence the $\sum_{i=1}^mt_ix_{j_i}$ component is always linear, and as such the critical points of $\mu_{\vec{t}}$ correspond to those points in $N \otimes \mathbb{Q}$ where $\sum_{i=1}^k \langle I_i, \lambda \rangle  = \sum_{i=1}^k\langle \overline{I}_i, \lambda \rangle $. Since $\langle-,-\rangle$ is bilinear, that is equivalent to say that $\big\langle \sum_{i=1}^{k}\big(I_i -   \overline{I}_i\big), \lambda \big\rangle = 0$. These points define a hyperplane in $N \otimes \mathbb{Q}$ and the intersection of this hyperplane with $\Delta$ is a simplex $\Delta_{(I_i,\overline{I}_i)}$ of dimension $n-1$.

The function $\mu_{\vec{t}} (f_1 \wedge \dots \wedge f_k,h_1,\dots,h_m, -)$ is linear on the complement of the union of hyperplanes defined by $\big\langle \sum_{i=1}^{k}\big(I_i -   \overline{I}_i\big), \lambda \big\rangle = 0$. Hence, its maximum is achieved on the boundary, i.e. either on $\partial \Delta$ or on $\Delta_{(I_i,\overline{I}_i)}$ which are all convex polytopes of dimension $n-1$. We can repeat this reasoning by finite inverse induction on the dimension until we conclude that the maximum of $\mu_{\vec{t}} (f_1 \wedge \dots \wedge f_k,h_1,\dots,h_m, -)$ is achieved at one of the
vertices of $\Delta$ or $\Delta_{(I_i,\overline{I}_i)}$. Notice that these correspond precisely, up to multiplication by a constant, to the finite set of one-parameter subgroups in $P_{n,k,d,m}$.
\end{proof}

%The following result applies to the case of pairs $(S,H)$.

\begin{corollary}\label{intervals_of_stab}
Let $(S, H) \in \mathcal{R} $ and
\begin{equation*}
    \begin{split}
        a &= \min\{t \in \mathbb{Q}_{>0} |\mu_t(g\cdot S, g \cdot H, \lambda)\geq0, \text{ } \forall \lambda \in  P_{n,d,k,1}, g \in G\},\\
        b &= \max\{t \in \mathbb{Q}_{>0} |\mu_t(g\cdot S, g \cdot H, \lambda)\geq 0, \text{ }  \forall \lambda \in  P_{n,d,k,1}, g \in G\}.
    \end{split}
\end{equation*}
If $(S, H)$ is $t$-semistable for some $t$, then
\begin{enumerate}
    \item $(S, H)$ is $t$-semistable if and only if $t \in [a, b] \cap \mathbb{Q}_{>0}$,
    \item if $(S, H)$ is $t$-stable for some $t\in (a,b)$, then $(S, H)$ is $t$-stable for all $t \in(a, b) \cap \mathbb{Q}_{>0}$.
\end{enumerate}
The interval $[a,b]$ is called the interval of stability of the pair. If $[a, b] = \emptyset$ then the pair is $t$-unstable for all $t$.

\end{corollary}
\begin{proof}
The proof follows the argument of the proof of \cite[Corollary 3.3]{gallardo_martinez-garcia_2018}.
%TODO add the proof here as well

%For tuples $(S,H_1,\dots,H_m)$ we define the \emph{space of stability conditions} $\operatorname{Stab}(n,d,k,m)\vcentcolon= \{\vec{t}\in (\mathbb{Q}_{\geq 0})^m|\text{ there exists a } \vec{t}-\text{semistable tuple } (S,H_1,\dots, H_m)\}$.
\end{proof}

We can also extend the above for the case of tuples.

\begin{definition}[see {\cite[Definition 2.4 ]{zhang_2018}}]\label{stab_conditions}
The \emph{space of stability conditions} is 
$$\operatorname{Stab}(n,d,k,m) \vcentcolon=\left\{\vec{t}\in (\mathbb{Q}_{\geq 0})^m|\text{ there exists a } \vec{t}\text{-semistable tuple }(S,H_1,\dots,H_m)\right\}. $$
\end{definition}

Lemma \ref{unstablelemma-vgit} determines the $\vec{t}$-stability for all tuples $(S,H_1,\dots,H_m)$ as the fundamental set does not depend on $t$. In more detail, in checking the stability conditions for such tuples, Lemma \ref{unstablelemma-vgit} shows that we may only consider a finite set of one-parameter subgroups $P_{n,d,k,m}$, to check which tuples are not $\vec{t}$-semistable, In particular, this Lemma allows us to determine the space of stability conditions, which is compact by \cite{dolgachev_1994}. This verifies the prediction in \cite{kraft} that there should be a finite set of one-parameter subgroups that determines the stability of tuples.

\subsection{The Centroid criterion}\label{section: centroid crit}
We will state a \emph{centroid criterion}, which will allow us to computationally deduce whether a point is stable or strictly semistable. This is known to be always possible by \cite[Theorem 9.2]{dolgachev_1994}, and is a well-known direct application of the Hilbert-Mumford numerical criterion. The contribution of this paper in this application, is the very explicit description of the proof in the notation we have introduced for Section \ref{VGIT_prelims}.

We define a map $A: (\Xi_d)^k\times (\Xi_1)^m \rightarrow \mathbb{Q}^{n+1}$ as follows:
Given distinct monomials $x^{I_1},\dots, x^{I_k} \in \Xi_d$, %i.e. monomials of the form $x^{I_1} = x_0^{d_{1,0}} \dots x_n^{d_{1,n}} $ and in general $x^{I_i} = x_0^{d_{i,0}} \dots x_n^{d_{i,n}}$ with $\sum d_{1,i} =\dots =\sum d_{k,i} = d$, 
and distinct monomials $x_{l_1},\dots, x_{l_m}\in \Xi_1$, let 
$$A\Bigg((x^{I_1}, \dots, x^{I_k}), x_{l_1},\dots, x_{l_m}\bigg) = \Big(\sum_{i=0}^k d_{i,0}, \dots, \sum_{i=0}^k d_{i,l_1}+ t_{1},\dots, \sum_{i=0}^k d_{i,l_m}+ t_{m},\dots, \sum_{i=0}^k d_{i,n}\Bigg)$$
where we add the $t_{p}$ at the position of monomial $x_{l_p}$, i.e. at position $l_p$, e.g. if we have monomials $x_{l_1} = x_0$, $x_{l_2} = x_3$ and $m=2$ we have 

$$A\bigg((x^{I_1}, \dots, x^{I_k}), x_{0}, x_{3}\bigg) = \Bigg(\sum_{i=0}^k d_{i,0}+t_1, \sum_{i=0}^k d_{i,1}, \sum_{i=0}^k d_{i,2}, \sum_{i=0}^k d_{i,3}+ t_{2}, \sum_{i=0}^k d_{i,4}, \dots, \sum_{i=0}^k d_{i,n}\Bigg).$$

Since:

$$\sum_{j=0}^{n}\big(\sum_{i=1}^{k}(d_{i,j})\big)+ \sum_{p=1}^mt_p= kd +\sum_{p=1}^mt_p$$
the image of $A$ is contained on the first quadrant of the hyperplane 

$$H_{n,k,d,m,\vec{t}} \vcentcolon= \bigg\{(z_0, \dots, z_n) \in \mathbb{Q}^{n+1}\Big| \sum_{i=0}^n z_i = kd+\sum_{p=1}^mt_p\bigg\}. $$

Then, for $S$, a complete intersection of $k$ hypersurfaces of degree $d$, $S=\{f_1=\dots= f_k=0\}$, as before and $H_1,\dots, H_m$ distinct hyperplanes embedded in $\mathbb{P}^n$, where $H_i = \{h_i =0\}$, we define their convex hull $\operatorname{Conv}(S,H_1,\dots,H_m)$ as the convex hull of 
$$\left\{A((x^{I_1},\dots, x^{I_k}), x_{l_1},\dots,x_{l_m})\big|\text{ } x^{I_i} \in \operatorname{Supp}(f_i) \text{ for all } i,\text{ }I_i\neq I_j, x_{l_p} =\max(\operatorname{Supp}(h_p))\right\},$$
where $\max$ is given with respect to the $\lambda$ order and is unique since the $\lambda$-order is a total order. Then $\overline{\operatorname{Conv}(S,H_1,\dots,H_m)}$ is a convex polytope in $H_{n,k,d,m,\vec{t}}$. Note that in particular, for 
$$A\big((f_1,\dots, f_k),h_1,\dots,h_m\big) \coloneqq A\Bigg(\big(\sum f_{I_1}x^{I_1},\dots, \sum f_{I_k}x^{I_k}\big), x_{l_1},\dots,x_{l_m}\Bigg)$$
we have 
$$A\big((f_1,\dots, f_k),h_1,\dots,h_m\big) \in \overline{\operatorname{Conv}(S,H_1,\dots,H_m)}.$$
We also define the $\vec{t}$-\textit{centroid} of $S,H_1,\dots,H_m$ as 

$$\mathcal{C}_{n,d,k,m,\vec{t}} = \Bigg(\frac{kd+\sum_{p=1}^mt_p}{n+1}, \dots, \frac{kd+\sum_{p=1}^mt_p}{n+1}\Bigg) \in H_{n,k,d,m,\vec{t} } \subset \mathbb{Q}^{n+1}.$$

\begin{theorem}\label{t-centroid_criterion}
A tuple $(S,H_1,\dots,H_m)$ is $\vec{t}$-semistable (respectively, $\vec{t}$-stable) if and only if $\mathcal{C}_{n,d,k,m,\vec{t}} \in \overline{\operatorname{Conv}(S,H_1,\dots,H_m)}$ (respectively, $\mathcal{C}_{n,d,k,m,\vec{t}} \in \operatorname{Int}(S,H_1,\dots, H_m)$, where $\operatorname{Int}(S,H_1,\dots, H_m)$ is the interior of $\overline{\operatorname{Conv}(S,H_1,\dots,H_m)}$).
\end{theorem}

\begin{proof}
The proof follows the argument of the proof of \cite[Lemma 1.5]{gallardo_martinez-garcia_2018}. Let $x_{l_p} = \max\operatorname{Supp}(h_p)$, for $p=1,\dots,m$. First, note that since
\begin{equation*}
    \begin{split}
       \mu_{\vec{t}}(S,H_1,\dots,H_m,\lambda)& = \mu(S,\lambda)+\sum_{p=1}^mt_p\max_{x_{z} \in \operatorname{Supp}(H_p)}\langle x_z,\lambda\rangle \\
       &= \mu(S,\lambda)+\sum_{p=1}^mt_p\langle x_{l_p},\lambda\rangle \\
       &= \mu_{\vec{t}}(S,S\cap\{x_{l_1} = 0\}\cap\dots\cap \{x_{l_m} = 0\},\lambda).
    \end{split}
\end{equation*}
%$\mu_{\vec{t}}(S,H_1,\dots,H_m,\lambda) = \mu(S,\lambda)+\sum_{p=1}^mt_p\max_{x_{z} \in \operatorname{Supp}(H_p)}\langle x_z,\lambda\rangle = \mu(S,\lambda)+\sum_{p=1}^mt_p\langle x_{l_p},\lambda\rangle = \mu_{\vec{t}}(S,S\cap\{x_{l_1} = 0\}\cap\dots\cap \{x_{l_m} = 0\},\lambda)$ we take $D = S\cap\{x_{l_1}=0\}\cap\dots\cap \{x_{l_m} = 0\}$.

We first show that if $(S,H_1,\dots,H_m)=((f_1,\dots, f_k),H_1,\dots,H_m)$ is $\vec{t}$-semistable then the condition holds. Suppose $\mathcal{C}_{n,d,k,\vec{t}} \notin \overline{\operatorname{Conv}(S,H_1,\dots,H_m)}$, then, there exists an affine map $\psi: \mathbb{R}^{n+1} \rightarrow \mathbb{R}$ with $\psi(\mathcal{C}_{n,d,k,\vec{t}}) = 0$ and $\psi|_{\overline{\operatorname{Conv}(S,H_1,\dots,H_m)}} >0$. We can write 

$$\psi(z_0, \dots, z_n) = \sum_{i=0}^n \alpha_j z_j + q$$
where $\alpha_i$ are integers, since the convex hull has vertices with rational coefficients. Thus, for $x^{I_i}\in \operatorname{Supp}(f_i)$:

$$\psi(A(x^{I_1}, \dots, x^{I_k}, x_{l_1}, \dots, x_{l_m})) = \sum_{j=0}^n \alpha_j \Big(\sum_{i=1}^{k}d_{i,j}\Big) +\sum_{p=1}^mt_p\alpha_{l_p}+q >0. $$
$$0 = \psi(\mathcal{C}_{n,d,k,\vec{t}}) = \frac{kd+\sum_{p=1}^m t_p}{n+1}\sum_{i=0}^{n}\alpha_i +q.$$
Let $\delta \vcentcolon= -\frac{q}{kd+\sum_{p=1}^mt_p} \in \mathbb{Q}$. Then $(n+1)\delta = \sum_{i=0}^n \alpha_i$, and we can choose some $r \in \mathbb{Z}_{\leq0}$ such that $r\delta \in \mathbb{Z}$. Then we can define one-parameter subgroup: 
$$ \lambda(s) \vcentcolon= \operatorname{Diag}(s^{r(\alpha_0 - \delta)}, \dots, s^{r(\alpha_n - \delta)}).$$
Using this one-parameter subgroup we have:
\begin{equation*}
    \begin{split}
        \mu_{\vec{t}}(S,H_1,\dots,H_m, \lambda)         &=\max_{x^{I_i}\in \operatorname{Supp}(f_{i})}\left\{\sum_{i=1}^k\langle I_i, \lambda \rangle|\text{ for all } i, I_i \neq I_j  \text{ for all } i,j \right\} + \sum_{p=1}^mt_p\langle x_{l_p},\lambda \rangle\\
        &= \max_{x^{I_i}\in \operatorname{Supp}(f_{i})} \Bigg\{ \sum_{j=0}^{n}r(\alpha_j-\delta)\sum_{i=1}^{k}d_{i,j}\Bigg\} +\sum_{p=1}^m t_pr(\alpha_{l_p}-\delta) \\
       &=r\Bigg(\max_{x^{I_i}\in \operatorname{Supp}(f_i)}\Bigg\{\sum_{j=0}^{n}\alpha_j\sum_{i=1}^{k}d_{i,j} +  \sum_{p=1}^mt_p\alpha_{l_p}\Bigg\}-(kd+\sum_{p=1}^mt_p)\delta\Bigg)\\
        &= r\max_{x^{I_i}\in \operatorname{Supp}(f_i)}\Big\{\psi (A((x^{I_1}, \dots, x^{I_k}),x_{l_1},\dots,x_{l_m})\Big\}\\
        & \leq r \max_{v\in \overline{\operatorname{Conv}(S,H_1,\dots,H_m)}} \psi(v) <0
    \end{split}
\end{equation*}
hence $(S,H_1,\dots,H_m)$ is not $\vec{t}$-semistable. This shows that if $(S,H_1,\dots,H_m)$ is $\vec{t}$-semistable then $\mathcal{C}_{n,d,k,m,\vec{t}} \in \overline{\operatorname{Conv}(S,H_1,\dots,H_m)}$. 

Conversely, let $(S,H_1,\dots,H_m)$ be not $\vec{t}$-semistable. As such, there exists a normalised one-parameter subgroup $ \lambda(s) \vcentcolon= \operatorname{Diag}(s^{\alpha_0}, \dots, s^{\alpha_n})$ with $\sum_{j=0}^n \alpha_j = 0$ and 

$$0> \mu_{\vec{t}}(S,H_1,\dots,H_m, \lambda) = \max_{I_i\in \operatorname{Supp}(f_i)} \Bigg\{\sum_{j=0}^n \alpha_i\sum_{i=1}^{k}d_{i,j}\Bigg\}+\sum_{p=1}^mt_p a_{l_p}. $$
We define the affine transformation $\psi: \mathbb{R}^{n+1} \rightarrow \mathbb{R}$ as $\psi(z_0, \dots, z_n) = \sum_{i=0}^n \alpha_i z_i$, such that 
$$\psi|_{\overline{\operatorname{Conv}(S,H_1,\dots,H_m)}} <0$$
(by convexity), and $$\psi(\mathcal{C}_{n,d,k,m,\vec{t}}) = \frac{kd+\sum_{p=1}^mt_p}{n+1}\sum_{j=0}^n \alpha_j = 0.$$
Hence $\mathcal{C}_{n,d,k,m,\vec{t}} \notin \overline{\operatorname{Conv}(S,H_1,\dots,H_m)}$. 

A similar argument shows the result for stable orbits, where the  $>$ and $<$ are replaced accordingly by $\geq$ and $\leq$ respectively.

\end{proof}

We can also find the dimension of the moduli space. 

\begin{theorem}\label{dimension_of_mvgit}
Assume that the ground field is algebraically closed with characteristic $0$ and that the locus of stable points is not empty and $d > 2$. Then
$$\dim \overline{M}^{GIT}_{n,d,k,m}(\vec{t})= k(n+1)\Bigg( \frac{k(n+2)\dots (n+d)}{d!} - (n+1)\Bigg)-n(n+m-2)-k^2.$$
%$$\dim \overline{M}^{GIT}_{n,d,k,t}= (n+1)\Big( \frac{k(n+2)\dots (n+d)}{d!} - (n+1)\Big) +(1-k)(1+k) + n.$$
\end{theorem}
\begin{proof}
Let $p = (S,H_1,\dots, H_m)\in \mathcal{R}_m$ be a tuple. Then, we have 
%Note that by \cite{wang-moduli}, for any such complete intersection $S$, $G = G$, the stabiliser $G_S$ satisfies
$$0\leq \dim(G_p)\leq \dim(G_S\cap G_{H_1}\cap \dots \cap G_{H_m})\leq \dim(G_S) \leq \dim(\operatorname{Aut}(S))=0,$$
where the last inequality follows from \cite[Lemma 2.13]{loughran}. We obtain the result using the following identity from \cite[Corollary 6.2]{dolgachev_2003}:
\begin{equation*}
    \begin{split}
        \dim \overline{M}^{GIT}_{n,d,k,\vec{t}} &= \dim(\mathcal{R}_m) - \dim(G) +\min_{S\in \mathcal{R}_m}G_S = \\
        &= k \Big(\binom{n+d}{d} -k\Big)+mn - \big((n+1)^2-1\big)\\
        &= k(n+1)\Bigg( \frac{k(n+2)\dots (n+d)}{d!} - (n+1)\Bigg)-n(n+m-2)-k^2.
    \end{split}
\end{equation*}
\end{proof}

For the rest of this section, we will restrict ourselves to the case $m=1$. The following Lemma is a generalization of \cite[Lemma 4.1]{gallardo_martinez-garcia_2018}.

\begin{lemma}\label{maxb}
Let $(S,H)\in \mathcal{R}$, with non-empty interval of stability $[a,b]$. Then
\begin{enumerate}
    \item $a=0$ if and only if $S$ is a GIT semistable variety which is the complete intersection of $k$ degree $d$ hypersurfaces.
    \item $b \leq t_{n,d,k}= \frac{kd}{n}$. %, i.e. $\operatorname{Stab}(n,d,k,1) = \{t\in \mathbb{Q}_{\geq 0}|0\leq t\leq \frac{kd}{n}\}$.
    \item $(S,H)$ is $t_{n,k,d}$ semistable if and only if $D = S \cap H$ is a GIT semistable complete intersection of $k$ hypersurfaces of degree $d$ in $H \cong \mathbb{P}^{n-1}$.
\end{enumerate} 
\end{lemma}
\begin{proof}
For $(i)$, notice that if $t=\frac{a}{b}= 0$ if and only if $a=0$, then $\mu_t(S,H,\lambda) = \mu(S,\lambda) \geq 0$ and thus this reduces to a GIT problem for complete intersections of $k$ degree $d$ hypersurfaces, since the natural projection $\mathcal{R}\rightarrow \operatorname{Gr}(k,W)$ is $G$-invariant.

For $(ii)$, assume that $t > t_{n,d,k}$. Without loss of generality, we may assume that \\$S = f_1(x_0,\dots,x_n)\cap\dots \cap f_k(x_0,\dots,x_n)$ and that $H = \{x_n=0\}$. Then, for the normalised one-parameter subgroup $\lambda(s) = \operatorname{Diag}(s,s\dots,s, s^{-n})$ we have
$$\mu_t(S,H,\lambda)\leq kd -tn < kd -\frac{kdn}{n}=0,$$
hence $(S,H)$ is unstable, and thus $b \leq t_{n,d,k}$.

For $(iii)$, first assume that $D = S \cap H$ is unstable, where we can find a coordinate system $(x_0\colon\dots\colon x_n)$ such that $H = \{x_n=0\}$, $D = \{f'_1(x_0, \dots, x_{n-1})= \dots = f'_k(x_0, \dots, x_{n-1})\}$. Notice that under the $\lambda$-order, the monomial $x_0^{d-1}x_n$ is the maximal monomial of degree $d$ containing $x_n$ as for any normalised one-parameter subgroup $\lambda$, where $\lambda(t) = \operatorname{Diag}(t^{\lambda_0}, \dots, t^{\lambda_n})$, we have:
\begin{equation*}
    \begin{split}
        \langle x_0^{a_0}\dots x_n^{a_n},\lambda \rangle & = \sum_{i=0}^{n-1}\lambda_i a_i+\lambda_n a_n\\
        &\leq \lambda_0\sum_{i=0}^{n-1}a_i + \lambda_n a_n \quad (\text{since } \lambda_i\geq \lambda_j \text{ for } i\leq j)\\
        &\leq \lambda_0 (d-1) +\lambda_n a_n \quad (\text{since }\sum_{i=1}^n a_i = d) \\
        &\leq \lambda_0 (d-1) +\lambda_n\\
        & = \langle x_0^{d-1}x_n,\lambda \rangle. 
    \end{split}
\end{equation*}
Note that we do not have to consider the case $a_n = 0$, as then the monomial will not contain $x_n$. Hence, $\overline{S} \vcentcolon= \{f'_1(x_0, \dots, x_{n-1})+x_0^{d-1}x_n =\dots = f'_k(x_0, \dots, x_{n-1})+x_0^{d-1}x_n  =0 \}$ is the complete intersection of $k$ hypersurfaces of degree $d$ which is maximal with respect to the Hilbert-Mumford function such that $\overline{S}\cap H = D$, i.e. $\mu(\overline{S},\lambda) = \max \{\mu(D,\lambda)| S\cap H = D\}$. This further implies that if $\mu_{t_{n,d,k}} (\overline{S},H,\lambda)<0$ then for all pairs $(S,H)$ such that $S\cap H=D$, the pair is $t_{n,k,d}$-unstable. Returning to the original assumption that $D$ is unstable, we have by the centroid criterion (Theorem \ref{t-centroid_criterion}), that $$\Big\{\frac{kd}{n}, \dots, \frac{kd}{n}, 0\Big\}\notin \operatorname{Conv}(f'_1(x_0,\dots, x_{n-1})+x_0^{d-1}x_n,\dots,f'_k(x_0,\dots, x_{n-1})+x_0^{d-1}x_n).$$

Notice that,
$$A \coloneqq \{A((x^{I_1},\dots, x^{I_k}), x_n)| x^{I_i} \in \operatorname{Supp}(f'_i+x_0^{d-1}x_n) \text{ for all } i,I_i\neq I_j\} \subset M \vcentcolon=\{y_n =t_{n,d,k}\}$$
and $p = \Big\{kd-1,0\dots, 0,1+t_{n,d,k}\Big\}\notin M$. Then, since $\overline{\operatorname{Conv}({S},H)} = \operatorname{ConvexHull}(A)$, it is a pyramid with base $M$ and vertex $p$. Thus 
$\mathcal{C}_{n,d,k,t_{n,d,k}}\notin \overline{\operatorname{Conv}({S},H)}$ and thus $(S,H)$ is $t_{n,k,d}$-unstable by the centroid criterion (Theorem \ref{t-centroid_criterion}).

Now, suppose $(S,H)$ is $t_{n,k,d}$-unstable, i.e. $\mathcal{C}_{n,d,k, \frac{kd}{n}}\notin \operatorname{Conv}_{t_{n,d,k}}(S,H)$. Without loss of generality, we can assume that $H = \{x_i = 0\}$, and let 
$$p = \Big\{\frac{kd}{n}, \dots,\frac{kd}{n},0,\frac{kd}{n},\dots, \frac{kd}{n} \Big\}$$
with $0$ in the $i$-th position. In particular $p \notin \overline{\operatorname{Conv}(f_1,\dots,f_k,x_i)}$. The monomials of $D = S\cap H$ are monomials of the form $I_j = (d_{0,j},\dots, d_{i-1,j},0,d_{i+1,j}, \dots,d_{n,j})$, which correspond to faces $F \in \overline{\operatorname{Conv}_{t_{n,d,k}}(S,H)}$. The projection $F_M$ of each face $F$ to hyperplane $M = \{y_i = 0\}$ shows $p \notin F_M$. Notice that $ p = \mathcal{C}_{n,d,k}$ and $F_M \subseteq \overline{\operatorname{Conv}(f_1,\dots,f_k,h)}$, and hence $D$ is unstable by the centroid criterion (Theorem \ref{t-centroid_criterion}).

\end{proof}
%\begin{corollary}\label{space of stability conditions}
%The space of GIT stability conditions is 

%$$\operatorname{Stab}(n,d,k,m) = \Big\{\vec{t}\in (\mathbb{Q}_{\geq 0})^m|0 \leq t_i\leq \frac{kd}{n}\Big\}. $$
%\end{corollary}
%\begin{proof}
%By Lemma \ref{maxb} it is clear that 
%$$\operatorname{Stab}(n,d,k,m) \subseteq \Delta\vcentcolon=\Big\{\vec{t}\in (\mathbb{Q}_{\geq 0})^m|0 \leq t_i\leq \frac{kd}{n}\Big\}.$$
%Since $\Delta$ is convex it suffices to prove that there exists a $\vec{t}$ semistable tuple $(S,H_1,\dots,H_m)$ at each of the vertices of $\Delta$. It suffices to show this at the vertex $\vec{t} =(t_{n,d,k},\dots, t_{n,d,k})$. Let $\vec{t} = (t_1.\dots,t_m)$ where $t_i = t_{n,d,k}$ for all $i$ except for some $j$, where $t_j >t_{n,d,k}$. Without loss of generality we take $H_1 = \{x_n = 0\}$, $H_m =\{x_{n-m+1} =0\}$ and in general $H_i = \{x_{n-i+1} = 0 \}$ and one parameter subgroup 
%$$\lambda(s) = \operatorname{Diag}(s^m,s^{\frac{m(n-1)}{n-m}}\dots,s^{\frac{m(n-1)}{n-m}}, s^{-n},\dots, s^{-n})$$
%where there are $n-m$ ${\frac{m(n-1)}{n-m}}$ exponents and  $m$ $-n$ exponents in the expression. We then have

%\begin{equation*}
%    \begin{split}
%        \mu_{\vec{t}}(S,H_1,\dots,H_m,\lambda) &\leq kdm - \sum_{i=1}^m t_in\\
%        &< kdm - n\frac{kdm}{n} = 0
%    \end{split}
%\end{equation*}
%and hence the tuple is $\vec{t}$-unstable.

%The existence of semistable tuples over the rest of the vertices follow from this proof, noting that the same one parameter subgroup destabilises similar cases for different $m$, and  from Lemma \ref{maxb}, $2$.

%\end{proof}

The following theorem generalises  \cite[Theorem 1.1]{gallardo_martinez-garcia_2018}.

\begin{theorem}\label{git walls}
All GIT walls $\{t_0, \dots, t_{n,d,k}\}$ correspond to a subset of the finite set
$$\Big\{- \frac{\sum_{i=1}^{k}\langle I_i, \lambda\rangle}{\langle x_i, \lambda\rangle}\Big|I_i\neq I_i,\text{ for all  }i\neq j \text{, }  I_i\in \Xi_d \text{, } 0 \leq i \leq n, \lambda \in P_{n,d,k} \Big\}$$
and they are contained in the interval $[0, t_{n,d,k}]$. Every pair $(S, H)$ has an interval of stability $[a, b]$ with $a, b \in \{t_0, \dots, t_{n,d,k}\}$. $(S, H)$ is $t$-semistable if and only if $t\in  [t_i, t_j]$ for some walls $t_i, t_j$ . If $(S, H)$ is $t$-stable for some $t$ then $(S, H)$ is $t$-stable if and only if $t \in (t_i, t_j )$.

\end{theorem}
\begin{proof}
Since $\Xi_d$ is finite for each $d$, $\mathcal{P}(\Xi_d)^k\times\mathcal{P}(\Xi_1)$, where $\mathcal{P}$ here denotes the power set, is finite, as by the Hilbert-Mumford numerical criterion stability only depends on the support of the polynomials involved and the combination of possible supports is thus finite. In addition, by Corollary \ref{intervals_of_stab} there is a finite number of intervals of stability $[a_i,b_i]$. Hence, $t_j \in \cup_i \{a_i,b_i\}$ with $b_i\leq  t_{n,d,k}$ by Lemma \ref{maxb}. For any wall $t_i$ there is at least a pair $(S, H)$ such that
$$\mu_t(S,H)\coloneqq \max_{\lambda
\in P_{n,d,k}}\{\mu_t(S,H,\lambda)\}\geq0$$ 
for all $t\leq t_i$ and $\mu_t(S,H)<0$ for all $t > t_i$. By the continuity of $\mu_t(S, H)$, $ \mu_{t_{i}}(S, H)= 0$ and hence 
$$t_i = -\frac{\sum_{i=1}^{k}\langle I_i, \lambda\rangle}{\langle x_i, \lambda\rangle}.$$

\end{proof}

\begin{remark}
In the general case, one can find a superset of the stability walls $\vec{t}$ by solving the simultaneous equations

$$\sum_{i=1}^k \langle I_i,\lambda\rangle+\sum_{j=1}^mt_j\langle x_{i_j},\lambda\rangle=0$$
for $I_i\in \Xi_d $, $I_i\neq I_k$ for all $i \neq k$, $x_{i_j}\in \Xi_1$ and $\lambda \in P_{n,d,k,m}$. The complexity of computations increases as $m$ increases, and in addition the walls are not $0$-dimensional, which implies that we need to treat them in a stratified way. In addition, we don't have a simple one-directional way of exploring the set of stability conditions to find all walls.
\end{remark}

\begin{theorem}\label{theorem-H not in supp S}
Every point in the GIT quotient $\overline{M}^{GIT}_{n,d,k,t} \vcentcolon=\overline{M}^{GIT}_{n,d,k}(t) $ parametrises a closed orbit associated
to a pair $(S, D)$ with $D \vcentcolon= S \cap H$ in the cases where $S$ is a Calabi-Yau or a Fano complete intersection of $k$ hypersurfaces
of degree $d > 1$. Furthermore, if $S$ is Fano, $t \leq t_{n,d,k}$ and $(S, D)$ is $t$-semistable, then $S$ does not
contain a hyperplane in the support of at least one of the hypersurfaces in the complete intersection.
%, unless t = tn,d, in which case (X, D) is strictly tn,d-semistable.
\end{theorem}
\begin{proof}
Suppose that $n+1 \geq kd$ (i.e. $S$ is Fano or Calabi-Yau) and  $(S, H)$ is a pair such that $\operatorname{Supp}(H) \subset \operatorname{Supp}(f_i)$ for some $i$.
Then, it suffices to show $(S, H)$ is $t$-unstable for all $t \geq 0$. Without loss of generality, we take $H = \{x_n =0\}$ and \\$S= \{ x_nf_1(x_0,\dots,x_n) =0 \}\cap \{f_2(x_0,\dots,x_n)=0\} \dots\cap \{f_k(x_0,\dots,x_n)=0\}$. Then for the $\lambda$-order, the monomial $x_0^{d-1}x_n$ is maximal in $\operatorname{Supp}(f_1) \cup \{x_0^{d-1}x_n\}$ and the monomial $x_0^{d}$ is maximal in $\operatorname{Supp}(f_i) \cup \{x_0^{d}\}$ for all $i>1$. For any normalised one-parameter subgroup $\lambda = \operatorname{Diag}(s^{r_0}\,\dots, s^{r_n})$ we have:

$$\mu_t(S,H,\lambda) \leq (k-1)dr_0 + (d-1)r_0 +r_n +tr_n.$$
Specifically, for the normalised 1-ps $\lambda = \operatorname{Diag}(s^{\boldsymbol{r}})$ with $ \boldsymbol{r}= (n, -\frac{nkd}{n-1}, \dots, -\frac{nkd}{n-1}, -n(kd-1))$ we have
$$\mu_t(S,H,\lambda) \leq (k-1)dn +(d-1)n -n(kd-1) -tn(kd-1) =-tn(kd-1)< 0,$$
and hence $(S,H)$ is unstable for all $t>0$.

Let $S$ be a Fano complete intersection, i.e. $kd \leq n$ and we assume that $S= \\ \{x_nf^{d-1}_1(x_0,\dots,x_n)=0\} \cap\dots\cap \{x_nf_k(x_0,\dots,x_n)=0\}$ and $H \neq \{x_n = 0\}$ by the previous step. Further, assume that $\operatorname{Supp}(f_1)$ contains a hyperplane in its support. 

Then for the normalised one-parameter subgroup $\lambda = \operatorname{Diag}(s, s,\dots,s, s^{-n})$ we have, noting that $t\leq t_{n,d,k}\leq 1$, since $kd \leq n$:
\begin{equation*}
    \begin{split}
        \mu_t(S,D,\lambda) &\leq k((d-1)-n)+t\\
         &= kd-k-kn+t\\
         &\leq n-k-kn+1\\
         &= (n+1)(1-k)\\
         &<0.
    \end{split}
\end{equation*}
Hence, the pair $(S,D)$ is $t$-unstable, so $S$ cannot contain a a hyperplane in the support of at least one of the hypersurfaces in the complete intersection.
\end{proof}

\subsection{Semi-destabilizing Families}\label{sec:semi-dist fams}

\begin{definition}\label{destabilised sets def-vgit}
We fix $\vec{t} \in \operatorname{Stab}(n,d,k,m)$ and let $\lambda$ be a normalised one-parameter subgroup. A non-empty $k+m$-tuple of sets $A_1\times \dots \times A_k \times B_1\times \dots\times B_m \subseteq (\Xi_d)^k \times (\Xi_1)^m$ is \emph{maximal $\vec{t}$-(semi-) destabilised with respect to $\lambda$}, if:

\begin{enumerate}
    \item Each $k+m$-tuple $(v_1, \dots, v_k,a_1,\dots,a_m) \in A_1\times\dots\times A_k \times B_1\times \dots\times B_m$ satisfies $\sum_{i=1}^k\langle v_i, \lambda  \rangle + \sum_{j=1}^m t_j\langle a_j, \lambda  \rangle < 0$ ($\leq 0$, respectively).
    \item If there is another $k+m$-tuple of sets $\overline{A}_1\times \dots \times \overline{A}_k \times \overline{B}_1\times \dots\times \overline{B}_m \subseteq (\Xi_d)^k \times (\Xi_1)^m$ such that $A_i \subseteq \overline{A}_i$, $B_i \subseteq \overline{B}_i$ for all $i$, and for all $(v_1, \dots, v_k,a_1,\dots,a_m) \in \bar{A}_1\times \dots \times \overline{A}_k \times \overline{B}_1\times \dots\times \overline{B}_m$ the inequality $\sum_{i=1}^k\langle v_i, \lambda  \rangle + \sum_{j=1}^mt_j\langle a_j, \lambda  \rangle < 0$ ($\leq 0$, respectively) holds, then $A_i = \overline{A}_i$ and $B_j = \overline{B}_j$ for all $i, j$.
\end{enumerate}
\end{definition}

We can characterise the semi-destabilizing sets as follows, generalizing \cite[\S 5]{gallardo_martinez-garcia_2018}.

\begin{lemma}\label{nminus_k-case-vgit}
Given one-parameter subgroup $\lambda$, any maximal  $\vec{t}$-destabilised Cartesian product of sets and $\vec{t}$-semi-destabili\-sed Cartesian product of sets as in Definition \ref{destabilised sets def-vgit} with respect to $\lambda$ can be written as:

\begin{equation*}
    \begin{split}
        N^{-}_{\vec{t}}(\lambda, x^{J_1},\dots, x^{J_{k-1}}, x_{j_1},\dots,x_{j_m}) &\vcentcolon= \\
    V_{\vec{t}}^-(\lambda, x^{J_1},&\dots, x^{J_{k-1}}, x_{j_1},\dots,x_{j_m})\times \prod_{i=1}^{k-1}B^-(\lambda, x^{J_i})%, \dots, %B^-(\lambda, x^{J_{k-1}}),
        \times \prod_{p=1}^mB^{-}(\lambda, x_{j_p}),\\ %\dots, B^{-}(\lambda, x_{j_m})\big)\\
       N_{\vec{t}}^{\ominus}(\lambda, x^{J_1},\dots, x^{J_{k-1}}, x_{j_1},\dots,x_{j_m})&\vcentcolon=\\
       V_{\vec{t}}^{\ominus}(\lambda, x^{J_1},&\dots, x^{J_{k-1}},x_{j_1},\dots,x_{j_m})\times \prod_{i=1}^{k-1} B^{\ominus}(\lambda, x^{J_i})%,\dots, B^{\ominus}(\lambda, x^{J_{k-1}}),
       \times\prod_{p=1}^mB^{\ominus}(\lambda,x_{j_p}),% \dots,B^{\ominus}(\lambda,x_{j_m})\big)
    \end{split}
\end{equation*}
where $x^{J_i}\in \Xi_d$ are support monomials with $J_r \neq J_s$ for all $r, s$, $x_{j_p} \in \Xi_1$ are arbitrary support monomials and
\begin{equation*}
    \begin{split}
        V^-_{\vec{t}}(\lambda, x^{J_1},\dots, x^{J_{k-1}},x_{j_1},\dots,x_{j_m}) &\vcentcolon= \{x^I \in \Xi_d| \langle I, \lambda \rangle + \sum_{i=1}^{k-1}\langle J_i , \lambda \rangle +\sum_{p=1}^m t_p\langle x_{j_p}, \lambda \rangle <0\},\\
        B^-(\lambda, x^{J_i})&\vcentcolon= \{x^J \in \Xi_d| x^J \leq_{\lambda} x^{J_i} \},\\
        V^{\ominus}_{\vec{t}}(\lambda, x^{J_1},\dots, x^{J_{k-1}},x_{j_1},\dots,x_{j_m}) &\vcentcolon= \{x^I \in \Xi_d| \langle I, \lambda \rangle + \sum_{i=1}^{k-1}\langle J_i , \lambda \rangle +\sum_{p=1}^m t_p\langle x_{j_p},\lambda\rangle \leq 0\},\\
  B^{\ominus}(\lambda, x^{J_i})&\vcentcolon= \{x^J \in \Xi_d| x^J \leq_{\lambda} x^{J_i}\},\\
  B^-(\lambda, x_{j_p}) \vcentcolon= \{x_i \in \Xi_1 | x_i \leq x_{j_p}\}, &\quad B^{\ominus}(\lambda, x_{j_p}) \vcentcolon= \{x_i \in \Xi_1 | x_i \leq x_{j_p}\}.
    \end{split}
\end{equation*}

%$$V^-_{\vec{t}}(\lambda, x^{J_1},\dots, x^{J_{k-1}},x_{j_1},\dots,x_{j_m}) \vcentcolon= \{x^I \in \Xi_d| \langle I, \lambda \rangle + \sum_{i=1}^{k-1}\langle J_i , \lambda \rangle +\sum_{p=1}^m t_p\langle x_{j_p}, \lambda \rangle <0\},$$  
%$$ B^-(\lambda, x^{J_i})\vcentcolon= \{x^J \in \Xi_d| x^J \leq_{\lambda} x^{J_i} \},$$
%$$V^{\ominus}_{\vec{t}}(\lambda, x^{J_1},\dots, x^{J_{k-1}},x_{j_1},\dots,x_{j_m}) \vcentcolon= \{x^I \in \Xi_d| \langle I, \lambda \rangle + \sum_{i=1}^{k-1}\langle J_i , \lambda \rangle +\sum_{p=1}^m t_p\langle x_{j_p} \leq 0\}, $$ 
%$$B^{\ominus}(\lambda, x^{J_i})\vcentcolon= \{x^J \in \Xi_d| x^J \leq_{\lambda} x^{J_i}\}. $$
%$$ B^-(x_{j_i}, \lambda) \vcentcolon= \{x_i \in \Xi_1 | x_i \leq x_{j_i}\}, \quad B^{\ominus}(x_{j_i}, \lambda) \vcentcolon= \{x_i \in \Xi_1 | x_i \leq x_{j_i}\}.$$
\end{lemma}
\begin{proof}
Let $ \Delta\vcentcolon=(A_1,\dots, A_k, B_1,\dots,B_m)$ be a maximal $\vec{t}$-(semi-)destabilised $k+m$-tuple with respect to $\lambda$. Let $x^{J_{i-1}} = \max{(A_{i})}$, for $2\leq i\leq k$ be the maximal element of $A_i$ with respect to the $\lambda$-order and $x_{j_p} = \max(B_p)$, for $1\leq p\leq m$. By the $\lambda$-order we have
$$\sum_{i=1}^{k-1}\langle I_1, \lambda  \rangle +\sum_{p=1}^m t_p\langle x_{l_p},\lambda\rangle  \leq  \langle I_1, \lambda  \rangle + \sum_{i=1}^{k-1}\langle J_i, \lambda  \rangle +\sum_{p=1}^m t_p\langle x_{j_p},\lambda \rangle < 0 \quad (\leq\text{ respectively)}$$
for all $(x^{I_0}, \dots, \dots x^{I_{k}}, x_{l_1}\dots,x_{l_m})\in \Delta$.
This implies that 
$$\Delta \subseteq N^{-}_{\vec{t}}(\lambda, x^{J_1},\dots, x^{J_{k-1}}, x_{j_1},\dots,x_{j_m})$$
$$(\Delta \subseteq N^{\ominus}_{\vec{t}}(\lambda, x^{J_1},\dots, x^{J_{k-1}}, x_{j_1},\dots,x_{j_m}), \text{ respectively})$$
and the maximality condition of Definition \ref{destabilised sets def-vgit} implies the equality.

%$(A_1,\dots, A_k,B) = N^{-}_t(\lambda, x^{J_1},\dots, x^{J_{k-1}},x_j)$ , which shows that $N^{-}_t(\lambda, x^{J_1},\dots, x^{J_{k-1}})$ is a  maximal $t$-semi-destabilised $k$-tuple with respect to $\lambda$.

%Similarly, let $(A_1,\dots, A_k,B)$ be a maximal $t$-destabilised $k$-tuple with respect to $\lambda$. Let $x^{J_{i-1}} = \max{(A_{i})}$, for $2\leq i\leq k$ be the maximal element of $A_i$ with respect to the $\lambda$-order, and $x_j = \max(B)$. Then
%$$\sum_{i=0}^{k-1}\langle I_i, \lambda  \rangle +t\langle x_l,\lambda\rangle  \leq  \langle I_0, \lambda  \rangle + \sum_{i=1}^{k-1}\langle J_i, \lambda  \rangle +t\langle x_j,\lambda\rangle < 0, \quad$$
%for all $(x^{I_0}, \dots, \dots x^{I_{k}}, x_l)\in A_1\times\dots \times A_k\times B$.
%This implies that $(A_1,\dots, A_k,B) \subseteq N^{\ominus}_t(\lambda, x^{J_1},\dots, x^{J_{k-1}},x_j)$ and the maximality condition shows that  $(A_1,\dots, A_k,B) = N^{\ominus}_t(\lambda, x^{J_1},\dots, x^{J_{k-1}},x_j)$, which shows that $N^{\ominus}_t(\lambda, x^{J_1},\dots, x^{J_{k-1}},x_j)$ is a  maximal $t$-destabilised $k$-tuple with respect to $\lambda$.
\end{proof}

\begin{theorem}\label{unstable families vgit}
Let $\vec{t} \in \operatorname{Stab}(n,d,k,m)$. A tuple $(S, H_1,\dots,H_m)$ is not $\vec{t}$-stable ($\vec{t}$-unstable, respectively), if and only if there exists $g \in G$, $\lambda \in P_{n,d,k,m}$, such that the set of monomials associated to $(g \cdot S, g \cdot H_1,\dots,g \cdot H_m)$ is contained in a pair of sets $N_{\vec{t}}^{\ominus}(\lambda, x^{J_1},\dots, x^{J_{k-1}}, x_{j_1}\dots,x_{j_m})$ \\($N_{\vec{t}}^{-}(\lambda, x^{J_1},\dots, x^{J_{k-1}}, x_{j_1}\dots,x_{j_m})$, respectively) defined in Lemma \ref{unstablelemma-vgit}.
Furthermore, the sets\\ $N_{\vec{t}}^{\ominus}(\lambda, x^{J_1},\dots, x^{J_{k-1}}, x_{j_1}\dots,x_{j_m})$ and $N_{\vec{t}}^{-}(\lambda, x^{J_1},\dots, x^{J_{k-1}}, x_{j_1}\dots,x_{j_m})$ which are maximal with respect to the containment order of sets define families of non-$\vec{t}$-stable tuples ($\vec{t}$-unstable tuples, respectively) in $\mathcal{R}_{n,d,k,m}$.
Any not $\vec{t}$-stable (respectively $\vec{t}$-unstable) tuple $(g \cdot S, g \cdot H_1,\dots, g\cdot H_m)$ belongs to one of these families for
some group element $g$.
\end{theorem}
\begin{proof}
Let $(S,H_1,\dots,H_m)$ be $\vec{t}$-unstable ($\vec{t}$-non stable respectively). Then by Lemma \ref{unstablelemma-vgit} there is $g \in G$ and $\lambda \in P_{n,d,k,m}$ such that 

$$\mu_{\vec{t}}(g \cdot(S,H_1,\dots,H_m), \lambda) <0 \quad (\leq 0, \quad \text{respectively.)}$$
Then, every $(x^{I_1},\dots,  x^{I_{k}}, x_{j_1},\dots,x_{j_m}) \in g\cdot (\operatorname{Supp}(f_1),\dots, \operatorname{Supp}(f_k), \operatorname{Supp}(h_1),\dots, \operatorname{Supp}(h_m))$ satisfies 
$$\sum_{i=1}^{k}\langle I_i, \lambda  \rangle +\sum_{p=1}^mt_p\langle x_{j_p},\lambda\rangle < 0\quad  (\leq 0,\quad \text{respectively)}.$$ 
By Definition \ref{destabilised sets def-vgit} and Lemma \ref{unstablelemma-vgit}, $g\cdot \operatorname{Supp}(f_1) \subseteq V^-_{\vec{t}}(\lambda, x^{J_1}, \dots, x^{J_{i-1}})$ and $g\cdot \operatorname{Supp}(f_i) \subseteq B^-(\lambda, x^{J_{i-1}})$, $g\cdot \operatorname{Supp}(H_i) \subseteq B^-(\lambda, x_{j_i})$ hold for some $\lambda \in P_{n,d,k,m}$, $x^{J_i}\in \Xi_d$, and $x_{j_i} \in \Xi_1$ ($g\cdot \operatorname{Supp}(f_1) \subseteq V^{\ominus}_{\vec{t}}(\lambda, x^{J_1}, \dots, x^{J_{i-1}})$ and $g\cdot \operatorname{Supp}(f_i) \subseteq B^{\ominus}(\lambda, x^{J_{i-1}})$, $g\cdot \operatorname{Supp}(H_i) \subseteq B^{\ominus}(\lambda, x_{j_i})$, respectively). Choosing the maximal Cartesian products of sets $ N^{-}_t(\lambda, x^{J_1},\dots, x^{J_{k-1}},x_{j_1}\dots,x_{j_m})$, ($ N^{\ominus}_t$ respectively) under the containment order where $\lambda \in P_{n,d,k,m}$, $x^{J_i} \in \Xi_d$, and $x_{j_p} \in \Xi_1$ we obtain families of Cartesian products of sets whose coefficients belong to maximal $\vec{t}$-(semi-)destabilised $k+m$-tuples. For the opposite direction, note that if the monomials associated to $(g \cdot S, g \cdot H_1,\dots,g \cdot H_m)$ are contained in $N_{\vec{t}}^{-}$ ($N_{\vec{t}}^{\ominus}$, respectively), then 
$$\mu_{\vec{t}}(g \cdot(S,H_1,\dots,H_m), \lambda) <0 \quad (\leq 0, \quad \text{respectively)}$$
and $(S,H_1,\dots,H_m)$ is $\vec{t}$-unstable ($\vec{t}$-non stable respectively).

\end{proof}

We can also define the annihilator as in \cite[\S 5]{gallardo_martinez-garcia_2018}:

\begin{proposition}\label{annihilator_k-case-vgit}
For $\vec{t} \in \operatorname{Int}\big(\operatorname{Stab}(n,d,k,m) \big)$ and normalised one-parameter subgroup $\lambda$ the annihilator of $\lambda$ with respect to $x^{J_1}, \dots, x^{J_{k-1}}, x_{j_1},\dots,x_{j_m}$ is the set
\begin{equation*}
    \begin{split}
       \operatorname{Ann}_{\vec{t}}(\lambda, x^{J_1}, \dots, x^{J_{k-1}}, x_{j_1},\dots,x_{j_m}) &\vcentcolon=  \Big\{(x^I, x^{I_1}, \dots x^{I_{k-1}}, x_{i_1},\dots,x_{i_m}) \in \\ N_{\vec{t}}^{\ominus}(\lambda, x^{J_1}, \dots, x^{J_{k-1}},x_{j_1},\dots, x_{j_m})
       &\Big| \langle I, \lambda  \rangle + \sum_{i=1}^{k-1}\langle I_i, \lambda  \rangle +\sum_{p=1}^m t_p\langle x_{i_p}, \lambda \rangle= 0\Big\}.
    \end{split}
\end{equation*}

%$$\operatorname{Ann}_{\vec{t}}(\lambda, x^{J_1}, \dots, x^{J_{k-1}}, x_{j_1},\dots,x_{j_m}) \vcentcolon=$$ 
%$$  \{(x^I, x^{I_1}, \dots x^{I_{k-1}}, x_{i_1},\dots,x_{i_m}) \in N_{\vec{t}}^{\ominus}(\lambda, x^{J_1}, \dots, x^{J_{k-1}},x_{j_1},\dots, x_{j_m})\Big| \langle I, \lambda  \rangle \\+ \sum_{i=1}^{k-1}\langle I_i, \lambda  \rangle +\sum_{p=1}^m t_p\langle x_{i_p}, \lambda \rangle= 0\}.$$
If this is not empty, it is equal to the Cartesian product $$V^{0}_{\vec{t}}(\lambda,x^{J_1}, \dots, x^{J_{k-1}},x_{j_1},\dots,x_{j_m}) \times \prod_{i=1}^{k-1}B^0(\lambda, x^{J_{i}}) \times \prod_{p=1}^m B^0(\lambda, x_{j_p}),$$
where 
$$V^0_{\vec{t}}(\lambda, x^{J_1}, \dots, x^{J_{k-1}},x_{j_1},\dots,x_{j_m}) \vcentcolon=$$
$$=\{x^I \in V^{\ominus}_{\vec{t}}(\lambda, x^{J_1}, \dots, x^{J_{k-1}},x_{j_1},\dots,x_{j_m})|\exists x^{J'_i} \in  B^{\ominus}(x^{J_i}), x_{i_p} \in B^{\ominus}(\lambda, x_{j_p}), $$
$\text{ such that } \langle I, \lambda \rangle + \sum_{i=1}^{k-1}\langle J'_i , \lambda \rangle +\sum_{p=1}^m t_p\langle x_{i_p}, \lambda \rangle= 0\}$, 
\begin{equation*}
    \begin{split}
        B^0(\lambda, x^{J_i})&\vcentcolon= \{x^J \in B^{\ominus}(\lambda, x^{J_i})| x^{\overline{J}} \leq_{\lambda} x^{J} \text{ for all } x^{\overline{J}} \in  B^{\ominus}(\lambda, x^{J_i})\},\\
        B^0(\lambda, x_{j_p})&\vcentcolon= \{x_{i} \in B^{\ominus}(\lambda, x_{j_p})|\langle x_{k}, \lambda \rangle \leq \langle x_{i}, \lambda \rangle \text{ for all } x_{k} \in  B^{\ominus}(\lambda, x_{j_p})\}.
    \end{split}
\end{equation*}

%$$B^0(\lambda, x^{J_i})\vcentcolon= \{x^J \in B^{\ominus}(\lambda, x^{J_i})| x^{\overline{J}} \leq_{\lambda} x^{J} \text{ for all } x^{\overline{J}} \in  B^{\ominus}(\lambda, x^{J_i})\}.$$
%$$B^0(\lambda, x_{j_p})\vcentcolon= \{x_{i} \in B^{\ominus}(\lambda, x_{j_p})|\langle x_{k}, \lambda \rangle \leq \langle x_{i}, \lambda \rangle \text{ for all } x_{k} \in  B^{\ominus}(\lambda, x_{j_p})\}$$

\end{proposition}
\begin{proof}
For one direction, let $(x^{I_1}, \dots, x^{I_k},x_{i_1},\dots,x_{i_m}) \in \operatorname{Ann}_{\vec{t}}(\lambda, x^{J_1}, \dots, x^{J_{k-1}},x_{j_1},\dots,x_{j_m})$. This implies that $x^{I_1} \in V^{0}_{\vec{t}}(\lambda,x^{J_1}, \dots, x^{J_{k-1}},x_{j_1},\dots,x_{j_m}) $. Suppose that there exist $x^{\overline{I}_i}$, for $2\leq i\leq k$ such that $x^{I_i} <_{\lambda} x^{\overline{I}_i}$ and $x_{l_p}$ such that $x_{i_p} <_{\lambda}x_{l_p} $. Without loss of generality we can take $i=2$ such that there exists $x^{\overline{I}_2}\in B^{\ominus}(\lambda, x^{J_2})$ such that $x^{I_2} <_{\lambda} x^{\overline{I}_2}$. Then either
$$0 = \sum_{i=1}^{k}\langle I_i, \lambda  \rangle +\sum_{p=1}^mt_p\langle x_{i_p},\lambda \rangle  < \langle I_1, \lambda  \rangle +\langle\overline{I}_2, \lambda  \rangle + \sum_{i=2}^{k} \langle I_i, \lambda  \rangle
 +\sum_{p=1}^mt_p\langle x_{i_p},\lambda \rangle,$$
or 
$$0 = \sum_{i=1}^{k}\langle I_i, \lambda  \rangle +\sum_{p=1}^mt_p\langle x_{i_p},\lambda \rangle  < \sum_{i=1}^{k}\langle I_i, \lambda  \rangle + \sum_{p=1}^mt_p\langle x_{l_p},\lambda \rangle.$$
%$$0 = \sum_{i=1}^{k}\langle I_i, \lambda  \rangle +\sum_{p=1}^m t_p\langle x_{i_p},\lambda \rangle  < \langle I_i, \lambda  \rangle + \sum_{i=2}^{k}\langle \overline{I}_i, \lambda  \rangle +\sum_{p=1}^m t_p\langle x_{l_p},\lambda \rangle.$$
In both cases this would imply that 
$$(x^{I_1}, x^{\overline{I}_2}, \dots, x^{\overline{I}_k}, x_{i_1},\dots,x_{i_m}) \notin  N^{\ominus}_{\vec{t}}(\lambda, x^{J_1}, \dots, x^{J_{k-1}},x_{j_1},\dots,x_{j_m}),$$
which is a contradiction.

%hen either

%; hence $x^{I_i} \in B^0(\lambda, x^{J_{i-1}})$, $x_i \in B^0(\lambda,x_j)$ i.e. $\operatorname{Ann}_t(\lambda, x^{J_1}, \dots, x^{J_{k-1}},x_j)\subseteq V^{0}_t(\lambda, x^{J_1}, \dots, x^{J_{k-1}},x_j) \times B^{0}(\lambda, x^{J_1})\times \dots \times B^{0}(\lambda, x^{J_{k-1}}) \times B^0(\lambda,x_j)$.

Now let $$(x^{I_1}, \dots, x^{I_k},x_{i_1},\dots,x_{i_m}) \in V^{0}_{\vec{t}}(\lambda,x^{J_1}, \dots, x^{J_{k-1}},x_{j_1},\dots,x_{j_m}) \times \prod_{i=1}^{k-1}B^0(\lambda, x^{J_{i-1}}) \times \prod_{p=1}^m B^0(\lambda, x_{j_p}).$$

Then, there exist $x^{I'_{i}}\in B^{\ominus}(\lambda, x^{J_{i-1}})$ for $2\leq i \leq k$ and $x_{i'_p}\in B^{\ominus}(\lambda,x_{j_p})$ such that 
$$\langle I_1, \lambda \rangle + \sum_{i=2}^{k}\langle I'_{i} , \lambda \rangle +\sum_{p=1}^m t_p\langle x_{i'_p},\lambda\rangle= 0.$$ 
Also, because $x^{I_i} \in B^{0}(\lambda, x^{J_i})$, $x^{I'_i} \leq x^{I_i}$, and $x_{i_p} \in B^0(\lambda,x_{j_p})$, $x_{i'_p}\leq x_{i_p}$ we obtain:
$$0 = \langle I_1, \lambda \rangle + \sum_{i=2}^{k}\langle I'_{i} , \lambda \rangle +\sum_{p=1}^m t_p\langle x_{i'_p},\lambda\rangle \leq \sum_{i=1}^{k}\langle I_i, \lambda \rangle +\sum_{p=1}^m t_p\langle x_{i_p},\lambda\rangle\leq 0 $$
since  $(x^{I_1}, \dots, x^{I_k},x_{i_1},\dots,x_{i_m}) \in N^{\ominus}_{\vec{t}}(\lambda, x^{J_1},\dots, x^{J_{k-1}},x_{j_1},\dots,x_{j_m})$. Hence 
$$(x^{I_1}, \dots, x^{I_k},x_{i_1},\dots,x_{i_m}) \in \operatorname{Ann}_{\vec{t}}(\lambda, x^{J_1}, \dots, x^{J_{k-1}},x_{j_1},\dots,x_{j_m})$$ 
and the result follows.

%, i.e. $V^{0}_t(\lambda,x^{J_1}, \dots, x^{J_{k-1}}) \times B^{0}(\lambda, x^{J_1})\times \dots \times B^{0}(\lambda, x^{J_{k-1}})  \times B^0(\lambda,x_j)\subseteq \operatorname{Ann}_t(\lambda, x^{J_1}, \dots, x^{J_{k-1}},x_j).$

\end{proof}

\begin{theorem}\label{strictly_semistable_k-case-vgit}
If the tuple $(S,H_1,\dots,H_m)$ belongs to a closed strictly $\vec{t}$-semistable orbit, there is $g \in SL(n+1)$, $\lambda \in P_{n,k,d,m}$ and support monomials $x^{J_1},\dots, x^{J_{k-1}} \in \Xi_d$, $x_{j_1},\dots,x_{j_m} \in \Xi_1$ such that the set of monomials associated to\\ $g \cdot \big((\operatorname{Supp}(f_1)\times\dots\times\operatorname{Supp}(f_k) ), \operatorname{Supp}(h_1),\dots, \operatorname{Supp}(h_m)\big)$ corresponds to those in the $k$-product of sets 
$$V^{0}_{\vec{t}}(\lambda,x^{J_1}, \dots, x^{J_{k-1}},x_{j_1},\dots,x_{j_m}) \times \prod_{i=1}^{i}B^0(\lambda, x^{J_{k-1}}) \times \prod_{p=1}^m B^0(\lambda, x_{j_p}).$$

\end{theorem}
\begin{proof}
Let $M = (S,H_1,\dots,H_m)$. By \cite[Remark 8.1 (5)]{dolgachev_2003}, since $M$ is strictly $\vec{t}$-semistable and represents a closed orbit, the stabilizer subgroup $G_M \subset G $ is infinite. Hence, there exists one parameter subgroup $\lambda \in G_M$ where $\lim_{s \to 0}\lambda(s)\cdot M = M$, i.e. $\mu_{\vec{t}}(S,H_1,\dots,H_m, \lambda) = 0$. By choosing an appropriate coordinate system and applying Lemma \ref{unstablelemma-vgit}, we may assume that $\lambda \in P_{n,d,k,m}$ and 
\begin{equation*}
    \begin{split}
        g\cdot& \big((\operatorname{Supp}(f_1),\dots, \operatorname{Supp}(f_k)),\operatorname{Supp}(h_1),\dots,\operatorname{Supp}(h_m)\big) = \operatorname{Ann}_{\vec{t}}(\lambda, x^{J_1}, \dots, x^{J_{k-1}},x_{j_1},\dots,x_{j_m})\\
        &=V^{0}_{\vec{t}}(\lambda,x^{J_1}, \dots, x^{J_{k-1}},x_{j_1},\dots,x_{j_m}) \times \prod_{i=1}^{k-1}B^0(\lambda, x^{J_{i}}) \times \prod_{p=1}^m B^0(\lambda, x_{j_p})
    \end{split}
\end{equation*}
by Proposition \ref{annihilator_k-case-vgit}. 
%$g \cdot \big((\operatorname{Supp}(f_1),\dots, \operatorname{Supp}(f_k)),\operatorname{Supp}(h_1),\dots,\operatorname{Supp}(H_m)\big) = \operatorname{Ann}_{\vec{t}}(\lambda, x^{J_1}, \dots, x^{J_{k-1}},x_{j_1},\dots,x_{j_m}) =  V^{0}(\lambda,x^{J_1}, \dots, x^{J_{k-1}},x_j) \times B^{0}(\lambda, x^{J_1})\times \dots \times B^{0}(\lambda, x^{J_{k-1}})\times B^0 (\lambda,x_j) $ 

\end{proof}

\begin{remark}
Note, that one should not expect the converse to be true.
\end{remark}

\begin{remark}
Both characterisations of the annihilator are necessary. The original definition given in Proposition \ref{annihilator_k-case-vgit} is mostly beneficial for proofs, while the `product' definition is mostly beneficial for algorithms and their implementations in generating all unstable sets.
\end{remark}

\subsection{How to study VGIT quotients computationally}\label{sec:how to study VGIT}
We will describe how the above sections give us a toolkit to study VGIT quotients (in the case $m=1$) computationally. This follows the methodology in Gallardo--Martinez-Garcia and Laza \cite{gallardo_martinez-garcia_2018, zhang_2018, laza-cubics}.

\begin{itemize}
    \item By Theorem \ref{git walls}, we know that the stability conditions of $(S,H)$ are determined by a finite set of one-parameter subgroups $P_{n,d,k}$ which can be determined using Definition \ref{finite_set_def_k_case} computationally in computer programs such as Python and Sage.
    \item Using this set one can find a superset of the GIT stability walls through solving the equations 
    $$\sum_{i=1}^k \langle I_i,\lambda\rangle+t\langle x_{i},\lambda\rangle=0$$
    for $I_i\in \Xi_d $, $I_i\neq I_k$ for all $i \neq k$, $x_{i}\in \Xi_1$ and $\lambda \in P_{n,d,k}$.
    \item Knowing the above, for each $t$, either a potential stability interval or equal to a potential stability wall, as determined by the superset previously introduced, one can compute the $N^{\ominus}_{t}(\lambda)$ and $N^{-}_{t}(\lambda)$ for each $\lambda \in P_{n,d,k}$ from their definitions in Lemma \ref{nminus_k-case-vgit} and find the maximal ones among them. For elements of $N^{\ominus}_{t}$ which parametrise non-$t$-stable pairs by Theorem \ref{unstable families vgit} the centroid criterion distinguishes which families are strictly $t-$semistable.
    \item For each of these of the strictly $t$-semistable families one can compute the annihilator $N^{0}_{t}(\lambda)\coloneqq \operatorname{Ann}(\lambda)$, using Proposition \ref{annihilator_k-case-vgit},  which correspond to potentially strictly $t$-polystable orbits.
    \item For each wall $t$, we compare the different sets $N^{\ominus}_{t}(\lambda)$ and $N^{-}_{t}(\lambda)$ obtained between the wall and the previous chamber, and in the cases where these are identical or the maximal families of the chamber are contained in the wall, we remove the wall as a false wall. Although this does reduce the steps to follow, it does not give a complete characterisation of the VGIT quotient.
    \item The above have been implemented in a computational package written in Sage (and in Python) which can be used for arbitrary initial parameters $n,k,d$. \cite{theodoros_stylianos_papazachariou_2022}. This in turn generalises a computational package \cite{bata458}, which deals with VGIT quotients of hypersurfaces, based on \cite{gallardo_martinez-garcia_2018}.
    \item For a geometric characterisation of VGIT quotients, one needs to study the singularities of pairs $(S,H)$ for each stability condition, which are dependent only on $n,d,k$. In the following sections, we will do so for the case of a complete intersection of two quadrics in $\mathbb{P}^2, \mathbb{P}^3$ and $\mathbb{P}^4$ and a single hyperplane section, which we obtained using our computer program.
    \item A precise characterisation of each VGIT quotient allows us to remove excess walls, but the algorithm alone cannot guarantee this at the moment. This is due to the fact that we don't have a method to prove that two non-isomorphic maximally destabilised families, which are not isomorphic via transformations by the specific maximal torus used which normalises the one-parameter subgroups, are not isomorphic via a different element of $G$.
\end{itemize}

The above computational methodology could potentially be extended to the general case of tuples $(S,H_1,\dots,H_m)$ as well. This has, however, not yet been implemented in Sage. The reason for this, is the added algorithmical complexity needed in calculating the stability walls $\vec{t}$, and thereafter the maximal (semi-)destabilised sets.

\section{GIT for complete intersection of two quadrics in \texorpdfstring{$\mathbb{P}^3$}{TEXT}}\label{p3 VGIT section}
%For two quadratic polynomials $f,g \in \mathbb{P}^3$, let $\Phi(f,g)$ be their pencil, with general element $\lambda f+\mu g$, and let $S$ be its base locus, such that $S = \{f = g =0\}$. 

\subsection{Some Results on the Singularities of Pencils of Quadrics}
For a quadric $q$ in $\mathbb{P}^n$ we can write $q(x) = xQx^T$, for $Q$ a $(n+1)\times (n+1)$ symmetric matrix with entries in $\mathbb{C}$. We denote by $\Phi(f,g) \in \operatorname{Gr}(2, \frac{(n+2)(n+1)}{2})$ the element of the Grassmanian naturally representing two quadrics $f$, $g$, in $\mathbb{P}^n$, i.e. $\Phi(f,g) \vcentcolon= \{\lambda f+\mu g |(\lambda, \mu)\in \mathbb{P}^1\}$. This pencil can also be written in terms of the symmetric matrices $F$ and $G$ of $f$ and $g$, respectively, i.e. $\Phi(f,g) = \{\lambda F+\mu G |(\lambda, \mu)\in \mathbb{P}^1\}$ (see \cite[\S 1]{reid}). The notion of the stability of pencils is defined in Section \ref{VGIT_prelims}, where, following from our previous discussion, we now define $k=d=2$. Note that for a complete intersection of quadrics, $S = \{f=0\}\cap \{g=0\}$, $S = \operatorname{Bs}(\Phi(f,g))$  is the base-locus of a pencil with no fixed part (see \cite[\S XIII]{sommerville}, or  \cite[\S XIII]{hodge_pedoe_1994}).

Note here, that a quadric $q$ is smooth (nondegenerate) if and only if the corresponding symmetric matrix $Q$ is non-degenerate, i.e. if $\det(Q)\neq 0$ (see for example Reid \cite[\S 1]{reid}). For a pencil $\Phi(f,g)$ of two quadrics in $\mathbb{P}^n$ with $f$ smooth, (i.e. the rank of the corresponding symmetric matrix $F$ is $n+1$) we consider the polynomial $\det(\lambda F+G)$ of degree $n+1$ with distinct roots $\alpha_1, \dots, \alpha_r$.

The Lemma below proves quite useful for detecting singular complete intersections of quadrics in $\mathbb{P}^n$. 

\begin{lemma}\label{singular intersection in pn}
Let $f$, $q$ be two quadrics in $\mathbb{P}^n$. Their complete intersection $S = f \cap q$ is singular if and only if up to an action of $\operatorname{SL}(n+1)$ the quadrics can be written either as 
\begin{equation*}
    \begin{split}
        f(x_0, \dots, x_n) &= q_1(x_1, \dots, x_n)\\
        q(x_0, \dots, x_n) &= x_0(b_1x_1+\dots+b_nx_n) +q_2(x_1, \dots, x_n)
    \end{split}
\end{equation*}
or 
\begin{equation*}
    \begin{split}
        f(x_0, \dots, x_n) &= a_0x_0x_n +q_1(x_1, \dots, x_n)\\
        q(x_0, \dots, x_n) &= b_nx_0x_n +q_2(x_1, \dots, x_n)
    \end{split}
\end{equation*}
or a degeneration of the above.
\end{lemma}
\begin{proof}
Without loss of generality, we assume that the singular point is $P=(1:0:\dots:0)$. Then since $P \in f \cap q$ we have:

\begin{equation*}
    \begin{split}
        f(x_0, \dots, x_n) &= x_0l_1(x_1,\dots, x_n) +q_1(x_1, \dots, x_n)\\
        q(x_0, \dots, x_n) &= x_0l_2(x_1,\dots, x_n) +q_2(x_1, \dots, x_n),
    \end{split}
\end{equation*}
where the $l_i$ are linear and the $q_i$ are quadratic forms.

We can choose a coordinate transformation fixing $P$ such that $x_n = l_1(x_1,\dots,x_n)$ and then 
\begin{equation*}
    \begin{split}
        f(x_0, \dots, x_n) &= a_0x_0x_n +q_1(x_1, \dots, x_n)\\
        q(x_0, \dots, x_n) &= x_0(b_1x_1+\dots+b_nx_n) +q_2(x_1, \dots, x_n).
    \end{split}
\end{equation*}
Recall that a point $P$ is singular on the intersection if and only if the matrix 
$$J = \begin{pmatrix}
\frac{\partial f}{\partial x_i}\\
\frac{\partial q}{\partial x_i}
\end{pmatrix}$$
at $P$ has rank $<2$. Then since 
$$J_p = \begin{pmatrix}
0&0&\dots&0&a_0\\
0&b_1&\dots&b_{n-1}&b_n
\end{pmatrix}$$
we see that $\operatorname{rank}J_p <2$ if either $a_0 = 0$ or $b_1 = b_2=\dots=b_{n-1}=0$, and the result follows.
\end{proof}

\begin{corollary}\label{smooth_pn}
Let $f,q$ be two quadrics in $\mathbb{P}^n$. Their complete intersection $S = f\cap q$ is smooth if and only if the determinant polynomial $\det(\lambda f+q)$ has only simple roots.
\end{corollary}
\begin{proof}
We assume without loss of generality that $P = (1:0:\dots:0) \in q\cap r$, hence we can write 
\begin{equation*}
    \begin{split}
        q(x_0, \dots, x_n) &= a_0x_0x_n +q_1(x_1, \dots, x_n)\\
        r(x_0, \dots, x_n) &= x_0(b_1x_1+\dots+b_nx_n) +q_2(x_1, \dots, x_n)
    \end{split}
\end{equation*}
and since $S$ is smooth,  $\alpha_0 \neq 0\neq b_i$, for $i=1, \dots,n$. The determinant polynomial $\det(\lambda q+r)$ thus has $n+1$ distinct roots. 
\end{proof}

%\begin{lemma}\label{Projective equivalnce of smooth quadrics}
%Let $q_1,q_2$ be two smooth quadrics in $\mathbb{P}^n$. Then they are projectively equivalent. 
%\end{lemma}
%\begin{proof}
%Let $Q_1$, $Q_2$ be the corresponding symmetric matrices. We can find general element $G \in \operatorname{PGL}(n+1)$ such that $\det(G\cdot Q_1) = \det(Q_2)\neq 0 $, and hence $G\cdot Q_1 = Q_2$.
%\end{proof}

\subsubsection{Segre Symbols}\label{sec: segre symb def}
%The below information is taken from an unpublished note of Pieter Belmans \cite{Belamns_segre} on Segre symbols. 
Let $\alpha_i$ be a root of the determinant polynomial of multiplicity $e_i$. Assume further that $\alpha_i$ is not only a zero
of $\det(G+\lambda F)$, but also of all its subdeterminants of size $n-h_i +2$, where  $n+1\geq h_i \geq 2$, and $h_i$ is the maximal number such that for each root $\alpha_i$, $\alpha_i$ is a solution of all the non-trivial subdeterminants of size $n-h_i-2$ of $G+\lambda F$. If $h_i = 1$, this implies that the solution is not a solution of any of the subdeterminants.
%, then the dimension of the vertex of the cone is precisely $h_i -1$.

We then define $l^i_j$ to be the minimum multiplicity of the root $\alpha_i$ for the set of subdeterminants of size $n+1-j$, for $j =0,1,\dots ,h_i-1$. We have $l^i_j \geq l^{i}_{j+1}$, and we define $e^i_j = l^i_j - l^{i}_{j+1}$. Thus, we obtain a factorisation
$$\det(G+\lambda F)=\prod_{j = 0}^{h_i-1} (\lambda - \alpha_i)^{e^i_j} f_i(\lambda)$$
where $f_i(\alpha_i)\neq 0$ (see \cite[\S XIII]{hodge_pedoe_1994} or \cite[\S 13.86]{sommerville}).

\begin{definition}\label{segre symbol}
The \emph{Segre symbol} of the pencil is $[(e^0_0,\dots e^0_{h_0-1}),\dots,((e^r_0,\dots e^r_{h_r-1})) ].$
\end{definition}
Note, that if we only have a $1$-tuple, we omit the brackets around the tuple.

\begin{theorem}[Weierstrass, Segre, \cite{weierstrass_2013}]\label{segre thm}
Consider two pencils $\Phi_1$ and $\Phi_2$ of quadric hypersurfaces 
in $\mathbb{P}^n$. Then their base loci are projectively equivalent if and only if they have the
same Segre symbol and there exists an automorphism of $\mathbb{P}^1$ taking each root $(1: \alpha_i)$ to the corresponding root $(1: \beta_i)$, where $\alpha_i$ and $\beta_i$ are the roots of the determinant polynomials of $\Phi_1$ and $\Phi_2$ respectively.
\end{theorem}

\subsubsection{Preliminaries in singularity theory}

We will give some brief preliminaries in singularity theory, which will be used implicitly throughout this paper.

\begin{definition}[{\cite[p.88]{Arnold1976}}]
A class of singularities $T_2$ is \emph{adjacent} to a class $T_1$, and one writes $T_1 \leftarrow T_2$ if every germ of $f \in T_2$ can be locally deformed into a germ in $T_1$ by an arbitrary small deformation. We say that the singularity $T_2$ is \emph{worse} than $T_1$, or that $T_2$ is a degeneration of $T_1$.
\end{definition}

The degenerations of the isolated singularities that appear in a complete intersection of two quadrics  in $\mathbb{P}^3$ and $\mathbb{P}^4$ are described in Figure \ref{fig:dynkin} (for details, see \cite[p.88]{Arnold1976} and \cite[\S 13]{Arnol_d_1975}). The above theory considers only local deformations of singularities. When we study degenerations in the GIT quotient, we are interested in global deformations. In the particular cases complete intersection of two quadrics  in $\mathbb{P}^3$ and $\mathbb{P}^4$, since by \cite{hodge_pedoe_1994,dolgachev_2012} the sum of the Milnor numbers of all ADE singularities satisfies $\sum_{i=1}^r\mu(T_r)\leq 7$, by \cite[Proposition 3.1]{hacking_prokhorov_2010} any local deformation of isolated singularities is induced by a global deformation. This will allow us to classify all stable, polystable and semistable orbits for each GIT and VGIT problem later on. 

\begin{figure}[h!]
    \centering
    \begin{tikzcd}
    \mathbf{A}_1 & \mathbf{A}_2\arrow[l] & \mathbf{A}_3\arrow[l]& \mathbf{A}_4\arrow[l]& \mathbf{A}_5\arrow[l] \\
    & &\mathbf{D}_4 \arrow[u] &\mathbf{D}_5 \arrow[l]\arrow[u]&   & 
    \end{tikzcd}
    \caption{Degeneration of germs of isolated singularities appearing in a complete intersection of two quadrics in $\mathbb{P}^3$ and $\mathbb{P}^4$}
    \label{fig:dynkin}
\end{figure}
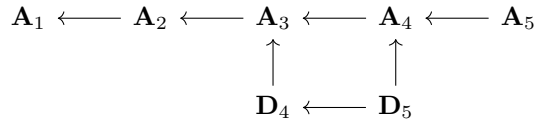

\subsection{General results}\label{general_results for P3}

Throughout this section we will make use of the classification of the base loci of such pencils found in Sommerville \cite[\S XIII]{sommerville}. We summarise these results in the following table.

\begin{table}[H]
\centering
\begin{tabular}{ c|c } 
 Segre symbol of pencil & Base locus  \\
 \hline
 $[4]$& Twisted cubic and tangent line\\
 $[(3,1)]$& A conic and two lines intersecting in one point\\
  $[(2,2)]$ & A double line and two lines in general position\\
 $[(2,1,1)]$&Two tangent lines\\
 $[3,1]$ & Cuspidal curve  \\
 $[(2,1),1]$ & Two tangent conics\\
 $[(1,1,1),1]$& Double conic\\
 $[2,2]$& Twisted cubic and bisecant\\
 $[2,(1,1)]$&A conic and two lines in a triangle\\
 $[(1,1),(1,1)]$&A quadrangle\\
 $[2,1,1]$& Nodal curve\\
 $[(1,1),1,1]$&Two conics in general position\\
 $[1,1,1,1]$ & Elliptic curve (smooth)  \\ 
\end{tabular}
\caption{Segre symbols of pencils of quadrics in $\mathbb{P}^3$ and classification of the base loci.}
    \label{tab:p3segtable}
\end{table}

%The lemma below gives a geometric classification of the above, following Mukai.  

The Segre symbols can be computed for any pencil by a simple Python script, that solves the determinant equation and finds the multiplicities of the solutions. In the case where the multiplicity is greater than $1$ one can compute the determinant minors to distinguish the different cases with similar multiplicity. In many cases, we will refer to a complete intersection $S$ having a particular Segre symbol (and not the pencil whose base locus $S$ is) to avoid confusion. Notice that by the above table and by the definition of the base locus description, if $S$ has Segre symbol $[2,1,1]$ it has $1$ $\mathbf{A}_1$ singularity, if it has Segre symbol $[2,2]$ it has $2$ $\mathbf{A}_1$ singularities while if it has Segre symbol $[(1,1),(1,1)]$ it has $4$ $\mathbf{A}_1$ singularities. If $S$ has Segre symbol $[3,1]$ it has $1$ $\mathbf{A}_2$ singularity, and if it has Segre symbol $[4]$ it has an $\mathbf{A}_3$ singularity. Finally, if $S$ has Segre symbol $[(2,1,1)]$ or $[(1,1,1),1]$ it has non-isolated singularities.

The Lemmas below give us a geometric description of the singularities of a pencil with respect to its determinant polynomial.

\begin{lemma}\label{Segre symbol [(2,1),1]}
Let $S$ be complete intersection with Segre symbol $[(2,1),1]$. Then $S$ has $1$ $\mathbf{A}_3$ singularity.
\end{lemma}
\begin{proof}
Since the Segre symbol of $S = \{f=0\}\cap \{g=0\}$ is  $[(2,1),1]$, $S$ is two conics tangent at a point $P$. We may assume, without loss of generality, that $f$ is smooth and $g$ is singular. Notice that $f = \mathbb{P}^1\times \mathbb{P}^1$ and each conic $C_1,$ $C_2$ is a $(1,1)$ divisor of $f$ as they are obtained via $H\cap f$, where $H$ is a hyperplane. The $C_i$ meet only at $P$, hence $C_1\cdot C_2 = 2$. We blow up $f$ at a point $Q$ such that $Q\notin C_i$, and $Q$ lies on the intersection of two lines $G_1,G_2$. We obtain $\pi \colon \operatorname{Bl}_Q f\rightarrow f$ and we then blow down twice at the proper transforms of $G_i$, $G'_i$ to points $Q_1$, $Q_2$, to obtain $\epsilon\colon \operatorname{Bl}_{Q_1,Q_2}\mathbb{P}^2 \rightarrow \mathbb{P}^2$. This allows us to obtain the usual birational map $f \xrightarrow{\phi} \mathbb{P}^2$. We have $Y\vcentcolon= \operatorname{Bl}_Q f = \operatorname{Bl}_{Q_1,Q_2}\mathbb{P}^2$. The ramification formula reads 

$$K_Y = \pi^*K_f +E = \epsilon^*K_{\mathbb{P}^2} +G'_1+G'_2$$
where the $E$, $G'_i$ are the corresponding exceptional divisors.

Notice that for the anticanonical divisor, $-K_f = -K_{\mathbb{P}^3}+f|_f = \mathcal{O}_f (2)$, by the adjunction formula, which in turn implies that $-K_f\cdot C_i = 4$. Notice also that for the proper transforms of $\pi$ of the $C_i$, denoted by $C'_i$ we have $C'_i =\pi^*(C_i) $, $G'_i = \pi^*(G_i)-E$, and for the proper transforms of the $C'_i$, denoted by $\overline{C}_i$, of $\epsilon$ we have $\epsilon^*(\overline{C}_i) = C'_i+G'_1+G'_2$, and $\overline{C_1}\cdot\overline{C_2} = 4$. Hence,

\begin{equation*}
    \begin{split}
        -K_{\mathbb{P}^2}\cdot \overline{C}_i =& \epsilon^*(-K_{\mathbb{P}^2})\cdot \epsilon^* (\overline{C}_i)\\
         =& -K_Y\cdot C'_i+2\\
         =&\pi^*(K_f)\cdot\pi^*(C_i)+2\\
         =& -K_f\cdot C_i+2\\
         =& 6
    \end{split}
\end{equation*}

This implies that the proper transforms $\overline{C}_i$ are also conics which intersect tangentially at $P$ and normally at $Q_1$, $Q_2$. We take $f,g $ conics in $\mathbb{P}^2$ such that $q\cap r = 2p+r+s$. Two such conics can be given by $q = x_0lx_1+x^2_1+x^2_2$, $r = x_0x_2+x^2_1+x^2_2$  with double point of intersection at $p = (1:0:0)$. Localising at $p$, $f|_{\operatorname{loc}} = x_1+x^2_1+x^2_2$, $g|_{\operatorname{loc}} = x_2+x^2_1+x^2_2$, and hence $p$ is an $\mathbf{A}_3$ singularity. Since the morphism $\phi$ is a locally analytic homomorphism around $p$, the singularity is the same in $P$ in $f$. 

\end{proof}

The proofs of the following Lemma uses the same techniques as in the proof of Lemma \ref{Segre symbol [(2,1),1]} and are thus omitted.

\begin{lemma}\label{Sings of all remaining varieties P3}
We have the following:
\begin{enumerate}
    \item Let $S$ be complete intersection with Segre symbol $[(3,1)]$. Then $S$ has $1$ $\mathbf{D}_4$ singularity.
    \item Let $S$ be complete intersection with Segre symbol $[(2,2)]$. Then $S$ has non-isolated singularities.
    \item Let $S$ be complete intersection with Segre symbol $[(1,1),1,1]$. Then $S$ has $2$ $\mathbf{A}_1$ singularities.
    \item Let $S$ be complete intersection with Segre symbol $[2,(1,1)]$. Then $S$ has $2$ $\mathbf{A}_1$ singularities.
\end{enumerate}
\end{lemma}

\begin{remark}\label{magma code to check sings}
In order to check the type of (isolated) hypersurface singularities, one can employ the following MAGMA script, adjusted accordingly for each different case.

\begin{verbatim}
Q:=RationalField();
PP<x0,x1,x2,x3>:=ProjectiveSpace(Q,3);
f1:=x0*x1;
f2:=x2*x3;
X:=Scheme(PP,[f1,f2]);
IsNonsingular(X);
p := X![1,0,0,0];
_,f,_,fdat := IsHypersurfaceSingularity(p,2);
R<a,b> := Parent(f);
f;
NormalFormOfHypersurfaceSingularity(f);
boo,f0,typ := NormalFormOfHypersurfaceSingularity(f : fData := [*fdat,3*]);
boo; f0; typ;
\end{verbatim}
Here, $f1$ and $f2$ are the generating polynomials, and \begin{verbatim}
    IsNonsingular(X);
\end{verbatim} 
verifies that $X = f_1\cap f_2$ is singular. The point $p = [1,0,0,0]$ refers to a specific singular point, whose type of singularity we want to check, which is given by the last command. Note, that this code also works for higher dimensional complete intersections.
\end{remark}

\begin{lemma}\label{mukai_P3}
Let $\Phi(f,g)$ be a pencil of two quadrics $f,g \in \mathbb{P}^3$, where we assume that $f$ is smooth. The base locus of the pencil has only $\mathbf{A}_1$ singularities if and only if the determinant polynomial $\det(\lambda f+g)$ has roots of multiplicity $2$.
%\begin{enumerate}
%    \item The pencil $\Phi(q,r)$ is stable if and only if its base locus $\operatorname{Bs}(q,r)$ is smooth;
%    \item The pencil $\Phi(q,r)$ is semi-stable if and only if its base locus $\operatorname{Bs}(q,r)$ has at worst \mathbf{A}_1 singularities;
%    \item The pencil $\Phi(q,r)$ is unstable if and only if its base locus $\operatorname{Bs}(q,r)$ has singularities worse than \mathbf{A}_1.
%    \item The pencil $\Phi(q,r)$ is polystable if and only if $q$ and $r$ are simultaneously diagonalizable, i.e. $q$, $r$ caan be written as:
%    \begin{equation*}
%        \begin{split}
%            q(x) =& x_0^2+x_1^2+x_2^2+x_3^2\\
%            r(x) =& \lambda_0 x_0^2+\lambda_1 x_1^2+\lambda_2x_2^2+\lambda_3x_3^2
%        \end{split}
%    \end{equation*}
%    where not all $2$ of the $\lambda_i$s are equal.
%\end{enumerate}
\end{lemma}
\begin{proof}
%$1$ follows from Avritzer, and the above discussion.

% We assume without loss of generality that $\det(\lambda f +g)$ has a root of multiplicity $2$ at $\lambda = 0$. Then $\operatorname{det}(g) = 0$, and $g$ is singular. We may thus write $g$ as $g(x) = x_0^2+ \dots x_k^2$ where $k \leq 2$. Notice that $\det(\lambda f +g)$ is divisible by $\lambda$, and hence we either have $k<2$ (i.e. $k=1$) or $k=2$, $f_{33}=0$ and $f_{03}^2+f_{13}^2+f_{23}^2 \neq 0$. 

% If $k=1$, then $f(x)$ has two simple zeroes on the line $x_0 = x_1 = 0$ in $\mathbb{P}^3$, and hence we can find a coordinate system such that $f(0,0,1,0) = f(0,0,0,1)=0$, which in turn implies that we can write 
% $$f(x) = q_1(x_0,x_1)+l_1(x_0,x_1)x_2+l_2(x_0,x_1)x_3+x_2x_3,$$
% where $q_1$ is a quadratic form and the $l_i$ are linear forms. Using a suitable change of basis we may further write $f(x)$ as $q_2(x_0,x_1)+x_2x_3$, and thus the base locus $\operatorname{Bs}(f,q)$ carries a $\mathbb{G}_m$ action fixing two $\mathbf{A}_1$ singularities $(0:0:1:0)$ and $(0:0:0:1).$

% Analysing the case where $k=2$, using a suitable coordinate change, we can write $f(x)$ as $f(x) = q_3(x_0,x_1,x_2)+x_2x_3$. The Segre Symbol for this pencil is $[2,1,1]$ and by the classification of pencils of quadrics in $\mathbb{P}^3$ the base locus is a nodal curve, with $\mathbf{A}_1$ singularity at $(0:0:0:1)$. 

%For the converse, 
Assume that the base locus $S$ has only $\mathbf{A}_1$ singularities. Then, by Table \ref{tab:p3segtable}, Theorem \ref{segre thm} and Lemmas \ref{Segre symbol [(2,1),1]} and \ref{Sings of all remaining varieties P3} that can only occur if $S$ is either
\begin{enumerate}
    \item a nodal curve, with Segre symbol $[2,1,1]$, with one $\mathbf{A}_1$ singularity,
    \item or a twisted cubic and a bisecant, with two $\mathbf{A}_1$ singularities and Segre symbol $[2,2]$,
    \item or two conics in general position with Segre symbol $[(1,1),1,1]$, with two $\mathbf{A}_1$ singularities, 
    \item or a conic and two lines in a triangle, with Segre symbol $[2,(1,1)]$, with two $\mathbf{A}_1$ singularities,
    \item or a quadrangle, with Segre symbol $[(1,1),(1,1)]$ with four $\mathbf{A}_1$ singularities.
\end{enumerate}
Notice that all of the above cases have determinant polynomial with a root of multiplicity $2$.
\end{proof}

%\begin{remark}
%From Table \ref{tab:p3segtable} we see that singularities worse than $\mathbf{A}_1$ occur when $\det(\lambda f+g)$ has roots of multiplicity greater than $2$. A similar argument as in Lemma \ref{mukai_P3} can be used to show, up to projective equivalence, which pencils correspond to these, emulating \cite{mabuchi_mukai_1990}.
%\end{remark}
Given the analysis in this section, we present the following table, which lists the possible complete intersections of two quadrics in $\mathbb{P}^3$, up to projective equivalence, and their singularities. The list of polynomials can be also found in \cite[\S XIII]{sommerville} (up to projective equivalence). One can also generate the polynomials $f$, $g$ via the Segre symbols, using the computational package \cite{pap_segre}.

% \newpage

%\newpage
%\begin{table}[H]
%\centering
%\begin{tabular}{ c|c|c } 
%\begin{longtable}{| p{.16\textwidth} | p{.50\textwidth} |p{.13\textwidth} |} 
\begin{center}
\begin{longtable}{| p{2.4cm} | p{8.1cm} |p{2cm} |}
\hline
 Segre symbol & Generating polynomials  & Singularities \\
 \hline
 $[4]$& \begin{center}
     $\begin{aligned}[c]
 f&= q_1(x_2,x_3)+x_3l_1(x_0,x_1)+x_2l_2(x_0,x_1)\\
g&= q_2(x_2,x_3)+x_3l_3(x_0,x_1)+x_2l_4(x_0,x_1)
 \end{aligned}$
  \end{center} & $\mathbf{A}_3$\\
 \hline
 $[(3,1)]$&  \begin{center}
     $\begin{aligned}[c]
         f&= q_1(x_1,x_2,x_3)+x_0x_3\\
         g&= x_3l_1(x_1,x_2,x_3)
     \end{aligned}$
  \end{center}& $\mathbf{D}_4$\\
 \hline
  $[(2,2)]$ &\begin{center}
     $\begin{aligned}[c]
         f&= q_1(x_2,x_3)+x_0x_3+ x_1l_1(x_2,x_3)\\
         g&= q_2(x_2,x_3)
     \end{aligned}$
  \end{center} & non-isolated \\
 \hline
 $[3,1]$ & \begin{center}
     $\begin{aligned}[c]
         f&= q_1(x_1,x_2,x_3)+x_0x_3\\
         g&= q_2(x_2,x_3)+x_1x_3
     \end{aligned}$
  \end{center}  & $\mathbf{A}_2$ \\
 \hline
 $[(2,1),1]$ & \begin{center}
     $\begin{aligned}[c]
         f&= q_1(x_1,x_2,x_3)+x_0l_1(x_1,x_2,x_3)+ x_1l_1(x_2,x_3)\\
         g&= q_2(x_2,x_3)+x_0x_3+ x_1l_2(x_2,x_3)
     \end{aligned}$
  \end{center} & $\mathbf{A}_3$ \\
 \hline
 $[(1,1,1),1]$& \begin{center}
     $\begin{aligned}[c]
         f&= q_1(x_0,x_1,x_2,x_3)\\
         g&= x_3^2
     \end{aligned}$
  \end{center} & non-isolated\\
 \hline
 $[2,2]$&
  \begin{center}
     $\begin{aligned}[c]
         f&= q_1(x_2,x_3)+x_0l_1(x_2,x_3)+ x_1l_2(x_2,x_3)\\
         g&= q_2(x_2,x_3)+x_0l_3(x_2,x_3)+ x_1l_4(x_2,x_3)
    \end{aligned}$
  \end{center} & $2\mathbf{A}_1$ \\
 \hline
 $[2,(1,1)]$& \begin{center}
     $\begin{aligned}[c]
         f&= q_1(x_2,x_3)+x_0x_3\\
         g&= x_3l_3(x_0,x_1,x_2,x_3)
    \end{aligned}$
  \end{center} & $2\mathbf{A}_1$ \\
 \hline
 $[(1,1),(1,1)]$& \begin{center}
     $\begin{aligned}[c]
         f&= q_1(x_2,x_3)\\
         g&= q_2(x_0,x_1)
     \end{aligned}$
  \end{center} & $4\mathbf{A}_1$ \\
 \hline
 $[2,1,1]$&  \begin{center}
     $\begin{aligned}[c]
         f&= q_1(x_1,x_2,x_3)+x_0x_3\\
         g&= q_2(x_1,x_2,x_3)
     \end{aligned}$
  \end{center} & $\mathbf{A}_1$ \\
 \hline
 $[(1,1),1,1]$& \begin{center}
     $\begin{aligned}[c]
         f&= q_1(x_0,x_1,x_2,x_3)\\
         g&= q_2(x_2,x_3)
     \end{aligned}$
  \end{center} & $2\mathbf{A}_1$ \\
 \hline
 $[1,1,1,1]$ &  \begin{center}
     $\begin{aligned}[c]
         f&= q_1(x_0,x_1,x_2,x_3)\\
         g&= q_2(x_0,x_1,x_2,x_3)
     \end{aligned}$
  \end{center} & None\\ 
 \hline
%\end{tabular}
\caption{Segre symbols, equations up to projective equivalence, and singularities of complete intersections of quadrics in $\mathbb{P}^3$.}
    \label{tab:p3segtable_with_equations}
%\end{table}
\end{longtable}
\end{center}

\begin{remark}
Note, that for all singular complete intersections, by Theorem \ref{segre thm}, and since any two sets of $3$ points in $\mathbb{P}^1$ are isomorphic, the above description is unique, since any two singular complete intersections of quadrics in $\mathbb{P}^3$ with the same Segre symbol are progressively equivalent.
\end{remark}

\subsection{GIT classification}\label{sec: GIT classification P3}
 In this section we will study the GIT quotient $\mathcal{R}_{3,2,2}\sslash \operatorname{SL}(4)$. The following families have been generated using the computational package \cite{theodoros_stylianos_papazachariou_2022}, based on the discussion in Section \ref{VGITsection}. In particular, they are maximal (semi-)destabilising families, in the sense of Definition \ref{destabilised sets def-vgit}. 

\begin{theorem}\label{p3 theorem 1- nonstable- nominus}
The following are equivalent:
\begin{enumerate}
    \item A pencil of two quadrics $\Phi(f,g)$ in $\mathbb{P}^3$ is non-stable;
    \item the base locus of the pencil $\operatorname{Bs}(f,g)$, is singular;
    \item the pencil is generated by one of the following families, or a degeneration:
    
Family $1$:
\begin{equation*}
    \begin{split}
        f(x_0,x_1,x_2,x_3) =& x_3 l_1(x_0,x_1,x_2,x_3) + x_2l_2(x_0,x_1,x_2) +x_1l_3(x_0,x_1)\\
        g(x_0,x_1,x_2,x_3) =& q_1(x_1,x_2,x_3)
    \end{split}
\end{equation*}
an irreducible smooth quadric $f$ and an irreducible singular quadric $g$ intersecting at a nodal curve, with $\mathbf{A}_1$ singularity at $(1:0:0,0)$; 

Family $2$:
\begin{equation*}
    \begin{split}
        f(x_0,x_1,x_2,x_3) =& q_2(x_0,x_1,x_2,x_3)\\
        g(x_0,x_1,x_2,x_3) =& q_3(x_2,x_3)
    \end{split}
\end{equation*}
an irreducible smooth quadric $f$ and a pair of intersecting planes $g$ such that  $\operatorname{Bs}(f,g)$ is a pair of conics in general position, with $\mathbf{A}_1$ singularities (up to a change of coordinates) at $(1:i:0:0)$, $(1:-i:0:0)$;

Family $3$:
\begin{equation*}
    \begin{split}
        f(x_0,x_1,x_2,x_3) =& x_3 l_4(x_0,x_1,x_2,x_3)+ x_2l_5(x_0,x_1,x_2)\\
        g(x_0,x_1,x_2,x_3) =&  x_3 l_6(x_0,x_1,x_2,x_3)+ x_2l_7(x_0,x_1,x_2)
    \end{split}
\end{equation*}
two irreducible non-singular quadrics $f,g$ intersecting at a twisted cubic and a bisecant, with $\mathbf{A}_1$-singularities at $(1:0:0:0)$ and $(0:1:0:0)$; 

Family $4$:
\begin{equation*}
    \begin{split}
        f(x_0,x_1,x_2,x_3) =& x_3 l_8(x_0,x_1,x_2,x_3)+ q_4(x_1,x_2)\\
        g(x_0,x_1,x_2,x_3) =& x_3 l_{9}(x_0,x_1,x_2,x_3)+ q_5(x_1,x_2)
    \end{split}
\end{equation*}
two irreducible smooth quadrics $f, g$ intersecting at a nodal curve, with an $\mathbf{A}_1$-singularity at $(1:0:0:0)$,

or a degeneration of the above. Here, the $l_i$ are linear forms  in $\mathbb{P}^3$ and the $q_i$ are quadratic forms. 

In particular, the above $4$ families are strictly semistable, and families $1$ and $4$ are projectively equivalent.
\end{enumerate}
\end{theorem}
\begin{proof}
The equivalence of $1$ and $3$ follows from the computational program \cite{theodoros_stylianos_papazachariou_2022} we detailed in Section \ref{VGITsection} and the centroid criterion (Theorem \ref{t-centroid_criterion}), where the above families are maximal destabilising families as in the sense of Definition \ref{destabilised sets def-vgit}, i.e. each family equals $N^{-}(\lambda, x^J, x_p)$ for some $\lambda \in P_{3,2,2}$ and $x^J$, $x_p$ support monomials.%, from Table \ref{tab:code outputs p3}. 
The description of the above families with respect to singularities follows from Section \ref{general_results for P3} and Table \ref{tab:p3segtable_with_equations}.
% From the Centroid criterion (Lemma \ref{t-centroid_criterion}) the above cases are non-stable, and they are semistable, hence they are strictly semistable. Furthermore, from lemma \ref{singular intersection in pn} the above $4$ pencils have singular base locus.

For the equivalence of $2$ and $3$ we have the following analysis.  Families $1$, $2$, $3$ and $4$ have Segre symbols $[2,1,1]$, $[(1,1),1,1]$, $[2,2]$ and $[2,1,1]$ respectively.

% For the first Family, notice that the base locus is singular, with singularity $P=(1:0:0:0)$. Furthermore, the Segre symbol of the pencil is $[2,1,1]$, and by the classification, the base locus is a nodal curve, hence $P$ is an $\mathbf{A}_1$ singularity. 

% For Family $2$, notice that the Segre symbol of the pencil is $[(1,1),1,1]$ and hence that the base locus is two conics $C_1$, $C_2$ in general position, intersecting at points $P, Q$,  with $\mathbf{A}_1$ singularities by Table \ref{tab:p3segtable}.

% For Family $3$, the Segre symbol of the pencil is $[2,2]$. By the classification, the base locus of the pencil is a twisted cubic and a bisecant intersecting the cubic at $P = (1:0:0:0)$ and $Q=(0:1:0:0))$ which are $\mathbf{A}_1$ singularities.

% For Family $4$, the base locus is singular, with singularity $P=(1:0:0:0)$. Furthermore, the Segre symbol of the pencil is $[2,1,1]$, and by the classification, the base locus is a nodal curve, hence $P$ is an $\mathbf{A}_1$ singularity. In particular, since families $1$ and $4$ have the same Segre symbol, they are projectively equivalent by Theorem \ref{segre thm}.
In addition, notice that a degeneration of family $3$ is 
\begin{equation*}
    \begin{split}
        f(x_0,x_1,x_2,x_3) &= x_0x_3\\
        g(x_0,x_1,x_2,x_3) &= x_1x_2
    \end{split}
\end{equation*}
which is the quadrangle with Segre symbol $[(1,1), (1,1)]$ and $4$ $\mathbf{A}_1$ singularities.

Also, notice that a degeneration of family $4$ is 
\begin{equation*}
    \begin{split}
        f(x_0,x_1,x_2,x_3) &= q(x_2,x_3)+x_0x_3\\
        g(x_0,x_1,x_2,x_3) &= x_3l(x_0,x_1,x_2,x_3)
    \end{split}
\end{equation*}
which is the complete intersection with Segre symbol $[2, (1,1)]$ and $2$ $\mathbf{A}_1$ singularities.

Hence, from the above discussion, and from Lemma \ref{singular intersection in pn} we notice that if a pencil has singular base locus, then it belongs to one of the above Families (or its degenerations). %In particular, we can see the above for Family $3$, by making the change of coordinates $x_0 = l_6$, $x_1 = l_7$. This shows that a pencil is singular if and only if belongs to one of these equations, or its degenerations. Further, since both 
% \begin{equation*}
%     \begin{split}
%         f(x_0,x_1, x_2, x_3,) &= x_0l(x_1, x_2, x_3) +q_1(x_1, x_2, x_3)\\
%         q(x_0, x_1, x_2, x_3,) &= q_2(x_1, x_2, x_3)
%     \end{split}
% \end{equation*}
% and
% \begin{equation*}
%     \begin{split}
%         f(x_0, x_1, x_2, x_3,) &= x_0x_3 +q_1(x_1, x_2, x_3)\\
%         q(x_0, x_1, x_2, x_3,) &= x_0x_3 +q_2(x_1, x_2, x_3)
%     \end{split}
% \end{equation*}
% have Segre symbol $[2,1,1]$ and are projectively equivalent to each other by Theorem \ref{segre thm}, and correspond to Family $4$.

%Similar to the proof of Theorem \ref{p3 theorem 1- unstable- n-}, for some $f,g$, $\Delta(\lambda, \mu)$ has double roots, for others $\Delta(\lambda, \mu)$ has triple roots while for others $\Delta(\lambda, \mu)$ has roots of multiplicity $4$. These correspond to all the possible configurations of Segre symbols for roots of multiplicity $3$ and $4$. By a theorem of Weirestrass and Segre, all pencils with the same Segre symbols have projectively equivalent base loci. Hence, the base locus of a pencil $\Phi(f',g')$, $\operatorname{Bs}(f',g')$ with roots of  $\operatorname{det}(\lambda f'+\mu g')$ of multiplicity $3$ will be projectively equivalent to the base locus of one of the Cases $1$, $3$ or $5$, and similarly the base locus of a pencil $\Phi(f',g')$, $\operatorname{Bs}(f',g')$ with roots of  $\operatorname{det}(\lambda f'+\mu g')$ of multiplicity $4$ will be projectively equivalent to the base locus of one of the Cases $2$, $4$, $6$ or $7$. Applying Lemma \ref{projective equivalence of pencils} completes the proof. (Needs some work!)
\end{proof}

\begin{corollary}\label{p3_stable}
A pencil $\Phi(f,g)$ of quadrics in $\mathbb{P}^3$ is stable if and only if it is smooth. 
\end{corollary}
\begin{proof}
Let $\Phi$ be a pencil which has smooth base locus. By  Theorem \ref{p3 theorem 1- nonstable- nominus} it is not non-stable, i.e. it is stable.
\end{proof}

\begin{theorem}\label{polystable_P3}
    \item A pencil of two quadrics $\Phi(f,g)$ in $\mathbb{P}^3$ is strictly polystable if and only if it is generated by:

\begin{equation*}
    \begin{split}
        f(x_0,x_1,x_2,x_3) =& q_2(x_0,x_1)\\
        g(x_0,x_1,x_2,x_3) =&   q_3(x_2,x_3)
    \end{split}
\end{equation*}
two reducible singular quadrics $f,g$ intersecting at a quadrangle, with $\mathbf{A}_1$-singularities at (up to change of basis) $(1:0:0:0)$, $(0:1:0:0)$, $(0:0:1:0)$, $(0:0:0:1)$,

% or equivalently

% \begin{equation*}
%     \begin{split}
%         f(x_0,x_1,x_2,x_3) =& x_2 l_4(x_0,x_1)+ x_3l_5(x_0,x_1)\\
%         g(x_0,x_1,x_2,x_3) =& x_2 l_{6}(x_0,x_1)+ x_3l_7(x_0,x_1)
%     \end{split}
% \end{equation*}
% two irreducible smooth quadrics $f, g$ intersecting at a quadrangle, with an $4$ $\mathbf{A}_1$-singularities at $(1:0:0:0)$, $(0:1:0:0)$, $(0:0:1:0)$, $(0:0:0:1)$. 

Here, the $l_i$ are linear forms in $\mathbb{P}^3$ and the $q_i$ are quadratic forms, which are maximal in their support.
\end{theorem}

\begin{proof}
First notice that by Lemma \ref{singular intersection in pn} the above family is singular, and by Theorem \ref{t-centroid_criterion} it is strictly semistable, i.e. in particular non-stable. The above family has determinant polynomial with roots of multiplicity $2$, hence by Lemma \ref{mukai_P3} it has $\mathbf{A}_1$ singularities.

In more detail, the above family has 
%a quadric $f'$ of the pencil $\Phi(f,g)$ is non-singular, hence we can derive the Segre symbols for these pencils. Using $f' = \lambda f+\mu g$, both of these families have 
Segre symbol $[(1,1),(1,1)]$ and hence its base locus is a quadrangle, with $4$ $\mathbf{A}_1$ singularities. %Hence, we have two divisors $C_1 = 2L$, $C_2 = 2G$, $C_1\cdot C_2 = 4$, $C_1 \cap C_2 = \{P,Q,R,S\}$. By blowing up, and blowing down twice to $\mathbb{P}^2$ as before, the strict transforms $\overline{C}_1$, $\overline{C}_2$ are each a pair of two lines, intersecting at the four points $\{P,Q,R,S\}$. More specifically, each of the four points of intersection are $\mathbf{A}_1$ singularities. 

% Notice that we can write these families (up to change of coordinates) as $f = x_0x_1$, $g = x_2x_3$.%, or $f = x_2x_0$, $g=x_1x_3$. 
% The reason for this, especially for the second family, is that $f = l'_1(x_2,x_3)l'_2(x_0,x_1)$, $g = l'_3(x_2,x_3)l'_4(x_0,x_1)$ so the transformation $x_i\rightarrow l'_{i+1}$ gives $f = x_0x_1$, $g = x_2x_3$

Note that $f\wedge g= x_0x_1\wedge x_2x_3$; %$f\wedge g= x_0x_2\wedge x_1x_3$;
for any one parameter subgroup $\lambda \colon \mathbb{G}_m \rightarrow \operatorname{SL}(4)$, with $\lambda(s) = \operatorname{Diag}(s^{a_0}, \dots, s^{a_3})$, with $\sum a_i =0$, we have 
$$\lim_{s \to 0}\lambda(s)\cdot f\wedge g = s^{\sum a_i}x_0x_1\wedge x_2x_3 = x_0x_1\wedge x_2x_3 = f\wedge g.$$ % or $\lim_{s \to 0}\lambda(s)\cdot f\wedge g = s^{-\sum a_i}x_0x_2\wedge x_1x_3 = x_0x_2\wedge x_1x_3= f\wedge g$. 
Hence, $\dim \operatorname{Stab}(f\wedge g) = \dim(\operatorname{SL}(4))$ is maximal and thus the orbit of $f \wedge g$ is closed, i.e. the pencil is polystable.

By Theorem \ref{strictly_semistable_k-case-vgit} a complete intersection $S$, defined by $S = \{f=g=0\}$, that belongs to a closed strictly semistable orbit is generated by monomials in the set $N^0(\lambda, x^{J_1})$, for some $(\lambda, x^{J_1})$. The above family corresponds to the only such $N^0(\lambda, x^{J_1})$ (up to projective equivalence). In particular, these are obtained by verifying which $N^{-}(\lambda, x^J)$ give strictly semistable families,  for various support monomials $x^J$, and then computing $N^0(\lambda, x^{J_1})$ by the description in Lemma \ref{annihilator_k-case-vgit}.
\end{proof}

%\begin{document}

\section{K-moduli compactification of family 2.25}\label{fano 3fold section}
Consider a smooth intersection of two quadrics $C_1$ and $C_2$ in $\mathbb{P}^3$. The resulting complete intersection $C = C_1\cap C_2$ is an elliptic curve; blowing up $\mathbb{P}^3$ along $C$ gives a smooth Fano threefold $X = \operatorname{Bl}_C \mathbb{P}^3$, with $(-K_X)^3 = 32$. It is known (see, e.g. \cite[Corollary 4.3.16.]{fano_book}), that all such smooth Fano threefolds which correspond to family $2.25$ in the Mori-Mukai classification \cite{mori_mukai_2003}, are K-stable.

%Let $X_{\infty}$ be the Gromov-Hausdorff limit of smooth K{\"a}hler-Einstein Fano threefolds in family $2-25$ as described above. There exists a Hitchin-Kobayashi homeomorphism $ \overline{M}^{GH}_{\operatorname{Bl}_{C_1\cap C_2} \mathbb{P}^3} = M^{KE}\cup X_{\infty}\rightarrow\overline{M}^{KE}_{\operatorname{Bl}_{C_1\cap C_2} \mathbb{P}^3}$. 

Let $C_1 = \{ x_0x_1 = 0\}$, $C_2 = \{ x_2x_3 = 0\}$ be two quadrics in $\mathbb{P}^3$. Then $\tilde{C} = C_1\cap C_2$ is GIT-polystable by Theorem \ref{polystable_P3}. Notice, that $C_1$ and $C_2$ are toric surfaces which intersect on a toric curve, which is the $4$ lines $\{x_0 = 0\}$, $\{x_1 = 0\}$, $\{x_2 = 0\}$ and $\{x_3 = 0\}$, hence $\tilde{C}$ is toric. As such, as $\mathbb{P}^3$ is toric, the blow up of $\mathbb{P}^3$ along $\tilde{C}$, $\tilde{X} = \operatorname{Bl}_{\tilde{C}} \mathbb{P}^3$ is a toric blow up and hence $\tilde{X}$ is a toric variety. The polytope of $\tilde{X}$, $P_{\tilde{X}}$, is created by `cutting' the corresponding polytope for $\mathbb{P}^3$, $P_{\mathbb{P}^3}$, along the $4$ edges corresponding to each $x_i$. The corresponding polytope $P_{\tilde{X}}$, is a polytope generated by vertices $(1,0,0)$, $(0,1,0)$, $(0,0,1)$, $(1,1,0)$, $(0,1,1)$, $(-1,-1,0)$, $(0,-1,-1)$, $(-1,-1,-1)$. This is the terminal toric Fano threefold with Reflexive ID %$\#255743$ 
$\#199$ in the Graded Ring Database (GRDB) ($3$-fold $\#255743$) \cite{Kas08}. Notice that the sum of the vertices is $(0,0,0)$, and hence, the barycenter of $P_X$ is $(0,0,0)$. Hence, by Theorem \ref{toric batyrev}, $\tilde{X}$ is K-polystable. In particular, we have proved:
\begin{lemma}\label{gh_limit}
Let $\tilde{X} \vcentcolon= \operatorname{Bl}_{\tilde{C}} \mathbb{P}^3$ where $\tilde{C} = C_1\cap C_2 $ for $C_1 = \{ x_0x_1 = 0\}$, $C_2 = \{ x_2x_3 = 0\}$. Then $\tilde{X}$ is K-polystable.%/K{\"a}hler-Einstein.
\end{lemma}

The Magma code below checks that this threefold is toric and singular, and generates the vertices of the polyhedron $P_X$.
\begin{verbatim}
Q4:=Polytope([[0,-1,0],[-1,0,1],[2,-1,0],[1,0,-1],[-1,0,-1],[0,-1,2],
[0,1,0],[-1,2,-1]]);
> ViewWithJmol(Q4: point_labels:=true, open_in_background:=true);   
> P:=Dual(Q4);
> P;
> IsCanonical(P);
> Volume(Q4);
> #Points(Q4);
> #Vertices(Q4);
> Faces(P);
> IsTerminal(P);
> IsSmooth(P);
\end{verbatim}

\begin{lemma}\label{strictly k-ss for 2.25}
Let $X \vcentcolon= \operatorname{Bl}_C \mathbb{P}^3$, where $C$ is a strictly GIT semistable complete intersection of two quadrics. Then $X$ is strictly K-semistable.
\end{lemma}
\begin{proof}

% Just take P to be P^3\times B where B=P^1 (more on this below). Take C to be the curve C in P^3 at the fibre t=1 in P. Apply lambda(t) for t\in B within P and let curly C be the closure of \lambda(t)\cdot C in P. Now blow-up Curly C in P. This is your curly X and by definition you have a morphism to B by taking the composition of X-> P given by the blow-up and P-> B given by the projection. This is naturally a test configuration of X with central fibre X_0, etc.

Let $C$ be a strictly GIT semistable complete intersection of two quadrics. Since $C$ is strictly GIT semistable, there exists a one-parameter subgroup $\lambda$ such that the limit $\lim_{t\to 0} \lambda(t)\cdot C = \tilde{C}$, where $\tilde{C}$ is the unique strictly GIT polystable complete intersection i.e. the quadrangle with $4$ $\mathbf{A}_1$ singularities as in Theorem \ref{polystable_P3}. This one-parameter subgroup induces a family $f\colon \mathcal{C}\rightarrow B$, over a curve $B$, such that the fibers $\mathcal{C}_t$ are isomorphic to $C$ for all $t\neq 0$, and $\mathcal{C}_0\cong \tilde{C}$.

Let $\mathcal{P} \coloneqq \mathbb{P}^3\times B = \mathbb{P}^3\times \mathbb{P}^1$. Then, notice that 
$\tilde{C} = \overline{\lambda(t)\cdot C}$ in $\mathcal{P}$. We define $ \mathcal{X} = \operatorname{Bl}_C \mathcal{P}$, and hence we have that $\mathcal{X}_0 \cong \tilde{X} \vcentcolon= \operatorname{Bl}_{\tilde{C}} \mathcal{P}$. By taking the composition of $\mathcal{X}\rightarrow \mathcal{P}$, with the projection $\mathcal{P}\rightarrow B$, we thus have a map $\mathcal{X}\rightarrow B$ which is naturally a test configuration of $\mathcal{X}$ with central fibre $\mathcal{X}_0$.
% We will construct a family $g\colon \mathcal{X} \rightarrow B$ as follows. Let $\mathcal{X}_t \cong X = \operatorname{Bl}_C \mathbb{P}^3$ for all $t\neq 0$ and $\mathcal{X}_0 \cong \tilde{C} \vcentcolon= \operatorname{Bl}_{\tilde{C}} \mathbb{P}^3$. Then we have a diagram 
% \begin{center}
%   \begin{tikzcd}
%   \mathcal{C}\arrow[r, "f"]\arrow[d] & B\\
%   \mathcal{X} \arrow[ur, "g"]& 
%   \end{tikzcd}
% \end{center}
% which implies that $g$ is also a flat and proper morphism, as $f$ is a family. 
Hence, we have constructed a test configuration $g\colon \mathcal{X} \rightarrow B$ where the central fiber $\mathcal{X}_0\cong \tilde{X}$ is a klt Fano threefold, which is K-polystable by Lemma \ref{gh_limit}, and the general fiber $\mathcal{X}_t\cong X$ is not isomorphic to $\mathcal{X}_0$. By \cite[Corollary 1.1.14]{fano_book} the central fiber $\mathcal{X}_t \cong X$ is strictly K-semistable. 
\end{proof}

\begin{remark}
In some cases, we can construct explicit descriptions of the above degeneration. For example, for Family $2$ in Theorem \ref{p3 theorem 1- nonstable- nominus} we can make a change of coordinates such that $C = \{f =0\}\cap \{g =0\}$ is given by
\begin{equation*}
  \begin{split}
    f(x_0,x_1,x_2,x_2) &= x_0x_1 + q(x_2,x_3)+l(x_0,x_1)\tilde{l}(x_2,x_3)\\
    g(x_0,x_1,x_2,x_2) &= x_2x_3.
  \end{split}
\end{equation*}
Then, defining for some parameter $t$, $C_t$ as follows $C_t = \{f_t=0\}\cap\{ g_t=0\}$, where
\begin{equation*}
  \begin{split}
    f_t(x_0,x_1,x_2,x_2) &= x_0x_1 + t\big(q(x_2,x_3)+l(x_0,x_1)\tilde{l}(x_2,x_3)\big)\\
    g_t(x_0,x_1,x_2,x_2) &= x_2x_3.
  \end{split}
\end{equation*}
we have $C_t \cong C$ for all $t\neq0$, and $C_0$ to be the strictly GIT polystable curve of Theorem \ref{polystable_P3}.

\end{remark}

In essence, we have shown:

\begin{corollary}\label{GIT implies K-stability p3}
Let $C = C_1\cap C_2$ be a complete intersection of quadrics. If $C$ is GIT (poly/semi-)stable then the threefold $X\coloneqq \operatorname{Bl}_C \mathbb{P}^3$ is K-(poly/semi-)stable.
\end{corollary}

Let $\mathcal{M}^K_{2.25}$ be the connected component of the K-moduli stack $\mathcal{M}^K_{3,32}$ parametrising K-semistable members in the Fano threefold family $2.25$ and let $\mathcal{M}^{GIT}_{3,2,2}$ be the GIT quotient stack parametrising GIT semistable complete intersections of two quadrics in $\mathbb{P}^3$. By \cite[Corollary 1.2]{alper_reductivity}, these admit good moduli spaces $M^{K}_{2.25}$ of K-polystable members and a GIT quotient $M^{GIT}_{3,2,2}$ respectively.

\begin{theorem}\label{2.25 compactification}
There exists an isomorphism $\mathcal{M}^K_{2.25} \cong \mathcal{M}^{GIT}_{3,2,2}$. In particular, there exists an isomorphism ${M}^K_{2.25} \cong {M}^{GIT}_{3,2,2}$.
%$\overline{M}^{K}_{\operatorname{Bl}_{Q_1\cap Q_2} \mathbb{P}^3}$ of and $\overline{M}^{GIT}_{3,2,2}$.
\end{theorem}
\begin{proof}

Let $\mathcal X$ be the Hilbert polynomial of smooth elements of the family of Fano threefolds $2.25$ pluri-anticanonically embedded by $-mK_{\mathcal{X}}$ in $\mathbb P^N$, and let $\mathbb H^{\mathcal X; N}\coloneqq \mathrm{Hilb}_\mathcal X(\mathbb P^N)$. Given a closed subscheme $X\subset \mathbb P^N$ with Hilbert polynomial $\mathcal X(X, \mathcal O_{\mathbb P^N}(k)|_X)=\mathcal X(k)$, let $\mathrm{Hilb}(X)\in \mathbb H^{\mathcal X; N}$ denote its Hilbert point. Let
$$\hat Z_m\coloneqq\left\{ \mathrm{Hilb}(X)\in \mathbb H^{\mathcal X; N} \;\middle|\; 
\begin{aligned}
  & X \text{ is a Fano manifold of family } 2.25,\\
  &\mathcal O_{P^N}(1)|_X\sim \mathcal O_X(-mK_X),\\
  &\text{ and }H^0(\mathbb P^N, \mathcal O_{\mathbb P^N}(1))\xrightarrow{\cong} H^0(X, \mathcal O_X(-mK_X)).
  \end{aligned}
\right\}$$
which is a locally closed subscheme of $\mathbb H^{\mathcal X; N}$. Let $\overline Z_m$ be its Zariski closure in  $\mathbb H^{\mathcal X; N}$ and $Z_m$ be the subset of $\hat Z_m$ consisting of K-semistable varieties.

Since by Corollary \ref{GIT implies K-stability p3} the blow-up of a smooth complete intersection of two quadrics in $\mathbb{P}^3$ is K-stable and by \cite{odaka-moduli}, the smooth K-stable loci is a Zariski open set of $ M^K_{2.25}$, we have that the connected component $\mathcal M^K_{2.25}$, is defined as $\mathcal M^K_{2.25}=[Z_m/\operatorname{PGL}(N_m+1)]$ for appropriate $m>0$. In fact, we also have $\mathcal M^{GIT}_{3,2,2}\cong[\overline Z_m/\operatorname{PGL}(N_m+1)]$.

Thus, by Lemmas \ref{gh_limit} and \ref{strictly k-ss for 2.25} and Corollary \ref{GIT implies K-stability p3} we can define an open immersion of representable morphism of stacks:
\begin{center}
    \begin{tikzcd}
     \mathcal{M}^{GIT}_{3,2,2}\arrow[r, "\phi"] & \mathcal{M}^K_{2.25}\\
    \left[C_1\cap C_2\right] \arrow[r, mapsto, "\phi"] & \left[\operatorname{Bl}_{C_1\cap C_2} \mathbb{P}^3\right].
    \end{tikzcd}
\end{center}

Note that representability follows once we prove that the base-change of a scheme mapping to the K-moduli stack is itself a scheme. Such a scheme mapping to the K-moduli stack is the same as a $\mathrm{PGL}$-torsor over $Z_m$, which produces a $\mathrm{PGL}$-torsor over $\overline Z_m$ after a $\mathrm{PGL}$-equivariant base change. This $\mathrm{PGL}$-torsor over $\overline Z_m$ shows the desired pullback is a scheme. By \cite[Lemma 06MY]{stacks-project}, since $\phi$ is an open immersion of stacks, $\phi$ is separated and, since it is injective, it is also quasi-finite.

% \begin{equation*}
%   \begin{split}
%     \phi \colon \mathcal{M}^{GIT}_{3,2,2} &\rightarrow \mathcal{M}^K_{2.25}\\
%     [C_1\cap C_2] &\rightarrow [\operatorname{Bl}_{C_1\cap C_2} \mathbb{P}^3]
%   \end{split}
% \end{equation*}
We now need to check that $\phi$ is an isomorphism that descends (as isomorphism of schemes) to the moduli spaces
\begin{center}
    \begin{tikzcd}
     {M}^{GIT}_{3,2,2}\arrow[r, "\overline{\phi}"] & M^K_{2.25}\\
    \left[C_1\cap C_2\right] \arrow[r, mapsto, "\overline{\phi}"] & \left[\operatorname{Bl}_{C_1\cap C_2} \mathbb{P}^3\right]
    \end{tikzcd}
\end{center}
since we have a morphism $\phi$ of stacks, both of which admit moduli spaces. Thus $\overline \phi$ is injective. 
% \begin{equation*}
%   \begin{split}
%     \overline{\phi} \colon M^{GIT}_{3,2,2} &\rightarrow M^{K}_{\operatorname{Bl}_{C_1\cap C_2} \mathbb{P}^3}\\
%     [C_1\cap C_2] &\rightarrow [\operatorname{Bl}_{C_1\cap C_2} \mathbb{P}^3].
%   \end{split}
% \end{equation*}

Now, by \cite[Prop 6.4]{Alper}, since $\phi$ is representable, quasi-finite and separated, $\overline \phi$ is finite and $\phi$ maps closed points to closed points, we obtain that $ \phi$ is finite. Thus, by Zariski's Main Theorem, as $\overline \phi$ is a birational morphism with finite fibers to a normal variety, $\phi$ is an isomorphism to an open subset, but it is also an open immersion, thus it is an isomorphism.

\end{proof}

\begin{corollary}\label{2.25 explicit descr}
The variety $X = \operatorname{Bl}_{C_1\cap C_2}\mathbb{P}^3$ in family $2.25$ is:
    \begin{enumerate}
        \item K-semistable if and only if $C_1\cap C_2$ has at worse $\mathbf{A}_1$ singularities;
        \item strictly K-polystable if and only if $C_1\cap C_2$ is the unique curve with $4$ $\mathbf{A}_1$ singularities;
        \item K-stable if and only if $C_1\cap C_2$ is smooth.
    \end{enumerate}
\end{corollary}
\begin{proof}
    Follows directly from Theorem \ref{2.25 compactification} and the results of Section \ref{sec: GIT classification P3}.
\end{proof}

\begin{remark}
We call the above method of proof, i.e. the construction of the open immersion from the GIT stack to the K-moduli stack, the \emph{reverse moduli continuity method}. In particular, we refer to the reverse moduli continuity method, as the practice of proving that GIT polystable orbits are K-polystable.

A keen reader will notice that this method adapts the moduli continuity method, which has appeared in different forms in \cite{odaka_spotti_sun_2016}, \cite{spotti_sun_2017},\cite{Gallardo_2020}. In that method, one defines a map 
\begin{equation*}
    \phi \colon \overline{M}^{K} \rightarrow \overline{M}^{GIT}
\end{equation*}
to prove the existence of a homeomorphism using the properties of the moduli spaces and the continuity of $\phi$. In this instance, the definition of the map is reversed, due to the existence of Lemma \ref{gh_limit} and Corollary \ref{GIT implies K-stability p3}. In more general cases, one can initiate the method of proof by showing that GIT polystable orbits are K-polystable. This would be an analogue of Lemma \ref{gh_limit} in different K-moduli examples. An analogue of Lemma \ref{strictly k-ss for 2.25}, where GIT semistable orbits are shown to be K-semistable, then follows directly from the methods of proof of Lemma \ref{strictly k-ss for 2.25}. These results, allow us to define a map 
\begin{equation*}
    \phi \colon \overline{M}^{GIT} \rightarrow \overline{M}^{K}
\end{equation*}
which is the reverse direction of the map defined in the moduli continuity method. The method of proof of Theorem \ref{2.25 compactification} can then prove isomorphisms between the K-moduli and GIT moduli stacks in the particular K-moduli problems studied.

In particular, we expect that more situations where GIT implies K-stability will arise in the near future due to new methods in detecting K-stability, such as the computational Abban-Zhuang theory \cite{abban-zhuang}, which has already been applied to Fano threefolds \cite{fano_book, abban2023onedimensional}. The reverse moduli continuity method is then expected to be used, as the proofs of Lemma \ref{gh_limit} and Theorem \ref{2.25 compactification} are general and should work in a number of different deformation families of Fano varieties. In fact, this method of proof has already appeared in \cite{cheltsov2023kstability} implicitly to prove their main results.
\end{remark}

\section{VGIT of complete intersections of quadrics in \texorpdfstring{$\mathbb{P}^4$}{TEXT} and a hyperplane}\label{p4 VGIT section}
In this Section, we will study VGIT quotients of complete intersection of quadrics in $\mathbb{P}^4$ and a hyperplane, using the computational methods presented in Section \ref{VGITsection}. We will first provide some general results on the singularities of such complete intersections based on \cite{mabuchi_mukai_1990}, and then we will provide a full GIT classification. We will then proceed to classify all possible singularities of pairs $(S, D = S\cap H)$ and provide a full VGIT classification using our computational method. This in turn will be of use, in later Sections, when we will study the K-stability of such pairs.

\subsection{General results}\label{general_results for P4}
Throughout this section, we will use the geometric classification of pencils of quadrics in $\mathbb{P}^4$ on their singularities based on their Segre symbols found in \cite[\S 8.6 and Table 8.6]{dolgachev_2012}. We summarise the results in the following table. We generate the polynomials $f$, $g$ via the computational study of each Segre symbol.

\begin{longtable}{| p{2.6cm} | p{8.45cm} |p{1.9cm} |}
\hline
 Segre symbol & Generating polynomials & Singularities \\
 \hline
 $[(4,1)]$&  \begin{center}
     $\begin{aligned}[c]
     f&= q_1(x_2,x_3,x_4)+x_1l_1(x_3,x_4)+x_0x_4\\
     g&= q_2(x_3,x_4)+x_4l_2(x_1,x_2)
   \end{aligned}$
  \end{center} & $\mathbf{D}_5$\\
 \hline
 $[4,1]$ &\begin{center}
     $\begin{aligned}[c]
     f&= q_1(x_2,x_3,x_4)+x_1l_1(x_2,x_3,x_4) +x_0x_4\\
     g&= q_2(x_2,x_3,x_4)+x_1x_4
   \end{aligned}$
  \end{center} & $\mathbf{A}_3$ \\
 \hline
 $[(3,1),1]$ & \begin{center}
     $\begin{aligned}[c]
     f&= q_1(x_1,x_2,x_3,x_4)+x_0x_4\\
     g&= q_2(x_3,x_4)+x_4l_1(x_1,x_2)
   \end{aligned}$
  \end{center} & $\mathbf{D}_4$ \\
 \hline
 $[3,2]$&
 \begin{center}
     $\begin{aligned}[c]
     f&= q_1(x_2,x_3,x_4)+x_4l_1(x_0,x_1)+ x_3l_2(x_0,x_1)\\
     g&= q_2(x_3,x_4)+x_4l_3(x_0,x_1,x_2)+ x_2x_3
   \end{aligned}$
  \end{center} & $2\mathbf{A}_1+\mathbf{A}_2$ \\
 \hline
 $[3,1,1]$&
 \begin{center}
     $\begin{aligned}[c]
     f&= q_1(x_1,x_2,x_3,x_4)+x_0x_4\\
     g&= q_2(x_2,x_3,x_4)+x_1x_4
   \end{aligned}$
  \end{center} & $\mathbf{A}_2$ \\
 \hline
 $[3,(1,1)]$&
\begin{center}
     $\begin{aligned}[c]
     f&= q_1(x_1,x_2,x_3,x_4)+x_0x_1\\
     g&=x_4^2+x_3^2+x_2x_1
   \end{aligned}$
  \end{center} & $\mathbf{A}_2+\mathbf{A}_1$ \\
 \hline
 $[(2,1),2]$&
\begin{center}
     $\begin{aligned}[c]
     f&= x_1^2+x_0l_1(x_2,x_3,x_4)\\
     g&= q_1(x_2,x_3,x_4)
   \end{aligned}$
  \end{center} & $\mathbf{A}_1+\mathbf{A}_3$ \\
 \hline
 $[(2,1),1,1]$&
\begin{center}
     $\begin{aligned}[c]
     f&= q_1(x_1,x_2,x_3,x_4)+x_0l_1(x_2,x_3,x_4)\\
     g&= q_2(x_2,x_3,x_4)
   \end{aligned}$
  \end{center} & $\mathbf{A}_3$ \\
 \hline
 $[(2,1),(1,1)]$&
\begin{center}
     $\begin{aligned}[c]
     f&= x_0x_4+x_1x_3+x_2^2\\
     g&= q_1(x_3,x_4)+x_3l_1(x_1,x_2)+ x_4l_2(x_1,x_2)
   \end{aligned}$
  \end{center} & $2\mathbf{A}_1+\mathbf{A}_3$ \\
 \hline
 $[2,1,1,1]$&\begin{center}
     $\begin{aligned}[c]
     f&= q_1(x_1,x_2,x_3,x_4)+x_0x_4\\
     g&= q_2(x_1,x_2,x_3,x_4)+x_0x_4
   \end{aligned}$
  \end{center} & $\mathbf{A}_1$ \\
 \hline
 $[2,2,1]$&\begin{center}
     $\begin{aligned}[c]
     f&= q_1(x_2,x_3,x_4)+x_4l_1(x_0,x_1) + x_3l_2(x_0,x_1)\\
     g&= q_2(x_2,x_3,x_4)+x_4l_3(x_0,x_1) + x_3l_4(x_0,x_1)
   \end{aligned}$
  \end{center} & $2\mathbf{A}_1$ \\
 \hline
 $[(1,1),2,1]$ &\begin{center}
     $\begin{aligned}[c]
    f(x_0,x_1,x_2,x_3,x_4) =& q_5(x_1,x_2,x_3,x_4)+x_0l_7(x_3,x_4)\\
    g(x_0,x_1,x_2,x_3,x_4) =& q_6(x_3,x_4)+x_4l_8(x_0,x_1,x_2)
    \end{aligned}$
  \end{center} & $3\mathbf{A}_1$\\
 \hline
 $[(1,1),1,1,1]$&\begin{center}
     $\begin{aligned}[c]
     f&= q_1(x_1,x_2,x_3)+x_0x_4\\
     g&= q_2(x_1,x_2,x_3)+x_0x_4
   \end{aligned}$
  \end{center} & $2\mathbf{A}_1$ \\
 \hline
 $[(1,1),(1,1),1]$&\begin{center}
     $\begin{aligned}[c]
     f&= x_4l_1(x_0,x_1)+x_3l_2(x_0,x_1)+x_2^2\\
     g&= x_4l_3(x_0,x_1)+x_3l_4(x_0,x_1)+x_2^2
  \end{aligned}$
  \end{center} & $4\mathbf{A}_1$ \\
 \hline
 $[1,1,1,1,1]$ & \begin{center}
     $\begin{aligned}[c]
     f&= q_1(x_0,x_1,x_2,x_3,x_4)\\
     g&= q_2(x_0,x_1,x_2,x_3,x_4)
   \end{aligned}$
  \end{center} & None\\ 
 \hline
%\end{tabular}
\caption{Segre symbols, equations up to projective equivalence, and singularities of complete intersections of quadrics in $\mathbb{P}^4$.}
  \label{tab:p4segtable_with_equations}
%\end{table}
\end{longtable}

\begin{remark}
Notice, that due to Theorem \ref{segre thm}, all singular pencils that have the same Segre symbol (except for $[2,1,1,1]$ and $[(1,1),1,1,1]$) are projectively equivalent, and the above table gives their full classification.
\end{remark}

\begin{remark}\label{magma code to check sings p4}
In order to check the type of (isolated) hypersurface singularities, one can employ the MAGMA script of Section \ref{general_results for P3}, adjusted accordingly for each different case.
\end{remark}

\subsection{Classifying the singularities of pairs \texorpdfstring{$(S,D = S\cap H)$}{TEXT}}\label{sings of pairs p4}
Following the discussion of Section \ref{VGITsection}, we take $S = C_1\cap C_2$, and $D = S\cap H$, where the $C_i$ are hyperquadrics in $\mathbb{P}^4$ and $H$ is a hyperplane. The lemmas below serve as to help with the geometric classification of such pairs.

\begin{lemma}\label{S_smooth_D_whatever}
Let $S$ be a smooth complete intersection of two quadrics and $H$ a general hyperplane. Then $D$ has at worse $\mathbf{D}_4$ singularities. 
\end{lemma}
\begin{proof}
%Consider $4$ unique (up to linear transformation) points $P_1,P_2,P_3,P_4$ in $\mathbb{P}^3$. Then, there exists a pencil of conics through these four points, where we pick one such conic $C$. Let $L$ be the tangent line to $C$ at a point $R$ which is not one of the $4$ points. Then $L+C\sim-K_{\mathbb{P}^2}$ and the singularity at $R$ is an $(\mathbf{A}_3)$ singularity. Consider a general point $Q$ on the line. We blow up the $5$ points $P_1,P_2,P_3,P_4,Q$ to obtain a del Pezzo surface of degree $4$, $S$ i.e. a complete intersection of $2$ quadrics in $\mathbb{P}^4$. Notice that a hyperplane section of $S$ is $-K_S$. The blow-up $\pi:S\rightarrow \mathbb{P}^2 $ is an isomorphism around R, i.e. so the proper transform $\tilde L+ \tilde C$ still has an $\mathbf{A}_3$ singularity at $\pi^{-1}(R)$. Then, since $-K_S = -\pi^*(K_{\mathbb{P^2}}) -\sum E_i$, where $E_i$ are the exceptional divisors, and $\pi^*(L+C)\sim \tilde L + \tilde C +\sum E_i$ we have 
%\begin{equation*}
%  \begin{split}
%    -K_S &= -\pi^*(K_{\mathbb{P^2}}) -\sum E_i\\
%    &= \pi^*(L+C)-\sum E_i\\
%    & = \tilde L+\tilde C.
%  \end{split}
%\end{equation*}
%Hence the pair $(S, D = \tilde L+ \tilde C)$ has $S$ smooth, and $D$ with $\mathbf{A}_3$ singularities.
%Since \pi^*(L+C)\sim \tilde L + \tilde C +\sum E_i, we have that -K_S\sim \tilde L + \tilde C. So there you go, we have (S, D=\tilde C + \tilde L) with an A_3 singularity on D and S smooth.

Let $S$ be given by $f$, $g$ and $H$ be a hyperplane where
\begin{equation*}
  \begin{split}
    f(x_0,x_1,x_2,x_3,x_4)=&q_1(x_1,x_2,x_3,x_4)+x_0l_1(x_3,x_4)\\
    g(x_0,x_1,x_2,x_3,x_4)=&q_2(x_3,x_4) + x_4l_2(x_0,x_1,x_2)+x_3l_3(x_1,x_2)\\
    H(x_0,x_1,x_2,x_3,x_4) =& x_4.
  \end{split}
\end{equation*}
Then $S$ has Segre symbol $[1,1,1,1,1]$ and is smooth and $D$ is given by 
\begin{equation*}
  \begin{split}
    f(x_0,x_1,x_2,x_3)=&q_1(x_1,x_2,x_3)+x_0lx_3      \\
    g(x_0,x_1,x_2,x_3)=&x_3^2+x_3l_1(x_1,x_2)
  \end{split}
\end{equation*}
which is an intersection of two quadrics in $\mathbb{P}^3$ with Segre symbol $[(3,1)]$ and by Sommerville \cite[\S XIII]{sommerville} and Lemma \ref{Sings of all remaining varieties P3} it has $\mathbf{D}_4$ singularities.

Similarly let $S$ be given by $f$, $g$ and $H$ be a hyperplane where
\begin{equation*}
  \begin{split}
    f(x_0,x_1,x_2,x_3,x_4) =& q_1(x_2,x_3,x_4)+x_1l_1(x_2,x_3,x_4)+x_0l_2(x_3,x_4)\\
    g(x_0,x_1,x_2,x_3,x_4) =& q_2(x_2,x_3,x_4)+x_1l_3(x_3,x_4)+x_0x_4\\
    H(x_0,x_1,x_2,x_3,x_4) =& x_4.
  \end{split}
\end{equation*}
Then $S$ has Segre symbol $[1,1,1,1,1]$ and is smooth and $D$ is given by 
\begin{equation*}
  \begin{split}
    f(x_0,x_1,x_2,x_3)=&q_1(x_2,x_3)+x_1l_1(x_2,x_3)+x_0x_3      \\
    g(x_0,x_1,x_2,x_3,x_4)=&q_2(x_2,x_3)+x_1x_3
  \end{split}
\end{equation*}
which is an intersection of two quadrics in $\mathbb{P}^3$ with Segre symbol $[4]$ and by Table \ref{tab:p3segtable_with_equations} it is a twisted cubic with a tangent line, and it has $\mathbf{A}_3$ singularities. 

Now let $S$ be given by $f$, $g$ and $H$ be a hyperplane where
\begin{equation*}
  \begin{split}
    f(x_0,x_1,x_2,x_3,x_4) =& q_1(x_1,x_2,x_3,x_4)+x_0l_1(x_2,x_3,x_4)\\
    g(x_0,x_1,x_2,x_3,x_4) =& q_2(x_2,x_3,x_4)+x_4l_1(x_0,x_1)\\
    H(x_0,x_1,x_2,x_3,x_4) =& x_4.
  \end{split}
\end{equation*}
Then $S$ has Segre symbol $[1,1,1,1,1]$ and is smooth and $D$ is given by 
\begin{equation*}
  \begin{split}
    f(x_0,x_1,x_2,x_3)=&q_1(x_1,x_2,x_3)+x_0l_1(x_2,x_3)\\
    g(x_0,x_1,x_2,x_3,x_4)=&q_2(x_2,x_3)
  \end{split}
\end{equation*}
which is an intersection of two quadrics in $\mathbb{P}^3$ with Segre symbol $[3,1]$ and by Table \ref{tab:p3segtable_with_equations} it is a cuspidal curve with $\mathbf{A}_2$ singularities.

To conclude, let $S$ be given by $f$, $g$ and $H$ be a hyperplane where
\begin{equation*}
  \begin{split}
    f(x_0,x_1,x_2,x_3,x_4) =& q_1(x_1,x_2,x_3,x_4)+x_0l_1(x_3,x_4)\\
    g(x_0,x_1,x_2,x_3,x_4) =& q_2(x_2,x_3,x_4)+x_1l_2(x_3,x_4)+x_0x_4\\
    H(x_0,x_1,x_2,x_3,x_4) =& x_4.
  \end{split}
\end{equation*}
Then $S$ has Segre symbol $[1,1,1,1,1]$ and is smooth and $D$ is given by 
\begin{equation*}
  \begin{split}
    f(x_0,x_1,x_2,x_3)=&q_1(x_1,x_2,x_3)+x_0x_3\\
    g(x_0,x_1,x_2,x_3,x_4)=&q_2(x_2,x_3)+x_1x_3
  \end{split}
\end{equation*}
which is an intersection of two quadrics in $\mathbb{P}^3$ with Segre symbol $[2,1,1]$ and by Table \ref{tab:p3segtable_with_equations} it is a nodal curve with $\mathbf{A}_1$ singularities. 
%Note as well,that to cook this guy we made no special choices. You could have started with 5 general points in P^2, take a line, find a conic through 4 other points and since there is a pencil of them it should be tangent to the line at some point. So the point of this observation is the following: given any line in S you can always find a tangent conic to it. In fact, you can find two of them as you have two choices for q given any line in S.
\end{proof}

\begin{lemma}\label{S 4A1s [(1,1),(1,1),1],P4}
Let $S$ be the complete intersection of two quadrics $f,g$ with Segre symbol $[(1,1),(1,1),1]$. Then $S$ has $4$ $\mathbf{A}_1$ singularities. Let $H$ be a hyperplane. Then, the hyperplane section $D = S\cap H$ has up to an $\operatorname{SL}(5)$-action:
\begin{enumerate}
%  \item an $\mathbf{A}_3$ singualarity if $S$ is given by $f = q_1(x_2,x_3,x_4)+l_1(x_2,x_3,x_4)x_1+l_2(x_2,x_3,x_4)x_0 +x_1x_0$, $g = x_2^2+x_3^2+x_4^2$, $S = \{f= g= \}$ and $H = \{l_3(x_2,x_3,x_4) = 0\}$; (this seems to be$\mathbf{A}_1$and not \mathbf{A}_3)
  
  \item no singularities if and only if  $H =\{l(x_1,x_2,x_3,x_4)=0\}$;
  \item an $\mathbf{A}_1$ singularity at one of the singularities of $S$ if and only if $H =\{l(x_2,x_3,x_4)=0\}$;
  \item $4$ $\mathbf{A}_1$ singularities at the singularities of $S$  if and only if $H = \{x_2 = 0\}$.
\end{enumerate}
\end{lemma}
\begin{proof}
From Table \ref{tab:p4segtable_with_equations} we know that $S$ is given, up to projective equivalence, by 
\begin{equation*}
   \begin{split}
     f&= x_4l_1(x_0,x_1)+x_3l_2(x_0,x_1)+x_2^2\\
     g&= x_4l_3(x_0,x_1)+x_3l_4(x_0,x_1)+x_2^2
   \end{split}
 \end{equation*}
 
Let $H =\{(x_1,x_2,x_3,x_4) = 0\} $; using a suitable change of coordinates $\tilde x_4 = l(x_1,x_2,x_3,x_4)$, $\tilde x_i = x_i$ for $i\neq 4$ we have (by abuse of notation):
\begin{equation*}
  \begin{split}
    f&= ll_1(x_0,x_1)+x_3l_2(x_0,x_1)+x_2^2\\
     g&= ll_3(x_0,x_1)+x_3l_4(x_0,x_1)+x_2^2\\
    H &= x_4 
  \end{split}
\end{equation*}
and hence $D$ will be given by:

\begin{equation*}
  \begin{split}
    f' &= q_1(x_1,x_2,x_3)+x_0l_1(x_1,x_2,x_3)\\
    g' & = q_2(x_1,x_2,x_3)+x_0l_2(x_1,x_2,x_3)
  \end{split}
\end{equation*}
which is a smooth complete intersection of two quadrics in $\mathbb{P}^3$.

Let $H =\{l(x_2,x_3,x_4) = 0\} $; using a similar suitable change of coordinates we have (by abuse of notation):
\begin{equation*}
  \begin{split}
    f&= ll_1(x_0,x_1)+x_3l_2(x_0,x_1)+x_2^2\\
     g&= ll_3(x_0,x_1)+x_3l_4(x_0,x_1)+x_2^2\\
    H &= x_4 
  \end{split}
\end{equation*}
and hence $D$ will be given by:

\begin{equation*}
  \begin{split}
    f' &= q_1(x_2,x_3)+x_0l_1(x_2,x_3)+x_1l_2(x_2,x_3)\\
    g' & = q_2(x_2,x_3)+x_0l_3(x_2,x_3)+x_1l_4(x_2,x_3)
  \end{split}
\end{equation*}
which is a singular complete intersection of quadrics in $\mathbb{P}^3$ with Segre symbol $[2,2]$ by Table \ref{tab:p3segtable_with_equations} and an $\mathbf{A}_1$ singularity at $(1:0:0:0)$.

Similarly, let $H =\{x_2 = 0\} $; here, $D$ will be given by:

\begin{equation*}
  \begin{split}
    f'&= x_4l_1(x_0,x_1)+x_3l_2(x_0,x_1)\\
     g'&= x_4l_3(x_0,x_1)+x_3l_4(x_0,x_1)
  \end{split}
\end{equation*}
which is a singular complete intersection of quadrics in $\mathbb{P}^3$ with Segre symbol $[(1,1),(1,1)]$ by Table \ref{tab:p3segtable_with_equations} and $4$ $\mathbf{A}_1$ singularities at $(1:0:0:0)$.

\end{proof}

%%%%%%%%%%%%%%%%%%%%%%%%%%%%%%%%%%%%%%%%

The method of proof for the remaining lemmas of this section are identical in method of proof of Lemma \ref{S 4A1s [(1,1),(1,1),1],P4}, and hence are omitted.

\begin{lemma}\label{S 2A1s [(1,1),1,1,1],P4}
Let $S$ be the complete intersection of two quadrics $f,g$ with Segre symbol $[(1,1),1,1,1]$. Then $S$ has $2$ $\mathbf{A}_1$ singularities. Let $H$ be a hyperplane. Then, the hyperplane section $D = S\cap H$ has/is up to an $\operatorname{SL}(5)$-action:
\begin{enumerate}
%  \item an $\mathbf{A}_3$ singualarity if $S$ is given by $f = q_1(x_2,x_3,x_4)+l_1(x_2,x_3,x_4)x_1+l_2(x_2,x_3,x_4)x_0 +x_1x_0$, $g = x_2^2+x_3^2+x_4^2$, $S = \{f= g= \}$ and $H = \{l_3(x_2,x_3,x_4) = 0\}$; (this seems to be$\mathbf{A}_1$and not \mathbf{A}_3)
  
  \item $1$ $\mathbf{A}_1$ singularity at one of the singularities of $S$ if and only if $H =\{l(x_1,x_2,x_3,x_4)=0\}$ or $H =\{l(x_2,x_3,x_4)=0\}$ or $H =\{l(x_3,x_4)=0\}$ ;
  \item $2$ $\mathbf{A}_1$ singularity at the singularities of $S$ if and only if $H =\{l(x_1,x_2,x_3)=0\}$;
  \item non-isolated singularities if and only if $H = \{x_4 = 0\}$.
\end{enumerate}
\end{lemma}

\begin{lemma}\label{S 3A1, [(1,1),2,1],P4}
Let $S$ be the complete intersection of two quadrics $f,g$ with Segre symbol $[(1,1),2,1]]$. Then $S$ has $3$ $\mathbf{A}_1$ singularities. Let $H$ be a hyperplane. Then, the hyperplane section $D = S\cap H$ has/is up to an $\operatorname{SL}(5)$-action:
\begin{enumerate}

  \item no singularities if and only if $H =\{l(x_1,x_2,x_3,x_4)=0\}$or $H =\{l(x_2,x_3,x_4)=0\}$;
  \item $2$ $\mathbf{A}_1$ singularities at two of the singularities of $S$  if and only if $H =\{l(x_3,x_4)=0\}$;
  \item non-isolated singularities if and only if $H = \{x_4=0\}$.
\end{enumerate}
\end{lemma}

\begin{lemma}\label{S 2A1, [2,21],P4}
Let $S$ be the complete intersection of two quadrics $f,g$ with Segre symbol $[2,2,1]$. Then $S$ has $2$ $\mathbf{A}_1$ singularities. Let $H$ be a hyperplane. Then, the hyperplane section $D = S\cap H$ has/is up to an $\operatorname{SL}(5)$-action:
\begin{enumerate}
%  \item an $\mathbf{A}_3$ singualarity if $S$ is given by $f = q_1(x_2,x_3,x_4)+l_1(x_2,x_3,x_4)x_1+l_2(x_2,x_3,x_4)x_0 +x_1x_0$, $g = x_2^2+x_3^2+x_4^2$, $S = \{f= g= \}$ and $H = \{l_3(x_2,x_3,x_4) = 0\}$; (this seems to be$\mathbf{A}_1$and not \mathbf{A}_3)
  
  \item no singularities if and only if $H =\{l(x_1,x_2,x_3,x_4)=0\}$; 
  \item $1$ $\mathbf{A}_1$ singularity at one of the singularities of $S$  if and only if $H =\{l(x_2,x_3,x_4)=0\}$ or $H =\{l(x_3,x_4)=0\}$ or $H = \{x_4 = 0\}$.
\end{enumerate}
\end{lemma}

\begin{lemma}\label{S A1 [2,1,1,1],P4}
Let $S$ be the complete intersection of two quadrics $f,g$ with Segre symbol $[2,1,1,1]$. Then $S$ has $1$ $\mathbf{A}_1$ singularities. Let $H$ be a hyperplane. Then, the hyperplane section $D = S\cap H$ has/is up to an $\operatorname{SL}(5)$-action:
\begin{enumerate}
%  \item an $\mathbf{A}_3$ singualarity if $S$ is given by $f = q_1(x_2,x_3,x_4)+l_1(x_2,x_3,x_4)x_1+l_2(x_2,x_3,x_4)x_0 +x_1x_0$, $g = x_2^2+x_3^2+x_4^2$, $S = \{f= g= \}$ and $H = \{l_3(x_2,x_3,x_4) = 0\}$; (this seems to be$\mathbf{A}_1$and not \mathbf{A}_3)
  \item $1$ $\mathbf{A}_1$ singularity at the singularity of $S$ if and only if $H =\{l(x_1,x_2,x_3,x_4)=0\}$ or $H =\{l(x_2,x_3,x_4)=0\}$ or $H =\{l(x_3,x_4)=0\}$;
  \item non-isolated singularities if and only if $H = \{x_4 =0\}$.
\end{enumerate}
\end{lemma}

\begin{lemma}\label{S 2A1+A3 [(2,1),(1,1)],P4}
Let $S$ be the complete intersection of two quadrics $f,g$ with Segre symbol $[(2,1),(1,1)]$. Then $S$ has $1$ $\mathbf{A}_1$ and $2$ $\mathbf{A}_3$ singularities. Let $H$ be a hyperplane. Then, the hyperplane section $D = S\cap H$ has/is up to an $\operatorname{SL}(5)$-action:
\begin{enumerate}
  \item $1$ $\mathbf{A}_1$ singularity at the $\mathbf{A}_3$ singularity if and only if $H =\{l(x_1,x_2,x_3,x_4)=0\}$ or $H =\{l(x_2,x_3,x_4)=0\}$;
  \item $1$ $\mathbf{A}_3$ singularity at the $\mathbf{A}_3$ singularity if and only if $H = \{x_2 =0\}$;
  \item non-isolated singularities if and only if $H =\{l(x_3,x_4)=0\}$ or $H = \{x_4 =0\}$.
\end{enumerate}
\end{lemma}
\begin{lemma}\label{S A3 [(2,1),1,1],P4}
Let $S$ be the complete intersection of two quadrics $f,g$ with Segre symbol $[(2,1),1,1]$. Then $S$ has $1$ $\mathbf{A}_3$ singularity. Let $H$ be a hyperplane. Then, the hyperplane section $D = S\cap H$ has/is up to an $\operatorname{SL}(5)$-action:
\begin{enumerate}
  \item $1$ $\mathbf{A}_1$ singularity at the $\mathbf{A}_3$ singularity if and only if $H =\{l(x_1,x_2,x_3,x_4)=0\}$;
  \item $1$ $\mathbf{A}_3$ singularity at the $\mathbf{A}_3$ singularity if and only if $H =\{l(x_2,x_3,x_4)=0\}$ or $H =\{l(x_3,x_4)=0\}$ or $H = \{x_4 =0\}$.
\end{enumerate}
\end{lemma}
\begin{lemma}\label{S A3 [(2,1),2],P4}
Let $S$ be the complete intersection of two quadrics $f,g$ with Segre symbol $[(2,1),2]$. Then $S$ has $1$ $\mathbf{A}_1$ and $1$ $\mathbf{A}_3$ singularity. Let $H$ be a hyperplane. Then, the hyperplane section $D = S\cap H$ has/is up to an $\operatorname{SL}(5)$-action:
\begin{enumerate}
  \item  $1$ $\mathbf{A}_1$ singularity at the $\mathbf{A}_3$ singularity if and only if $H =\{l(x_1,x_2,x_3,x_4)=0\}$;
  \item $1$ $\mathbf{A}_3$ singularity at the $\mathbf{A}_3$ singularity if and only if $H =\{l(x_2,x_3,x_4)=0\}$ or $H =\{l(x_3,x_4)=0\}$;
  \item $1$ $\mathbf{D}_4$ singularity at the $\mathbf{A}_3$ singularity if and only if $H = \{x_4 =0\}$.
\end{enumerate}
\end{lemma}
\begin{lemma}\label{S A3 [3,(1,1)],P4}
Let $S$ be the complete intersection of two quadrics $f,g$ with Segre symbol $[3,(1,1)]$. Then $S$ has $1$ $\mathbf{A}_1$ and $1$ $\mathbf{A}_2$ singularity. Let $H$ be a hyperplane. Then, the hyperplane section $D = S\cap H$ has/is up to an $\operatorname{SL}(5)$-action:
\begin{enumerate}
  \item $2$ $\mathbf{A}_1$ singularities at the singularities of $S$ if and only if $H =\{l(x_1,x_2,x_3,x_4)=0\}$;
  \item $1$ $\mathbf{A}_1$ singularity at the $\mathbf{A}_2$ singularity if and only if $H =\{l(x_2,x_3,x_4)=0\}$;
  \item $1$ $\mathbf{A}_2$ singularity at the $\mathbf{A}_2$ singularity if and only if or $H =\{l(x_3,x_4)=0\}$ or $H = \{x_4 =0\}$.
\end{enumerate}
\end{lemma}
\begin{lemma}\label{S A2 [3,1,1],P4}
Let $S$ be the complete intersection of two quadrics $f,g$ with Segre symbol $[3,1,1]$. Then $S$ has $1$ $\mathbf{A}_2$ singularity. Let $H$ be a hyperplane. Then, the hyperplane section $D = S\cap H$ has/is up to an $\operatorname{SL}(5)$-action:
\begin{enumerate}
  \item  $1$ $\mathbf{A}_1$ singularity at the $\mathbf{A}_2$ singularity if and only if $H =\{l(x_1,x_2,x_3,x_4)=0\}$;
  \item $1$ $\mathbf{A}_2$ singularity at the $\mathbf{A}_2$ singularity if and only if $H =\{l(x_2,x_3,x_4)=0\}$;
  \item $1$ $\mathbf{A}_3$ singularity at the $\mathbf{A}_2$ singularity if and only if or $H =\{l(x_3,x_4)=0\}$;
  \item non-isolated singularities if $H = \{x_4 =0\}$.
\end{enumerate}
\end{lemma}
\begin{lemma}\label{S A1+A2 [3,2],P3}
Let $S$ be the complete intersection of two quadrics $f,g$ with Segre symbol $[3,2]$. Then $S$ has $1$  $\mathbf{A}_1$ and $1$ $\mathbf{A}_2$ singularity. Let $H$ be a hyperplane. Then, the hyperplane section $D = S\cap H$ has/is up to an $\operatorname{SL}(5)$-action:
\begin{enumerate}
  \item no singularities if and only if $H =\{l(x_1,x_2,x_3,x_4)=0\}$;
  \item $2$ $\mathbf{A}_1$ singularities at the singularities of $S$ if and only if $H =\{l(x_2,x_3,x_4)=0\}$;
  \item non-isolated singularities if $H =\{l(x_3,x_4)=0\}$ or $H = \{x_4 =0\}$.
\end{enumerate}
\end{lemma}
\begin{lemma}\label{S D4 [(3,1),1],P4}
Let $S$ be the complete intersection of two quadrics $f,g$ with Segre symbol $[(3,1),1]$. Then $S$ has $1$ $\mathbf{D}_4$ singularity. Let $H$ be a hyperplane. Then, the hyperplane section $D = S\cap H$ has/is up to an $\operatorname{SL}(5)$-action:
\begin{enumerate}
  \item no singularities if and only if $H =\{l(x_1,x_2,x_3,x_4)=0\}$;
  \item $2$ $\mathbf{A}_1$ singularities, where one singularity is the $\mathbf{D}_4$ singularity if and only if $H =\{l(x_2,x_3,x_4)=0\}$;
  \item a double conic if and only if $H =\{l(x_3,x_4)=0\}$;
  \item $1$ $\mathbf{A}_3$ singularity away from the $\mathbf{D}_4$ singularity if and only if $H = \{x_4 =0\}$.
\end{enumerate}
\end{lemma}

\begin{lemma}\label{S A3 [4,1],P4}
Let $S$ be the complete intersection of two quadrics $f,g$ with Segre symbol $[4,1]$. Then $S$ has $1$ $\mathbf{A}_3$ singularity. Let $H$ be a hyperplane. Then, the hyperplane section $D = S\cap H$ has/is up to an $\operatorname{SL}(5)$-action:
\begin{enumerate}
  \item $1$ $\mathbf{A}_1$ singularity if and only if $H =\{l(x_1,x_2,x_3,x_4)=0\}$;
  \item $2$ $\mathbf{A}_1$ singularities if and only if $H =\{l(x_2,x_3,x_4)=0\}$;
 \item $1$ $\mathbf{A}_3$ singularity if and only if $H =\{l(x_3,x_4)=0\}$;
  \item non-isolated singularities if and only if $H = \{x_4 =0\}$.
\end{enumerate}
\end{lemma}

\begin{lemma}\label{S D5 [(4,1)],P4}
Let $S$ be the complete intersection of two quadrics $f,g$ with Segre symbol $[(4,1)]$. Then $S$ has $1$ $\mathbf{D}_5$ singularity. Let $H$ be a hyperplane. Then, the hyperplane section $D = S\cap H$ has/is up to an $\operatorname{SL}(5)$-action:
\begin{enumerate}
  \item $1$ $\mathbf{A}_1$ singularity at the $\mathbf{D}_5$ singularity if and only if $H =\{l(x_1,x_2,x_3,x_4)=0\}$;
  \item $2$ $\mathbf{A}_1$ singularities where one lies is the $\mathbf{D}_5$ singularity at $S$ if and only if $H =\{l(x_2,x_3,x_4)=0\}$;
  \item non-isolated singularities if and only if $H =\{l(x_3,x_4)=0\}$ or $H = \{x_4 =0\}$;
  \item $1$ $\mathbf{A}_2$ singularity away from the singular point $P$ if and only if $S$ is given by 
  \begin{equation*}
  \begin{split}
    f(x_0,x_1,x_2,x_3,x_4) =& x_2^2+ax_1x_3+bx_0x_4\\
    g(x_0,x_1,x_2,x_3,x_4) =& x_3^2+cx_1x_4
  \end{split}
\end{equation*}
and $H = \{x_4 =0\}$.
\end{enumerate}
\end{lemma}

As a direct result from the above theorems, we have the following Lemma:

\begin{lemma}\label{C star action invariance}
Let $(S,D)$ be a pair that is invariant under a non-trivial $\mathbb{G}_m$-action. Suppose the singularities of $S$ and $D$ are given as in the first and second entries in one of the rows of Table \ref{tab:G_m invariant pairs}, respectively. Then $(S,D)$ is projectively equivalent to $({f =g= 0},{f = g = H = 0})$ for $f$, $g$ as in Table \ref{tab:p4segtable_with_equations} corresponding to the Segre symbol of row $3$ of Table \ref{tab:G_m invariant pairs}, and $H$ as in the fourth entries in the same row of Table \ref{tab:G_m invariant pairs}, respectively. In particular, any such pair $(S,D)$ is unique up to projective equivalence. Conversely, if $(S,D)$ is given by equations as in the third and fourth entries in a given row of Table \ref{tab:G_m invariant pairs}, then $(S,D)$ has singularities as in the first and second entries in the same row of Table and $(S,D)$ is $\mathbb{G}_m$-invariant. Furthermore the one-parameter subgroup $\lambda(s) \in \operatorname{SL}(5)$, given in the entry of the corresponding row of Table \ref{tab:G_m invariant pairs} is a generator of the $\mathbb{G}_m$-action.
\end{lemma}
\begin{table}[H] 
  \centering
  \begin{tabular}{|c|c|c|c|c|}
  \hline
    $\operatorname{Sing}(S)$ &$\operatorname{Sing}(D)$ & Segre Symbol& $H$ & $\lambda(s)$ \\
    \hline
    $4\mathbf{A}_1$ & $4\mathbf{A}_1$ at points & $[(1,1),(1,1),1]$ &$x_2$ & $\operatorname{Diag}(s,s,1,s^{-1},s^{-1})$\\
    \hline
    \makecell{$2\mathbf{A}_1$ at \\$P$, $Q$ }& \makecell{$2\mathbf{A}_1$ at points} & $[(1,1),1,1,1]$ &$l(x_1,x_2,x_3)$ & $\operatorname{Diag}(s,1,1,1,s^{-1})$\\
    \hline
    \makecell{$2\mathbf{A}_1+\mathbf{A}_2$ at \\$P$, $Q$, $R$ }& \makecell{non-isolated,\\ $D=2L+L_1+L_2$} & $[3,2]$ &$x_4$ & $\operatorname{Diag}(s,1,1,1,s^{-1})$\\
    \hline
    \makecell{$\mathbf{A}_1+\mathbf{A}_2$ at \\$P$, $Q$, }& \makecell{non-isolated,\\ $D=2L+L_1+L_2$} & $[3,(1,1)]$ &$x_4$ & $\operatorname{Diag}(s,1,1,1,s^{-1})$\\
    \hline
    \makecell{$2\mathbf{A}_1+\mathbf{A}_3$ at\\ $P$, $Q$, $R$ }& double conic & $[(2,1),(1,1)]$ &$x_4$ & $\operatorname{Diag}(s^7,s^2,s^{-3},s^{-3},s^{-3})$\\
    \hline
    \makecell{$\mathbf{A}_1+\mathbf{A}_3$ at\\ $P$, $Q$ }&  $\mathbf{A}_3$ at $P$& $[(2,1),(1,1)]$ &$x_4$ & $\operatorname{Diag}(s^7,s^2,s^{-3},s^{-3},s^{-3})$\\
    \hline
    $\mathbf{A}_3$ at $P$ & $\mathbf{D}_4$ at $P$ & $[4,1]$ &$x_4$ & $\operatorname{Diag}(s^7,s^2,s^2,s^{-3},s^{-8})$\\
    \hline
     \makecell{$\mathbf{A}_3+\mathbf{A}_1$\\ at $P$, $R$} & $\mathbf{A}_1$ at $R$ & $[(2,1),2]$ &$x_4$ & $\operatorname{Diag}(s^7,s^2,s^2,s^{-3},s^{-8})$\\
    \hline
    $\mathbf{D}_4$ at $P$ & $\mathbf{A}_3$ not $P$ & $[(3,1),1]$ &$x_4$ & $\operatorname{Diag}(s^9,s^4,s^{-1},s^{-1},s^{-11})$\\
    \hline
    $\mathbf{D}_5$ at $P$ & $\mathbf{A}_2$ not $P$ & $[(4,1)]$ &$x_0$ & $\operatorname{Diag}(s^9,s^4,s^{-1},s^{-1},s^{-11})$\\
    \hline
  \end{tabular}
  \caption{Some pairs $(S,D)$ invariant under a $\mathbb{G}_m$-action.}
  \label{tab:G_m invariant pairs}
\end{table}

\subsection{VGIT classification}
From the algorithm described in Section \ref{VGITsection} and the computational package \cite{theodoros_stylianos_papazachariou_2022} we obtain the following walls and chambers: 

\begin{center}
\begin{tabular}{ c |c c c c c c c c c c c c c c c}
 & $t_0$ & &$t_1$ & & $t_2$ & &$t_3$ & &$t_4$ & & $t_5$ & & $t_6$\\ 
 walls & 0 & &$\frac{1}{6}$ & & $\frac{2}{7}$ & &$\frac{3}{8}$ & &$\frac{6}{11}$ & & $\frac{2}{3}$ & & $1$\\ 
 chambers & &$\frac{37}{228}$ & & $\frac{327}{1162}$& & $\frac{113}{304}$ & &$\frac{1039}{1914}$ & &$\frac{355}{534}$ & & $\frac{37}{38}$ 
\end{tabular}
\end{center}

We thus obtain $11$ non-isomorphic quotients $M^{GIT}_{4,2,2}(t_i)$, which are characterised by the following two Theorems.

\begin{theorem}\label{main thm in p4 vgit}
Let $(S,D)$ be a pair where $S$ is a complete intersection of two quadrics in $\mathbb{P}^4$ and $D = S\cap H$ is a hyperplane section. 
\begin{enumerate}
  \item $t\in (0, \frac{1}{6})$: The pair $(S,D)$ is $t$-stable if and only $S$ has at worse finitely many $\mathbf{A}_1$ singularities, where $D$ may be non-reduced, or if $S$ is smooth and $D$ has at worse a $\mathbf{D}_4$ singularity.
  \item $t = \frac{1}{6}$: The pair $(S,D)$ is $t$-stable if and only $S$ has at worse finitely many $\mathbf{A}_1$ singularities and $D$ is reduced and has at worst $\mathbf{A}_1$ singularities, or if $S$ is smooth and $D$ has at worse a $\mathbf{D}_4$ singularity.
  \item $t\in (\frac{1}{6}, \frac{2}{7})$: The pair $(S,D)$ is $t$-stable if and only if $S$ has at worse finitely many $\mathbf{A}_2$ singularities and $D$ is reduced and smooth, or if $S$ is smooth and $D$ has at worse a $\mathbf{D}_4$ singularity.
  \item $t= \frac{2}{7}$: The pair $(S,D)$ is $t$-stable if and only $S$ has at worse finitely many $\mathbf{A}_3$ singularities and $D$ can have at worse singularities of type $\mathbf{A}_1$, or if $S$ is smooth and $D$ has at $D$ has at worse a $\mathbf{D}_4$ singularity.
  \item $t\in (\frac{2}{7}, \frac{3}{8})$: The pair $(S,D)$ is $t$-stable if and only $S$ has at worse finitely many $\mathbf{A}_3$ singularities and $D$ is smooth, or if $S$ is smooth and $D$ has at worse a $\mathbf{D}_4$ singularity.
  \item $t= \frac{3}{8}$: The pair $(S,D)$ is $t$-stable if and only $S$ has at worse finitely many $\mathbf{A}_3$ singularities and $D$ is smooth, or if $S$ is smooth and $D$ has at worse an $\mathbf{A}_3$ singularity.
  \item $t \in (\frac{3}{8}, \frac{6}{11})$: The pair $(S,D)$ is $t$-stable if and only $S$ has at worse finitely many $\mathbf{A}_3$ singularities and $D$ is smooth, or if $S$ is smooth and $D$ has at worse an $\mathbf{A}_2$ singularity.
  \item $t = \frac{6}{11}$: The pair $(S,D)$ is $t$-stable if and only $S$ has at worse finitely many $\mathbf{A}_3$ singularities and $D$ is smooth, or if $S$ is smooth and $D$ has at worse a $\mathbf{A}_1$ singularities.
  \item $t \in (\frac{6}{11},\frac{2}{3})$: The pair $(S,D)$ is $t$-stable if and only $S$ has at worse finitely many $\mathbf{A}_3$ singularities and $D$ is smooth, or if $S$ is smooth and $D$ has at worse a $\mathbf{A}_1$ singularities.
  \item $t =\frac{2}{3}$: The pair $(S,D)$ is $t$-stable if and only $S$ has at worse finitely many $\mathbf{A}_3$ singularities and $D$ is smooth, or if $S$ is smooth and $D$ is smooth.
  \item $t \in (\frac{2}{3},1)$: The pair $(S,D)$ is $t$-stable if and only $S$ has at worse finitely many $\mathbf{D}_5$ singularities and $D$ is smooth, or if $S$ is smooth and $D$ is smooth.
\end{enumerate}
\end{theorem}
\begin{proof}
As before, let $S = \{f_1=0\}\cap \{f_2 = 0\}$, $H = \{h=0\}$, $D = S\cap H$. By Theorem \ref{unstable families vgit}, the pair $(S, D)$ is $t$-stable if and only if for any $g \in \operatorname{SL}(5, \mathbb{C})$ the monomials with non-zero coefficients of $(g\cdot f_1, g\cdot f_2, g \cdot H)$ are not contained in a tuple of sets $N^{\ominus}_t(\lambda, x^J, x_p)$-characterized geometrically in Section \ref{sings of pairs p4}- which is maximal for every given $t$, as stated in Theorem \ref{unstable families vgit}. These maximal sets can be found algorithmically using the computational methods of Section \ref{VGITsection} and the computational package \cite{theodoros_stylianos_papazachariou_2022}. This is equivalent to the conditions in the statement. We verify the conditions for each $t \in (0, 1)$. We will refer to the singularities of $S$ and $D$ in terms of the ADE classification as in Sections \ref{p3 VGIT section} and \ref{general_results for P4}. These will be equivalent to the global description used in the statement of Theorem \ref{main thm in p4 vgit}.

Let $t\in (0,\frac{1}{6})$, and let $\lambda, x^J, x_p = (12,2,-3,-3,-8), x_0x_4, x_0$. Then $S$ cannot have an $\mathbf{A}_2$ singularity or a degeneration of one. Similarly, if we consider $\lambda, x^J, x_p = (1,1,0,-1,-1)$, $x_0x_3$, $x_2$ or $\lambda, x^J, x_p = (1,0,0,0,-1)$, $x_0x_4$, $x_0$, we deduce that if $P \in D$ then $P$ is a singular point of $S$ of type at worst $\mathbf{A}_1$. This completes the proof when $t\in (0.\frac{1}{6})$.

For $t = \frac{1}{6}$, the maximal sets $N^{\ominus}_{\frac{1}{6}}(\lambda, x^J, x_p)$ are the same as those for $t\in (0.\frac{1}{6})$ with the addition of sets $N^\ominus_t\big((4,4,-1,-1,-6), x_0x_4,x_4 \big)$, $N^\ominus_t\big((9,4,-1,-6,-6), x_0x_3,x_3 \big)$ and $N^\ominus_t\big((8,3,3,-2,-12), x_0x_4,x_4 \big)$, which represent the monomials of the equations of any pair $(S', D')$ such that $D'$ is not reduced. Therefore, $(S, D)$ is $\frac{1}{6}$-stable if and only if in addition to the conditions for $t$-stability when $t\in (0,\frac{1}{6})$, $D$ is not reduced. Hence, (ii) follows.

Let $t\in (\frac{1}{6}, \frac{2}{7})$. The maximal $t$-non-stable sets $N^{\ominus}_t(\lambda, x^J, x_p)$ are the same as for $t = \frac{1}{6}$ but replacing the sets $N^\ominus_t\big((12,2,-3,-3,-8), x_0x_4, x_0 \big)$ and $N^\ominus_t\big((6,6,1,-4,-9), x_0x_4, x_1 \big)$, with the 
sets $N^\ominus_t\big((2,0,0,-1,-1), x_1x_3,x_1 \big)$, $N^\ominus_t\big((8,3,-2,-2,-7), x_1x_4,x_0 \big)$ and \\ $N^\ominus_t\big((7,7,2,-3,-13), x_0x_4,x_0 \big)$. A pair $(S', D')$ whose defining equations have coefficients in the set $N^\ominus_t\big((8,3,-2,-2,-7), x_1x_4,x_0 \big)$ require that $S'$ has (a degeneration of) an $\mathbf{A}_2$ singularity. Hence, a $t$-stable pair $(S, D)$ may now have $\mathbf{A}_2$ singularities but not $\mathbf{A}_3$ singularities. Therefore,
$(S, D)$ is $t$-stable if and only if $S$ has at worst $\mathbf{A}_2$ singularities, $D$ is reduced and smooth, hence (iii) follows.

Let $t = \frac{2}{7}$. $t$-non-stable sets $N^{\ominus}_t(\lambda, x^J, x_p)$ are the same as for $t \in (\frac{1}{6}, \frac{2}{7})$ but replacing the sets $N^\ominus_t\big((9,4,-1,-6,-6), x_0x_3,x_3 \big)$ and $N^\ominus_t\big((8,3,3,-2,-12), x_0x_4,x_4 \big)$ with the sets $N^\ominus_t\big((3,3,3,-2,-7), x_0x_4,x_4 \big)$, $N^\ominus_t\big((8,3,-2,-2,-7), x_1x_4,x_4 \big)$ and \\ $N^\ominus_t\big((13,3,-2,-7,-7), x_1x_3,x_3 \big)$, which represent the monomials of the equations of any pair $(S', D')$ such that $D'$ has at worse (a degeneration of) an $\mathbf{A}_2$ singularity at a singular point $P$ of $S$. Therefore, $(S, D)$ is $\frac{2}{7}$-stable if and only if in addition to the conditions for $t$-stability when $t\in (\frac{1}{6}, \frac{2}{7})$, $D$ is not reduced but does not have (a degeneration of) an $\mathbf{A}_3$ singularity at a singular point of $P$. Hence, (iv) follows.

Let $t\in (\frac{2}{7}, \frac{3}{8})$. The maximal $t$-non-stable sets $N^{\ominus}_t(\lambda, x^J, x_p)$ are the same as for $t = \frac{2}{7}$,  replacing sets $N^\ominus_t\big((6,1,1,-4,-4), x_1x_3,x_0 \big)$, $N^\ominus_t\big((7,2,-3,-3,-13), x_2^2,x_0 \big)$, \\$N^\ominus_t\big((8,3,3,-2,-12), x_0x_4,x_4 \big)$ and $N^\ominus_t\big((8,3,-2,-2,-7), x_1x_4,x_4 \big)$, with the sets \\ $N^\ominus_t\big((49,4,-1,-16,-36), x_1x_4,x_0 \big)$, $N^\ominus_t\big((2,0,0,-1,-1), x_1x_3,x_1 \big)$, \\$N^\ominus_t\big((8,3,-2,-2,-7), x_1x_4,x_4 \big)$, $N^\ominus_t\big((24,14,-11,-11,-16), x_2^2,x_0 \big)$ and \\$N^\ominus_t\big((11,1,-4,-4,-4), x_2^2,x_1 \big)$. The set $N^\ominus_t\big((8,3,-2,-2,-7), x_1x_4,x_4 \big)$ represents the monomials of the equations of any pair $(S', D')$ such that $D'$ is a general hyperplane section and $S'$ has at worse an $\mathbf{A}_3$ singularity. Therefore, $(S, D)$ is $t$-stable if and only if in addition to the conditions for $t$-stability when $t= \frac{2}{7}$, $D$ is a general hyperplane section and $S$ has at worse a $\mathbf{A}_3$ singular point. Hence, (v) follows.

Let $t = \frac{3}{8}$. The maximal $t$-non-stable sets $N^{\ominus}_t(\lambda, x^J, x_p)$ are the same as for $t \in (\frac{2}{7}, \frac{3}{8})$, with the addition of one new set $N^\ominus_t\big((7,2,2,-3,-8), x_0x_4,x_4 \big)$ which represents the monomials of the equations of any pair $(S', D')$ such that $D'$ has at worse (a degeneration of) a $\mathbf{D}_4$ singularity at a  point $P$ of $S$ (which is smooth). Furthermore, the restrictions for $t \in (\frac{2}{7}, \frac{3}{8})$ regarding D still apply. Therefore, a pair $(S, D)$ is $t$-stable if and only if satisfies the conditions in (vi).

Let $t\in (\frac{3}{8}, \frac{6}{11})$. The maximal $t$-non-stable sets $N^{\ominus}_t(\lambda, x^J, x_p)$ are the same as for $t = \frac{3}{8}$,  but replacing set $N^\ominus_t\big((8,3,-2,-2,-7), x_1x_4,x_0 \big)$ with the set $N^\ominus_t\big((1,1,0,0,-2), x_2^2,x_4 \big)$, which represents the monomials of the equations of any pair $(S', D')$ such that $D'$ has at worse (a degeneration of) an $\mathbf{A}_3$ singularity at a  point $P$ of $S'$ (which is smooth). Hence, a pair $(S, D)$ is $t$-stable if and only if satisfies the conditions in (vii).

Let $t=\frac{6}{11}$. The maximal $t$-non-stable sets $N^{\ominus}_t(\lambda, x^J, x_p)$ are the same as for $t \in (\frac{3}{8}, \frac{6}{11})$,  replacing the sets $N^\ominus_t\big((4,4,-1,-1,-6), x_0x_4,x_4 \big)$ and $N^\ominus_t\big((8,3,-2,-2,-7), x_1x_4,x_4 \big)$ with the set $N^\ominus_t\big((9,4,-1,-1,-11), x_0x_4,x_4 \big)$, which represents the monomials of the equations of any pair $(S', D')$ such that $D'$ has at worse (a degeneration of) an $\mathbf{A}_2$ singularity at a  point $P$ of $S'$ (which is smooth). Hence, a pair $(S, D)$ is $t$-stable if and only if satisfies the conditions in (viii).

Let $t\in (\frac{6}{11}, \frac{2}{3})$. The maximal $t$-non-stable sets $N^{\ominus}_t(\lambda, x^J, x_p)$ are the same as for $t=\frac{6}{11}$, replacing set $N^\ominus_t\big((8,3,3,-2,-12), x_0x_4,x_4 \big)$ with the set $N^\ominus_t\big((3,1,0,-1,-3), x_1x_4,x_0 \big)$, which represents the monomials of the equations of any pair $(S', D')$ such that $S'$ has at worse (a degeneration of) a $\mathbf{D}_5$ singularity and  $D'$ is a general hyperplane section. Therefore, $(S, D)$ is $t$-stable if and only if in addition to the conditions for $t$-stability when $t= \frac{6}{11}$, $D$ is a general hyperplane section and $S$ has at worse a $\mathbf{A}_4$ singular point. Hence, a pair $(S, D)$ is $t$-stable if and only if satisfies the conditions in (ix).

Let $t = \frac{2}{3}$. The maximal $t$-non-stable sets $N^{\ominus}_t(\lambda, x^J, x_p)$ are the same as for $t\in (\frac{6}{11}, \frac{2}{3})$, replacing the set $N^\ominus_t\big((1,1,0,0,-2), x_2^2,x_4 \big)$ with the set $N^\ominus_t\big((3,1,0,-1,-3), x_0x_4,x_4 \big)$, which represents the monomials of the equations of any pair $(S', D')$ such that $D'$ and $S'$ are both smooth. Hence, a pair $(S, D)$ is $t$-stable if and only if satisfies the conditions in (x).

Let $t\in (\frac{2}{3},1 )$. The maximal $t$-non-stable sets $N^{\ominus}_t(\lambda, x^J, x_p)$ are the same as for $t = \frac{2}{3} $, removing the set $N^\ominus_t\big((3,1,0,-1,-3), x_1x_4,x_0 \big)$. Hence, $S$ can have at worse $\mathbf{D}_5$ singularities, with $D$ general, and a pair $(S, D)$ is $t$-stable if and only if it satisfies the conditions in (xi).

For a more detailed description of the above maximally destabilizing tuples, written via generating polynomials, we prompt the reader to the thesis \cite[\S 7.4]{my_phd_thesis}.
\end{proof}

\begin{theorem}\label{main thm in vgit p4_polystable}
Let $t \in (0,1)$. If $t$ is a chamber, or $t = t_5$, then $\overline{M(t)}$ is the compactification of the stable loci $M(t)$ by the closed $\operatorname{SL}(5)$-orbit in $\overline{M(t)}\setminus M(t)$ represented by the pair $(\tilde{S},\tilde{D})$, where $\tilde{S}$ is the unique $\mathbb{G}_m$-invariant complete intersection of two quadrics with Segre symbol $[(1,1),(1,1),1]$ and $4$ $\mathbf{A}_1$ singularities, and $\tilde{D}$ is the union of the unique four lines in $\tilde{S}$, each of them passing through two of those singularities, or the pair $(S', D')$, where $S'$ is the complete intersection with Segre symbol $[(1,1),1,1,1]$ and $2$ $\mathbf{A}_1$ singularities, and $D'$ is two conics in general position with $2$ $\mathbf{A}_1$ singularities at the singular points of $S'$. If $t = t_i$, for $i = 1,2,3,4,5$, then $\overline{M(t_i)}$ is the compactification of the stable loci $M(t_i)$ by the three closed $\operatorname{SL}(5)$-orbits in $\overline{M(t)}\setminus M(t)$ represented by the uniquely defined pairs $(\tilde{S},\tilde{D})$, $(S',D')$ described above, and the $\mathbb{G}_m$-invariant pairs $(S_i, D_i)$, $(S'_i, D'_i)$ uniquely defined as follows:
\begin{enumerate}
  \item the complete intersection $S_1$ of two quadrics with Segre symbol $[3,2]$ with $2$ $\mathbf{A}_1$ and $1$ $\mathbf{A}_2$ singularities, and the divisor $D_1\in |-K_{S_1}|$, where $D_1 = 2L+L_1+L_2$ (a double line and two lines meeting at two points), with non-isolated singularities, and the complete intersection $S'_1$ of two quadrics with Segre symbol $[3,(1,1)]$ with $1$ $\mathbf{A}_1$ and $1$ $\mathbf{A}_2$ singularities, and the divisor $D'_1\in |-K_{S'_1}|$, where $D'_1 = 2L+L_1+L_2$ (a double line and two lines meeting at two points), with non-isolated singularities;
  \item the complete intersection $S_2$ of two quadrics with Segre symbol $[(2,1),(1,1)]$ with $2$ $\mathbf{A}_1$ and $1$ $\mathbf{A}_3$ singularities, and the divisor $D_2\in |-K_{S_2}|$, where $D_2$ is a double conic, and the complete intersection $S'_2$ of two quadrics with Segre symbol $[(2,1),2]$ with $1$ $\mathbf{A}_1$ and $1$ $\mathbf{A}_3$ singularities, and the divisor $D'_2\in |-K_{S'_2}|$, where $D'_2$ has an $\mathbf{A}_3$ singularity at the $\mathbf{A}_1$ singularity of $S'_2$;
  \item the complete intersection $S_3$ of two quadrics with Segre symbol $[4,1]$ with $1$ $\mathbf{A}_3$ singularity, and the divisor $D_3\in |-K_{S_3}|$, where $D_3$ is a conic and two lines intersecting in one point, with a $\mathbf{D}_4$ singularity at the $\mathbf{A}_3$ singularity of $S_2$;
  \item the complete intersection $S_4$ of two quadrics with Segre symbol $[(3,1),1]$ with $1$ $\mathbf{D}_4$ singularity, and the divisor $D_4\in |-K_{S_4}|$, where $D_4$ is two tangent conics with a $\mathbf{A}_3$ singularity, and $D_4$ does not contain the singular point of $S_4$;
  \item the complete intersection $S_4$ of two quadrics with Segre symbol $[(4,1)]$ with $1$ $\mathbf{D}_5$ singularity, and the divisor $D_5\in |-K_{S_5}|$, where $D_5$ is a cuspidal curve with a $\mathbf{A}_2$ singularity, and $D_5$ does not contain the singular point of $S_5$.
\end{enumerate}
\end{theorem}

\begin{proof}
    As before, let $S = \{f=0\}\cap \{g = 0\}$, $H = \{h=0\}$, $D = S\cap H$, and let the pair $(S,D)$ belong to a closed
strictly $t$-semistable orbit. By Theorem \ref{strictly_semistable_k-case-vgit}, the polynomials $f$, $g$, $h$ are generated by monomials in $N^{0}_t(\lambda, x^J, x_p)$ for some $(\lambda, x^J, x_p)$ such that $N^{\ominus}_t(\lambda, x^J, x_p)$ is maximal with respect to the containment of order of sets. Since there is a finite number of $\lambda \in P_{n,d,k}$ to consider (see, \ref{unstablelemma-vgit}), this is a finite computation which can be carried out by the software package \cite{theodoros_stylianos_papazachariou_2022}.

We have two possible choices for each $f_i$ and $h$ if $t \neq t_1,\dots, t_5$, four choices if $t = t_1, t_2$ and three choices if $t = t_3, t_4, t_5$. By Theorem \ref{t-centroid_criterion} we can check that the pairs $(S', D')$ and $(\Bar{S}, \Bar{D})$ given by 
\begin{equation*}
  \begin{split}
    f'(x_0,x_1,x_2,x_3,x_4) =& x_2^2+ x_4l_{1}(x_0,x_1)+x_3l_{2}(x_0,x_1)\\
    g'(x_0,x_1,x_2,x_3,x_4) =& x_2^2+ x_4l_{3}(x_0,x_1)+x_3l_{4}(x_0,x_1)\\
    H'(x_0,x_1,x_2,x_3,x_4) =& x_2
  \end{split}
\end{equation*}
and
\begin{equation*}
   \begin{split}
     \bar{f}(x_0,x_1,x_2,x_3,x_4)&= x_4l_1(x_0,x_1)+x_3l_2(x_0,x_1)+q_1(x_2,x_3,x_4)\\
     \bar{g}(x_0,x_1,x_2,x_3,x_4)&= x_4l_3(x_0,x_1)+x_3l_4(x_0,x_1)+q_2(x_2,x_3,x_4)\\
     \bar{H}(x_0,x_1,x_2,x_3,x_4) &= l(x_1,x_2,x_3)
   \end{split}
 \end{equation*}
are strictly $t$-semistable. Suppose that $\lambda, x^J, x_p  = (1,1,0,-1,-1)$, $x_0x_3$, $x_2$ or $(1,0,0,0,-1)$, $x_0x_4$, $x_1$. Then we have 
\begin{equation*}
  \begin{split}
    f_1 =& q_{1}(x_2,x_3,x_4)+ x_4l_{1}(x_0,x_1)+x_3l_{2}(x_0,x_1)\\
    g_1 =& q_{2}(x_2,x_3,x_4)+ x_4l_{3}(x_0,x_1)+x_3l_{4}(x_0,x_1)\\
    H_1 =& l(x_2,x_3,x_4)
  \end{split}
\end{equation*}
and 
\begin{equation*}
  \begin{split}
    f_2 =& q_{17}(x_1,x_2,x_3,x_4)+ x_4x_0\\
    g_2 =& q_{18}(x_1,x_2,x_3,x_4)+ x_4x_0\\
    H_2 =& l(x_1,x_2,x_3,x_4).
  \end{split}
\end{equation*}
Let $\gamma_1 = (1,1,0,-1,-1)$ and $\gamma_2 = (1,0,0,0,-1)$. Then we have, 
$$\lim_{t\to 0}\gamma_1(t)\cdot (S_1, D_1) = (S', D') \quad \lim_{t\to 0}\gamma_2(t)\cdot (S_2, D_2) = (\bar{S}, \bar{D}).$$
Hence, the closure of the orbit of $(S_1, D_1)$ contains
$(S', D')$, and the closure of the orbit of $(S_2, D_2)$ contains
$(\bar{S}, \bar{D})$, which we tackle next.

Let, for instance, $\lambda, x^J, x_p = (12, 2, -3, -3, -8)$, $x_0x_4$, $x_1$. Then
\begin{equation*}
  \begin{split}
    f(x_0,x_1,x_2,x_3,x_4) =& q_2(x_2,x_3)+ x_1x_4\\
    g(x_0,x_1,x_2,x_3,x_4) =& x_1^2+x_0x_4\\
    H(x_0,x_1,x_2,x_3,x_4) =& x_0
  \end{split}
\end{equation*}
and after a change of variables we may assume that 
\begin{equation*}
  \begin{split}
    f(x_0,x_1,x_2,x_3,x_4) =& x_0x_4+ x_1x_3\\
    g(x_0,x_1,x_2,x_3,x_4) =& x_1^2+x_2x_1\\
    H(x_0,x_1,x_2,x_3,x_4) =& x_2.
  \end{split}
\end{equation*}
We can do similar changes of variables in the rest of the cases and end up with $f$, $g$ and $g$ not depending on any parameters. Observe that since $(S, D)$ is strictly $t$-semistable,
the stabilizer subgroup of $(S, D)$, namely $G_{(S,D)} \subset \operatorname{SL}(5, \mathbb{C})$ is infinite (see \cite[Remark 8.1 (5)]{dolgachev_2003}). In particular, there is a $\mathbb{C}^*$-action on $(S, D)$. Lemma \ref{C star action invariance} classifies the singularities of $(S, D)$ uniquely according to their equations. For each $t \in (0, 1)$, the proof of the Theorem follows once we recall the classification of complete intersections of two quadrics in $\mathbb{P}^4$ according to their isolated singularities \ref{tab:p4segtable_with_equations}.
\end{proof}

The families presented below have been produced via the algorithm described in Section \ref{VGITsection}, via the computational package \cite{theodoros_stylianos_papazachariou_2022} and are all $t$-unstable with respect to the respective $t$ via the centroid criterion (Lemma \ref{t-centroid_criterion}). In addition, they are maximal $t$-destabilising families with respect to each wall/chamber $t$, in the sense of Definition \ref{destabilised sets def-vgit}. Notice, that we omit the characterisation of the GIT quotient for wall $t=0$. This is because for wall $t=0$, this is the GIT quotient $\mathcal{R}_{4,2,2}\sslash \operatorname{SL}(5)$ which has been classified in \cite{mabuchi_mukai_1990}. For more information on how one can obtain the same classification using the computational GIT package \cite{theodoros_stylianos_papazachariou_2022} we prompt the reader to the thesis \cite{my_phd_thesis} for a detailed analysis.

\section{CM line bundle for complete intersections and a hyperplane section}\label{cm_linebundle section}
Following the discussion in Section \ref{VGITsection} we aim to calculate the log CM line bundle for complete intersections with a hyperplane section, introduced in Section $2$. Let $S = \{f_1 = \dots=f_k=0\}$ be the complete intersection of $k$ hypersurfaces of degree $d$, with $f_i = \sum f_{I_{i,j}}x^{I_{i,j}}$, and $H$ a hyperplane in $\mathbb{P}^n$. We define the sets

\begin{equation}
  \begin{split}
  M_1 =& \Big\{(\bigwedge f_{I_{i,j}}x^{I_{i,j}}, l)\in \mathcal{R}| S = \{f_1 = \dots=f_k \}, H=\{ l = 0\}, \operatorname{Supp}(H)\subseteq \operatorname{Supp}(f_i)\text{ for some } i \Big\} \\
  M_2 =& \Big\{(\bigwedge f_{I_{i,j}}x^{I_{i,j}}, l)\in \mathcal{R}| S = \{f_1 = \dots=f_k \}, H=\{ l = 0\}, \exists H'\neq H, S\cap H' = S\cap H \Big\}.
  \end{split}
\end{equation}
Notice that by Theorem \ref{theorem-H not in supp S} the set $M_1$ contains only $t-$unstable elements for all $0\leq t\leq t_{n,d,k}$ for $kd \leq n$. 

\begin{lemma}\label{all in m2 unstable}
The elements of $M_2$ are $t-$unstable for all $0\leq t\leq t_{n,d,k}$, for $kd \leq n$. 
\end{lemma}
\begin{proof}
Let $(S,H)$ be a pair parametrised by $(\bigwedge f_i,l)\in M_2$, such that $S\cap H = S\cap H'$ where without loss of generality, we may assume that $H = \{x_n = 0\}$, $H' = \{x_{n-1}=0\}$. Without loss of generality, we can choose a coordinate system such that $S = \{f_1 = \dots = f_k\}$ is given by $f_i = x_n^d+x_{n-1}^d+x_nx_{n-1}f^i_{d-2}(x_0, \dots,x_n)$, $H = \{l(x_n,x_{n-1})$. Then for one-parameter subgroup $\lambda(s) = \operatorname{Diag}(s^2,\dots,s^2,0,s^{-n}, s^{-n})$, and for $0\leq t \leq t_{n,d,k}$ we have:
\begin{equation*}
  \begin{split}
    \mu(S,H,\lambda) <& k[(d-2)2-n]-tn\\
    \leq& 2kd - 4k - kn - kd\\
    <&0
  \end{split}
\end{equation*}
so the pair is $t$-unstable for all $0\leq t\leq t_{n,d,k}$.
\end{proof}

Let $\mathcal{T} \vcentcolon= \mathcal{R}\setminus (M_1 \cup M_2)$, with natural embedding $j \colon \mathcal{T}\rightarrow \mathcal{R}$. The above discussion shows:
\begin{lemma}
Let $kd \leq n$, $0\leq t \leq t_{n,d,k}$. Then $(\mathcal{T})_t^{ss} \vcentcolon= (\mathcal{R}\setminus (M_1 \cup M_1))_t^{ss} = (\mathcal{R})_t^{ss}$.
\end{lemma}
Thus $\mathcal{T}$ parametrises pairs $(\bigwedge a_{I_{i,j}},l)$ which correspond to complete intersections of $k$ homogeneous polynomials of degree $d$ and polynomials of degree $1$ respectively.

\begin{lemma}\label{codimension}
For $kd \leq n$, $d \geq 2$ we have:
\begin{enumerate}
  \item $\operatorname{codim}(M_1) = \binom{n+d-1} {d} -(k-1)^2 \geq 2$;
  \item $\operatorname{codim}(M_2)= k\Big(\binom{n+d}{d} - \binom{n+d-2} {d-2}\Big) -k^2 +n\geq 2$.
\end{enumerate}
\end{lemma}
\begin{proof}
$1$. Since $\operatorname{Supp}(H)\subseteq \operatorname{Supp}(f_i)$ for some $i$, without loss of generality we may assume $i=1$, and assuming $H = \{l(x_0,\dots,x_n)=0\}$ we can write $f_1 = l(x_0,\dots,x_n) f^i_{d-1}(x_0,\dots,x_n)$ $f_i = f^i_{d}(x_0,\dots,x_n)$. We first find the dimension of $M_1$: there are $n+1$ coefficients in the equation of $l$ and $\binom{n+d-1}{d-1}$ coefficients in $f^1_{d-1}$. Similarly, each $f_i$ has $\binom{n+d}{d}$ coefficients. Hence 
$\dim(M_1) = (k-1) \binom{n+d}{d} + \binom{n+d-1}{d-1} +n+1 -2k$, where we subtract $2$ degrees of freedom for each $f_i$. Then:
\begin{equation*}
  \begin{split}
    \operatorname{codim}(M_1) =& \dim(\mathcal{R}) - \dim(M_1)\\
    =& k\Big( \binom{n+d}{d}-k \Big)+n - (k-1) \binom{n+d}{d} - \binom{n+d-1} {d-1} - n -1 + 2k\\
    =& \binom{n+d-1}{d} -(k-1)^2 \\
    \geq& \Big( \frac{n+d-1}{d}\Big)^d - \big(\frac{n}{d}-1\big)(k-1)
    %\geq& k\big(k(\frac{n-1}{n})+1 \big)^2-k(n-1)-k^2\\
    %=& k^2\Big(k\big(\frac{n-1}{n}\big)^2 - 1\Big) +k(n-1)\Big(\frac{2k}{n}-1\Big)+1\\
    %\geq& 2.
  \end{split}
\end{equation*}

Now, since 

\begin{equation*}
  \begin{split}
    2 +\frac{n}{d}(k-1) \leq& 2+ \Big(\frac{n}{d}-1 \Big)^2\\
    =& \frac{3d^2+n^2-2nd}{d^2}\\
    \leq& \frac{(n+d-1)^2}{d^2}\\
    \leq& \frac{(n+d-1)^d}{d^d}
  \end{split}
\end{equation*}
we get $\operatorname{codim}(M_1) \geq 2$

$2$. Let $(\bigwedge a_{I_{i,j}}, l)\in H_1$, $S = \{f_1 = \dots=f_k=0\}$, $H = \{l=0\}$ and $H'\neq H$ such that $S\cap H = S\cap H'$. Without loss of generality, we assume that $H = \{x_n=0\}$, $H' = \{x_{n-1} = 0\}$. Since the Cartier divisor $S\cap H = S\cap H'$ is supported on $L = \{x_n = x_{n-1}=0\}\simeq \mathbb{P}^{n-2}$ and is a complete intersection of $k$ hypersurfaces of degree $d$ in $\{x_n=0\}\simeq \mathbb{P}^{n-1}$, we have that $S\cap H $ is $kdL$, hence we can write $f_i = a_ix_{n-1}^d+x_nf_{d-1}^i(x_0,\dots,x_n)$, and similarly for $H'$, $f_i = b_ix^d_{n}+x_{n-1}(f^i)'_{d-1}(x_0,\dots,x_n)$, which implies that we can write each $f_i = a_ix_n^d+b_ix^d_{n-1}+x_nx_{n-1}f_{d-2}^i$, and $l$ as $l(x_n,x_{n-1})$. Similar to the proof of $1$, there are $\binom{n+d-2}{d-2 }$ coefficients in each $f_{d-2}^i$, and $2$ coefficients in $l$. Hence, $\dim(M_2) = k\Big(\binom{n+d-2} {d-2 }+2\Big)-2k$. Thus:
\begin{equation*}
  \begin{split}
    \operatorname{codim}(M_2) =& \dim(\mathcal{R}) - \dim(M_2)\\
    =& k\Big( \binom{n+d}{d}-k \Big)+n - k\Big(\binom{n+d-2}{d-2} \Big) \\
    =& k\Big(\binom{n+d}{d} - \binom{n+d-2}{d-2}\Big) -k^2 +n \\
    \geq& k\Big(\binom{n+d}{d} - \binom{n+d-1}{d-1}\Big) -k^2 +n\\
    =& k\binom{n+d-1}{d} - k^2+n\\
    \geq& k\binom{n+d-1}{d} - k^2+kd\\
    \geq& k\binom{n+d-1}{d} - k^2+2k-1\\
    \geq& \binom{n+d-1}{d} - k^2+2k-1\\
    =& \operatorname{codim}(M_1)\\
    \geq& 2.
  \end{split}
\end{equation*}

\end{proof}
A direct consequence of the above lemma, is the result below, via  \cite[Prop. II.6.5.b]{hartshorne_2010}.
\begin{lemma}\label{picard_group_fam}
$$\operatorname{Pic}(\mathcal{T}) \simeq \operatorname{Pic}(\mathcal{R}) \simeq \mathbb{Z}^2,$$
and for $\mathcal{L} \in \operatorname{Pic}(\mathcal{U})$,
$$\mathcal{L}\simeq \mathcal{O}_{\mathcal{U}}(a,b) \vcentcolon= j^*(\mathcal{O}_{\mathcal{R}_{n,d,k}}(a) \boxtimes \mathcal{O}_{\mathcal{R}_{n,1}}(b))$$
\end{lemma}

We consider now the universal family of the complete intersection of $k$ hypersurfaces of degree $d$:

$$\pi_{n,d,k}\colon\mathcal{X}_{n,d,k}\rightarrow \mathcal{R}_{n,d,k}$$
where:
$$\mathcal{X}_{n,d,k} = \bigg\{(x_0,\dots,x_n)\times \bigwedge a_{I_{i,j}} \in \mathbb{P}^n \times \mathcal{R}_{n,d,k}| \sum a_{I_{1,j}}x^{I_{1,j}}= \dots = \sum a_{I_{k,j}}x^{I_{k,j}}=0  \bigg\}.$$

(Here, by abuse of notation, we denote the class $[\bigwedge a_{I_{i,j}}]$ by $\bigwedge a_{I_{i,j}}$). We then have a commutative diagram
\begin{center}
\begin{tikzcd}
\mathcal{X} \arrow[r] \arrow[d, "\pi"]
& \mathcal{X}_{n,d,k}\times \mathcal{R}_{n,1} \arrow[d, "\pi_{n,d,k}\times \operatorname{Id}_{\mathcal{R}_{n,1}}" ] \arrow[r]& \mathcal{X}_{n,d,k} \arrow[d, "\pi_{n,d,k}"] \\
\mathcal{T} \arrow[r, "j" ]
& \mathcal{R}\arrow[r, "p_1"] \arrow[d, "p_2"]& \mathcal{R}_{n,d,k}\\
& \mathcal{R}_{n,1} & 
\end{tikzcd}
\end{center}
with 
$$\mathcal{X} = \bigg\{ (x_0,\dots,x_n)\times \bigwedge a_{I_i}\times (b_0,\dots,b_n) \in \mathbb{P}^n \times \mathcal{U}|\sum a_{I_{1,j}}x^{I_{1,j}}= \dots = \sum a_{I_{k,j}}x^{I_{k,j}}=0 \bigg\}$$ the fiber product in the first diagram. Here, $j$ is the natural embedding and $p_i$ the projections. Since $\pi_{n,d,k}\colon\mathcal{X}_{n,d,k}\rightarrow \mathcal{R}_{n,d,k}$ is a universal family, it is flat and proper, and thus, by commutativity, $\pi$ is also flat and proper. 

Defining 

$$\mathcal{D}\vcentcolon=\bigg\{ (x_0,\dots,x_n)\times \bigwedge a_{I_i}\times (b_0,\dots,b_n) \in \mathcal{X}|\sum b_ix_i=0 \bigg\}$$
$\mathcal{D}$ is a Cartier divisor of $\mathcal{X}$, and the restriction $\pi_{\mathcal{D}}\colon \mathcal{D} \rightarrow \mathcal{T}$ is also flat and proper. This implies that $\pi\colon(\mathcal{X},\mathcal{D})\rightarrow \mathcal{T}$ is a $\mathbb{Q}$-Gorenstein flat family.

Notice, that in the Fano case, i.e. $kd \leq n$, $-K_{\mathcal{X}/\mathcal{T}}$ is relatively ample and by Theorem \ref{theorem-H not in supp S} and Lemma \ref{picard_group_fam} $\Lambda_{CM,\beta}(-K_{\mathcal{X}/\mathcal{T}}) \simeq \mathcal{O}(a,b)$. Hence, we can extend the CM-line bundle to $\mathcal{R}$: for $\beta \in (0,1) \cap \mathbb{Q}$
$$\Lambda_{CM,\beta}\vcentcolon= \Lambda_{CM,\beta}(-K_{\mathcal{X}/\mathcal{T}})\vcentcolon= \Lambda_{CM,\beta}(\mathcal{X},\mathcal{D},-K_{\mathcal{X}/\mathcal{T}}).$$
\begin{lemma}\label{calculating the line bundle}
Let $kd \leq n$, $\beta \in (0,1] \cap \mathbb{Q}$. Then $\Lambda_{CM,\beta} \simeq \mathcal{O}(a(\beta),b(\beta))$ 

where: 

\begin{equation*}
  \begin{split}
    %a(\beta) =& \big(n+1-kd \big)^{n-k-1}d^{k-1}\cdot\\
    % &\bigg(-\big(1+(n-k) (1-\beta)\big)\big(n+1-kd\big)\Big(-d(n-k+1)+(n+1-kd)\Big)\\ 
    % &+(n-k+1)(1-\beta)\big(-d(n-k)+(n+1-kd) \big) \bigg)>0\\
    a(\beta) =& \big(n+1-kd \big)^{n-k-1}d^{k-1}\cdot \bigg((n-k+1)(1-\beta)\big((1-d)n+1\big)\\
    &+\big(1+(n-k) (1-\beta)\big)\big(n+1-kd\big)(d-1)(n+1)\bigg)>0\\
    b(\beta) =& (n+1-kd)^{n-k}d^k(n-k+1)(1-\beta)>0,
  \end{split}
\end{equation*}
i.e. $\Lambda_{CM,\beta}$ is ample. In particular, if $kd = n$,
\begin{equation*}
  \begin{split}
    a(\beta) =& d^{k-1}\big(d(n-k+1)-\beta\big)\\
    b(\beta) =& d^k(n-k+1)(1-\beta)>0\\
    t(\beta) =& \frac{d(n-k+1)(1-\beta)}{d(n-k+1)-\beta}
  \end{split}
\end{equation*}
\end{lemma}
\begin{proof}

For $a$: Consider $k+1$ hypersurfaces of degree $d$ $f_1,\dots,f_{k+1}$. Then blowing up $\mathbb{P}^n$ along $C = f_1\cap f_2$ we obtain a map $\operatorname{Bl}_C\mathbb{P}^n\rightarrow \mathbb{P}^{1}$. Then we have a pencil of hypersurfaces of degree $d$ $\{a_1f_1+a_2f_2|(a_1,a_2)\in \mathbb{P}^{1}\}$ and we define $f_{1:2} \vcentcolon= \{a_1f_1 +a_2f_2 = 0\}$. Notice that the proper transform of $f_{1:2}$ is isomorphic to $f_{1:2}$. Intersecting $\tilde f_{3}, \dots, \tilde f_{k+1}$ with $\tilde f_{1:k}$ gives a well-defined family for complete intersections of $k$ hypersurfaces of degree $d$ over $\mathbb{P}^{1}$. Letting $\mathcal{S} = \tilde f_{3}\cap \dots\cap \tilde f_{k+1} \cap \tilde f_{1:2}$ we have a commutative diagram

\begin{center}
\begin{tikzcd}

\mathcal{S} \arrow[hookrightarrow]{r}{i} \arrow[d, "\pi"] & \mathcal{X}\arrow[d, "\pi"] \arrow[r] 
& \mathcal{X}_{n,d,k}\times \mathcal{R}_{n,1} \arrow[d, "\pi_{n,d,k}\times \operatorname{Id}_{\mathcal{R}_{n,1}}" ] \arrow[r]& \mathcal{X}_{n,d,k} \arrow[d, "\pi_{n,d,k}"] \\
\mathbb{P}^{1} \arrow[hookrightarrow]{r}{i} &
\mathcal{T} \arrow[r, "j" ]
& \mathcal{R}\arrow[r, "p_1"] \arrow[d, "p_2"]& \mathcal{R}_{n,d,k}\\
& & \mathcal{R}_{n,1} & 
\end{tikzcd}
\end{center}

For a general hyperplane $H$, let $p_H \in \mathcal{R}_{n,1}$ be the point that parametrises $H$. Then, we have $\mathcal{S}\times p_H \subset \mathbb{P}^1\times \mathbb{P}^n\times \mathcal{R}_{n,1}$. Notice that $\mathcal{S}$ is the one dimensional family of complete intersection of a hypersurface of bidegrees $(1,d)$ and $k-1$ hypersurfaces of bidegrees $(0,d)$ in $\mathbb{P}^{1}\times \mathbb{P}^n$.

Then, for a hyperplane $H$ parametrised by $P_H \in \mathcal{R}_{n,1}$ we have that
$$\mathcal{D} = \{p \in \mathbb{P}^{1}\times \mathbb{P}^n|p \in \mathcal{S}, p|_{\mathbb{P}^n}\in H \}$$
is a divisor obtained as the complete intersection of a hypersurface of bidegree $(1,d)$, and $k-1$ hypersurfaces of bidegrees $(0,d)$ and a hypersurface of bidegree $(0,1)$ in $\mathbb{P}^{1}\times \mathbb{P}^n$. For the corresponding projections $p_{\mathbb{P}^n},p'_{\mathbb{P}^1}$ let $H_{\mathbb{P}^n} = p^*_{\mathbb{P}^n}(\mathcal{O}_{\mathbb{P}^n}(1))$, $H_{\mathbb{P}^{1}} = (p')^*_{\mathbb{P}^{1}}(\mathcal{O}_{\mathbb{P}^{1}}(1))$. Then we have by adjunction:
\begin{equation*}
  \begin{split}
    K_{\mathcal{S}} =& (K_{\mathbb{P}^{1}\times \mathbb{P}^n}) +\mathcal{S})|_{\mathcal{S}}\\
    =& (-2H_{\mathbb{P}^{1}} -(n+1-d)H_{\mathbb{P}^{n}} + H_{\mathbb{P}^{1}} +(k-1)dH_{\mathbb{P}^{n}})|_{\mathcal{S}}\\
    =& \big(-H_{\mathbb{P}^{1}} +(kd-n-1)H_{\mathbb{P}^{n}}\big)|_{\mathcal{S}}
  \end{split}
\end{equation*}
Hence

\begin{equation*}
  \begin{split}
    K_{\mathcal{S}/\mathbb{P}^1} =& K_\mathcal{S} - \pi^*K_{\mathbb{P}^1}\\
    =& \big(-H_{\mathbb{P}^{1}} +(kd-n-1)H_{\mathbb{P}^{n}}\big)|_{\mathcal{S}} +(2 H_{\mathbb{P}^{k-1}})|_\mathcal{S}\\
    =& (H_{\mathbb{P}^{1}} +(kd-n-1)H_{\mathbb{P}^{n}})|_{\mathcal{S}}\\
    -K_{\mathcal{S}/\mathbb{P}^1} =& (n+1-kd)H_{\mathbb{P}^{n}}|_{\mathcal{S}}-H_{\mathbb{P}^{1}}|_{\mathcal{S}}
  \end{split}
\end{equation*}
Therefore, 
\begin{equation*}
  \begin{split}
    (-K_{\mathcal{S}/\mathbb{P}^1})^{n-k} &= (n+1-kd)^{n-k}H_{\mathbb{P}^{n}}^{n-k}|_{\mathcal{S}}-(n-k)(n+1-kd)^{n-k-1}H_{\mathbb{P}^{n}}^{n-k-1}|_{\mathcal{S}}\cdot H_{\mathbb{P}^{1}}|_{\mathcal{S}}\\
    (-K_{\mathcal{S}/\mathbb{P}^1})^{n-k+1} &= (n+1-kd)^{n-k+1}H_{\mathbb{P}^{n}}^{n-k+1}|_{\mathcal{S}}-(n-k+1)(n+1-kd)^{n-k}H_{\mathbb{P}^{n}}^{n-k}|_{\mathcal{S}}\cdot H_{\mathbb{P}^{1}}|_{\mathcal{S}}
  \end{split}
\end{equation*}
and since $\mathcal{S} = (H_{\mathbb{P}^{1}}+dH_{\mathbb{P}^{n}})\cdot (dH_{\mathbb{P}^{n}})^{k-1}$, $\mathcal{D}|_{\mathcal{S}} = H_{\mathbb{P}^{n}}$ we obtain 
\begin{equation*}
  \begin{split}
    c_1(-K_{\mathcal{S}/\mathbb{P}^1})^{n-k+1} =& \Big((n+1-kd)^{n-k+1}H_{\mathbb{P}^{n}}^{n-k+1}-(n-k+1)(n+1-kd)^{n-k}H_{\mathbb{P}^{n}}^{n-k}\cdot H_{\mathbb{P}^{1}} \Big)\\
    &\cdot\Big(H_{\mathbb{P}^{1}}+ dH_{\mathbb{P}^{n}}\Big)\cdot d^{k-1}H_{\mathbb{P}^{n}}^{k-1}\\
    =&-d^{k}(n-k+1)(n+1-kd)^{n-k}H_{\mathbb{P}^{n}}^{n}\cdot H_{\mathbb{P}^{1}}\\
    &+ d^{k-1}(n+1-kd)^{n-k+1}H_{\mathbb{P}^{n}}^{n}\cdot H_{\mathbb{P}^{1}}\\
    c_1(-K_{\mathcal{S}/\mathbb{P}^1})^{n-k}\cdot \mathcal{D}=&-d^{k}(n-k)(n+1-kd)^{n-k-1}H_{\mathbb{P}^{n}}^{n}\cdot H_{\mathbb{P}^{1}}\\
    &+d^{k-1}(n+1-kd)^{n-k}H_{\mathbb{P}^{n}}^{n}\cdot H_{\mathbb{P}^{1}}
    \end{split}
\end{equation*}
Using \cite[Theorem 2.7]{Gallardo_2020}, since $\mathcal{L} = -K_{\mathcal{S}/\mathbb{P}^1}$ and $\mathcal{D}|_{\mathcal{S}_i}\in |-K_{\mathcal{S}_i} |$, we have 
\begin{equation*}
  \begin{split}
    \deg((j\circ i)^*(\Lambda_{CM,\beta})) &= -\big(1+(n-k)(1-\beta)\big)\pi_* \Big(c_1(-K_{\mathcal{S}/\mathbb{P}^1})^{n-k+1} \Big) \\
    &+ (1-\beta)(n-k+1)\pi_* \Big(c_1(-K_{\mathcal{S}/\mathbb{P}^1})^{n-k}\cdot \mathcal{D} \Big) 
  \end{split}
\end{equation*}
%$$\deg((j\circ i)^*(\Lambda_{CM,\beta})) = -\big(1+(n-k)(1-\beta)\big)\pi_* \Big(c_1(-K_{\mathcal{S}/\mathbb{P}^1})^{n-k+1} \Big) + (1-\beta)(n-k+1)\pi_* \Big(c_1(-K_{\mathcal{S}/\mathbb{P}^1})^{n-k}\cdot \mathcal{D} \Big) $$
and since $\deg((j\circ i)^* (\Lambda_{CM,\beta})) =\deg((j\circ i)^* \circ p_1^*\mathcal{O}_{\mathcal{R}_{n,d,k}}(a) = a$ the result follows.

For $b$: Consider a hypersurface $S$ which is the complete intersection of $k$ hypersurfaces of degree $d$, $S = \{f_1, \dots, f_k\}$, represented by $p_S \in \mathcal{R}_{n,d,k}$, and pencil of hyperplanes $H(t)$, $t\in \mathbb{P}^1$. Then:

$$\mathcal{D}|_{p_S\times \mathbb{P}^1} = \{f_1=\dots=f_k = f_{H(t)}\} \subset p_S \times \mathbb{P}^n\times \mathbb{P}^1.$$ 
This implies that $\mathcal{D}|_{p_S\times \mathbb{P}^1} $ is the complete intersection of $k$ hypersurfaces of bidegree $(d,0)$ and one of bidegree $(1,1)$. Notice that $S\times \mathbb{P}^1/\mathbb{P}^1$ is a trivial fibration, so $c_1(-K_{S\times \mathbb{P}^1/\mathbb{P}^1})^{n-k+1} = 0$. We have:
$$K_{(S\times \mathbb{P}^1)/\mathbb{P}^1} = K_{S\times \mathbb{P}^1}-\pi^*K_{\mathbb{P}^1} = K_S\otimes \mathcal{O}_{\mathbb{P}^1} = (kd-n-1)H_{\mathbb{P}^n}|_S\otimes \mathcal{O}_{\mathbb{P}^1},$$
where $H_{\mathbb{P}^n} = \pi^*_{\mathbb{P}^n}(\mathcal{O}_{\mathbb{P}^n}(1))$. Hence, from \cite[Theorem 2.7]{Gallardo_2020}:

\begin{equation*}
  \begin{split}
  \deg(\Lambda_{CM,\beta}) =& (1-\beta)(n-k+1)c_1(-K_{(S\times \mathbb{P}^1)/\mathbb{P}^1})^{n-k}\cdot \mathcal{D}\\
  =& (1-\beta)(n-k+1)(n+1-kd)^{n-k}H_{\mathbb{P}^n}^{n-k}\cdot d^k H_{\mathbb{P}^n}^{k}\cdot(H_{\mathbb{P}^n}+H_{\mathbb{P}^1})\\
  =&(1-\beta)(n-k+1)(n+1-kd)^{n-k}d^kH_{\mathbb{P}^n}^{n}\cdot H_{\mathbb{P}^1}
\end{split}
\end{equation*}
i.e. $b=(n+1-dk)^{n-k}d^k(n-k+1)(1-\beta)>0$.

\end{proof}

\begin{corollary}\label{cm of ci in p4}
If $n=4$, $d=k=2$, $\beta \in (0,1] \cap \mathbb{Q}$ then $\Lambda_{CM,\beta} \simeq \mathcal{O}(a(\beta),b(\beta))$

where: 

\begin{equation*}
  \begin{split}
    a(\beta) =& 2(6-\beta)>0\\ 
    b(\beta) =& 12(1-\beta)>0,
  \end{split}
\end{equation*}
and $\Lambda_{CM,\beta}$ is ample. In particular
$$t(\beta) = \frac{b(\beta)}{a(\beta)} = \frac{6(1-\beta)}{6-\beta}.$$
\end{corollary}
\begin{remark}
A similar theorem is shown in \cite[Theorem 3.8]{Gallardo_2020} for the case $k=1$.
\end{remark}

We restrict ourselves to the case where $n=4$, $d=2$ and $k=2$, i.e. $S$ in the smooth case is a del Pezzo surface of degree $4$, and $D = S\cap H\in |-K_X|$ is an anticanonical divisor. Let $\Sigma$ be the $\operatorname{PGL}(5)$ orbits of the points of $\operatorname{Gr}(2, 15)$ corresponding to the subschemes $\{x_4x_{3}=x_4x_{2}=0\}$ and $\{x_4^2=x_4x_{3}=0\}$, which by \cite[Appendix]{spotti_sun_2017} is the non-equidimensional locus of the family. Then let
$$M_3 \coloneqq \Big\{(f_1\wedge f_2, l)\in \mathcal{R}| S = \{f_1 =f_2=0 \}\in \Sigma, H=\{ l = 0\}\Big\}. $$

% $$M_3\coloneqq \bigg\{ (S,H)| S\in \Sigma,\text{ } H \text{ a hyperplane section} \bigg\}\subset \mathcal{R}_{4,2,2}. $$

\begin{lemma}\label{all in M3 unstable}
    The elements of $M_3$ are $t$-unstable for all $0\leq t\leq t_{4,2,2} =1$.
\end{lemma}
\begin{proof}
    Let $(S,H)$ be a pair parametrised by $( f_i\wedge f_2,l)\in M_3$, such that $S\in \Sigma$. Let $\lambda$ be the diagonal one-parameter subgroup given by $\lambda(s) = \operatorname{Diag}(s,1,1,1,s^{-1})$. Since $H = \{l(x_0,x_1,x_2,x_3,x_4) = 0\}$ is a general hyperplane we have:
    $$\mu_t(S,H,\lambda) \leq \begin{cases} 
    -2 +t<0 & \text{if } S = \{x_4x_{3}=x_4x_{2}=0\} \\
    -3+t <0 & \text{if } S = \{x_4^2=x_4x_{3}=0\}
\end{cases}$$
for all $0\leq t\leq t_{4,2,2}=1$, so the pair is $t$-unstable.
\end{proof}
% We claim that all pairs $(S,H)\in M_3$ are unstable. Indeed, taking the diagonal one-parameter subgroup $\lambda(s) = \operatorname{Diag}(s,1,1,1,s^{-1})$, and a general hyperplane $H = \{l(x_0,x_1,x_2,x_3,x_4) = 0\}$ we have  
% $$\mu_t(S,H,\lambda) = \begin{cases} 
%     -2 +t<0 & \text{if } S = \{x_4x_{3}=x_4x_{2}=0\} \\
%     -3+t <0 & \text{if } S = \{x_4^2=x_4x_{3}=0\}
% \end{cases}$$
% for all $0\leq t\leq t_{4,2,2}=1$.
% \begin{definition}
%     Let 
% \end{definition}

\begin{proposition}\label{one-ps induces t.c.}
    Let $\mathcal{T'}\coloneqq \mathcal{T}\setminus M_3$, $t\in [0,1)$. Then:
    \begin{enumerate}
        \item If an element $P = (S,H) \in \mathcal{T}'$ is GIT $t$-semistable but not GIT $t$-polystable, then there exists a one-parameter subgroup $\lambda(s)$ such that the corresponding limit $P_0 = \lim_{s\to 0} \lambda(s)\cdot P$ is GIT $t$-polystable and is contained in $\mathcal{T}'$.
        \item Suppose an element $P = (S,H) \in \mathcal{T}'$ is GIT $t$-unstable, then there exists a one-parameter subgroup $\lambda(s)$ such that the corresponding limit $P_0 = \lim_{s\to 0} \lambda(s)\cdot P$ has negative GIT weight and is contained in $\mathcal{T}'$.
    \end{enumerate}
\end{proposition}
\begin{proof}
    For (1), we know by standard GIT that for such $P$, there is a one-parameter subgroup $\lambda(s)$, such that $P_0$ is GIT polystable. By Theorem \ref{theorem-H not in supp S} and Lemmas \ref{all in m2 unstable} and \ref{all in M3 unstable}, we know that all elements of $M_1\cup M_2\cup M_3$ are $t$-unstable for all t, hence $P_0$ must be contained in $\mathcal{T}'$.

    For (2), notice that if $S$ has non-isolated singularities, or if the discriminant polynomial $\Delta(S)$ is identically zero, by \cite[\S 8.6 and Table 8.6]{dolgachev_2012} these are given by \cite[Lemma 7.24 Family 1, Family 2, Family 7, Family 10]{my_phd_thesis}. Notice, that after a change of coordinates \cite[Lemma 7.24 Family 7 and Family 10]{my_phd_thesis} are contained in $M_1$, and \cite[Lemma 7.24 Family 1 and Family 2]{my_phd_thesis} belong to the set $M_2$. Hence, by Theorem \ref{theorem-H not in supp S} and Lemma \ref{all in m2 unstable} all such pairs $(S,H)$ are $t$-unstable for all $t$, for a general hyperplane. But, all the above pairs do not belong in $\mathcal{T}'$, and are thus excluded from this analysis. 
    
    We are only left to study $t$-unstable pairs $(S,H)$, where $\Delta(S)$ is not identically zero and $S$ has isolated singularities, i.e. $S$ has du Val singularities or is smooth. Notice, from Theorem \ref{main thm in p4 vgit} and \cite[\S 7.4]{my_phd_thesis} that if $S$ has du Val singularities, or is smooth, that there exists a value $t_0$ corresponding to each singularity, such that a pair $P = (S,H)$, where $S$ has at worse that type of singularity, is $t$-unstable for all values of $t<t_0$, then becomes $t$-strictly semistable for $t=t_0$ and is $t$-stable for all $t>t_0$. This means that for a prescribed one-parameter  subgroup $\lambda$, the weight function is $\mu_t(S,H,\lambda)<0$ for all $t<t_0$, $\mu_t(S,H,\lambda)=0$ for $t=t_0$ and $\mu_t(S,H,\lambda)>0$ for all $t>t_0$. In particular, at the value $t=t_0$, by (1), the corresponding limit $P_0 = \lim_{s\to 0} \lambda(s)\cdot P$ is GIT $t$-polystable and is contained in $\mathcal{T}'$. Thus, for each value $t<t_0$, the limit $P_0 = \lim_{s\to 0} \lambda(s)\cdot P$ exists, as taking limits is independent of the value of $t$, and the GIT weight is negative, since $\mu_t(S,H,\lambda)<0$ for all $t<t_0$. The same analysis as above follows for pairs that are $t$-stable for all values of $t<t_0$, then become $t$-strictly semistable for $t=t_0$ and are $t$-unstable for all $t>t_0$. Hence, (2) is satisfied for all GIT $t$-unstable elements $P = (S,H)\in \mathcal{T}'$, completing the proof.

\end{proof}

Since $\operatorname{codim}(\Sigma) \geq 2$, we have $\operatorname{codim}(M_3) \geq 2$, and hence $\Lambda_{CM, \mathcal{T}}(\beta)\cong \Lambda_{CM, \mathcal{T'}}(\beta)\cong \mathcal{O}(a,b)$, i.e. the CM line bundle on $\mathcal{T}'$ extends to the CM line bundle on the universal family $\operatorname{Gr}(2,15)\times \mathbb{P}^4$ naturally. Then, noting that our construction for $\mathcal{X}$ above is the same for $\mathcal{T}'$ we obtain the following Theorem.

\begin{theorem}\label{k-stab implies git stab}
Let $(S,D)$ be a log Fano pair, where $S$ is a complete intersection of two quadrics in $\mathbb{P}^4$ and $D$ an anticanonical section. Let $\pi \colon \mathcal{X}\rightarrow \mathcal{T}'$ the family introduced before, with ample log CM line bundle $\Lambda_{CM,\beta} \simeq \mathcal{O}(a(\beta),b(\beta))$. Suppose $(S, (1-\beta)D)$ is log $K$- (semi/poly)stable. Then, $(S,D)$ is GIT$_{t(\beta)}$-(semi/poly)stable, with slope $t(\beta) = \frac{b(\beta)}{a(\beta)} = \frac{6(1-\beta)}{6-\beta}$.
\end{theorem}
\begin{proof}

The proof follows the idea of proof in \cite[Theorem 3.4]{odaka_spotti_sun_2016} and \cite[Theorem 3.10]{Gallardo_2020}. Consider a one-parameter subgroup $\lambda$ acting on $p \in \mathcal{T}'$, representing a log pair $(S,H)$ with $H \not \subset S$, $D=S\cap H$. We have a natural projection $\pi \colon \overline{\mathcal{Y}}\vcentcolon= \overline{\lambda \cdot p}\subset \mathcal{T}'\times \mathbb{P}^1\rightarrow\mathbb{P}^1$. By abuse of notation, we extend $\pi$ to $\pi \colon \overline{\mathcal{Y}}\setminus\{\pi^{-1}(\infty)\}\rightarrow\mathbb{C}$, with $q \vcentcolon=\pi^{-1}(0)\in \mathcal{T}'$ the central fiber of $\pi$. Here $q$ is a pair $(\overline{S},\overline{H})$, where $\overline{S}$ is a complete intersection of two quadrics in $\mathbb{P}^4$, $\overline{H} \not \subset S$ a hyperplane and $\overline{D}=\overline{S}\cap \overline{H}\in | -K_{\overline{S}}|$ a hyperplane section. Since from Proposition \ref{one-ps induces t.c.} $\lambda$ induces a test configuration, we know from \cite[Theorem 2.6]{Gallardo_2020} that 
$$w(\Lambda_{CM,\beta}(\mathcal{X},\mathcal{D},\mathcal{L}^r)) = (n+1)!\operatorname{DF}_{\beta}(\mathcal{X},\mathcal{D},\mathcal{L})$$
hence if $(S, (1-\beta)D)$ is K-semistable then, since the one-parameter subgroup is arbitrary we obtain the result from Corollary %Corollaries \ref{corollary ch1 p4} and 
\ref{cm of ci in p4}, since 

$$w(\Lambda_{CM,\beta}(\mathcal{X},\mathcal{D},\mathcal{L}^r)) = \mu^{\Lambda_{CM,\beta}}(S,H,\lambda) = \mu_{t(\beta)}(S,H,\lambda).$$ 

Suppose now that $(S, (1-\beta)D)$ is K-polystable. Then in particular it is K-semistable, and the point $p \in \mathcal{T}$ is GIT$_{t(\beta)}$ semistable from the above discussion. Suppose that $p$ which parametrises $(S,H)$ is not GIT$_{t(\beta)}$ polystable. Then, by Proposition \ref{one-ps induces t.c.}.1 there exists a one-parameter subgroup $\lambda$ such that $\overline{p}= \lim_{t\to 0}\lambda(t)\cdot p$ is GIT$_{t(\beta)}$ polystable but not GIT$_{t(\beta)}$ stable. $\lambda$ induces a test configuration $(\mathcal{X},\mathcal{D},\mathcal{L})$ with $\operatorname{DF}_{\beta}(\mathcal{X},\mathcal{D},\mathcal{L})=0$ by \cite[Theorem 2.6]{Gallardo_2020}. Since $(S, (1-\beta)D)$ is K-polystable, we know that $(\mathcal{X}_p,\mathcal{D}_p)\simeq (S\times \mathbb{C}, D\times \mathbb{C})$. But then for the central fiber $(\overline{S},\overline{D})$ of the test configuration corresponding to $\overline{p}$, we have $(\overline{S},\overline{D})= (S,D)$, i.e. $\overline{p} = p$, and hence $p$ is GIT$_{t(\beta)}$ polystable.
\end{proof}

\begin{remark}
    We should point out that in the Thesis \cite{my_phd_thesis}, where these results first appeared, there is an error in the proof of Theorem \ref{k-stab implies git stab} (c.f. \cite[Theorem 8.6]{my_phd_thesis}). This is due to the fact that we try to apply \cite[Theorem 2.22]{ascher2019wall}, and the methods of \cite[Theorem 3.4]{odaka_spotti_sun_2016} and \cite[Theorem 3.10]{Gallardo_2020} directly, which is impossible, as the family $\mathcal{T}$ is not flat. We overcome this by following the Appendix of \cite{spotti_sun_2017} which had to deal with a similar issue in the non-pairs case, and proving Proposition \ref{one-ps induces t.c.}, which ensures that one-parameter subgroups induce test configurations for the same pairs. We thank both Junyan Zhao and Yuji Odaka for pointing this error out.
\end{remark}

\section{Proof of Main Theorem and the first wall crossing}\label{main theorem section}

Consider now a log Fano pair $(S,(1-\beta)D)$ where $S$ is a complete intersection of $2$ quadrics (degree $2$ hypersurfaces) in $\mathbb{P}^4$, $D$ is a hyperplane section and $\beta \in (0,1)\cap \mathbb{Q}$. 
%Then for $S, D$ smooth there exist conical K{\"a}hler-Einstein metrics on the log pair with positive Einstein constant if $\beta$ is greater than the log Calabi-Yau threshold. For the case where $(S,(1-\beta)D)$ is log-smooth, log-Fano by Chen-Donaldson-Sun \cite{chen_donaldson_sun_2013}, a sequence of singular KE metrics $(S_i,(1-\beta)D_i)$ with conical singularities of angle $2\pi\beta$ have convergent subsequences to $(S_{\infty}, (1-\beta)D_{\infty})$, which is klt and weak K{\"a}hler-Einstein Fano pair, where $S_{\infty}$ is $\mathbb{Q}$-Gorenstein smoothable Fano with $\deg(S_{\infty}) = \deg(S_i)$. %We first present the following result by Li-Liu \cite{li2017kahlereinstein} (Proposition $4.6$) which gives an algebro-geometric bound on volume.
We will consider $\mathbb{Q}$-Gorenstein smoothable K-semistable log Fano pairs $(S_{\infty},(1-\beta)D_{\infty})$ such that their smoothing is a log Fano pair $(S,(1-\beta)D)$ as above. We can think of these as  K-polystable limits of a degeneration family $\mathcal{X}$ of smooth K-polystable log Fano pairs, that have to be "added" to the boundary in order to compactify the K-moduli spaces.

%\begin{theorem}[{\cite[Proposition 4.6]{li2017kahlereinstein}}]\label{li-liu}
%For any K-semistable log Fano pair $(X,(1-\beta)D)$,
%$$\big(-K_X-(1-\beta)D\big)^n\leq \Big(1+\frac{1}{n}\Big)^n\hat{vol}_{(X,(1-\beta)D),p}$$
%for any point $p$ in $X$.
%\end{theorem}
%Here, in particular, $\hat{vol}_{(X,(1-\beta)D),p}$ is the normalised volume at the singularity $p \in X$, as seen in \cite{Li-Wang-Xu}. In particular, \cite{Gallardo_2020}, for any $\beta \in (1-\operatorname{lct}(S,D),1)$,
%$$\hat{vol}_{(X,(1-\beta)D),p} \leq \hat{vol}_{X,p}.$$

\begin{lemma}\label{gh convergence}
Let $(S_{\infty}, (1-\beta)D_{\infty})$ be a $\mathbb{Q}$-Gorenstein smoothable K-semistable log Fano pair such that its smoothing is a log Fano pair $(S_i,(1-\beta)D_i)$, where $S_i$ is a del Pezzo surface of degree $4$ (i.e. a smooth intersection of two quadrics in $\mathbb{P}^4$) and $D_i$ a smooth hyperplane section. For any $\beta > \frac{3}{4}$, $S_{\infty}$ is also an intersection of two quadrics in $\mathbb{P}^4$ whose singular locus consists of $\mathbf{A}_1$ or $\mathbf{A}_2$ singularities, and $D_{\infty}$ is also a hyperplane section.
\end{lemma}
\begin{proof}
Since each $D_i$ is a hyperplane section, $D_i \sim -K_{S_i}$ and hence, by continuity of volumes, the degree of the limit pair is 
$$(-K_{S_{\infty}}-(1-\beta)D_{\infty})^2 = (-K_{S_{i}}-(1-\beta)D_{i})^2 = (\beta K_{S_i})^2 = 4\beta^2.$$
By \cite[Theorem 3]{Liu-volume}, since $S_{\infty}$ is at worse klt, it must have only isolated quotient singularities isomorphic to $\mathbb{C}^2/G$, where $G$ is a finite subgroup of $U(2)$ acting freely on $S^3$. This is implied by the fact that klt surface singularities are precisely quotient singularities \cite[, Proposition 6.11]{clemens-collar-mori}, and that normal surfaces have only isolated singularities.

Moreover, the normal localised volume for quotient singularities is given by $\widehat{\operatorname{vol}}_{\mathbb{C}^2/G,0} = \frac{4}{|G|}$. Then, by \cite[Proposition 4.6]{li2017kahlereinstein}, and by \cite[Theorem 4.1]{Gallardo_2020} 
$$\widehat{\operatorname{vol}}_{(X,(1-\beta)D),p} \leq \widehat{\operatorname{vol}}_{X,p},$$
and hence, we have:

\begin{equation*}
  \begin{split}
    4\beta^2 = (-K_{S_{\infty}}-(1-\beta)D_{\infty})^2\leq& \Big(1+\frac{1}{2}\Big)^2\widehat{\operatorname{vol}}_{S_{\infty},(1-\beta)D_{\infty},p}\\
    \leq&\frac{9}{4}\widehat{\operatorname{vol}}_{\mathbb{C}^2/G,0}\\
    =& \frac{9}{|G|}
  \end{split}
\end{equation*}
i.e. $|G| \leq \frac{9}{4\beta^2}$.

%Similarly, for a del Pezzo surface of degree $4$, using the corresponding bound with $\beta=1$ we have $|G|\leq \frac{9}{4}$. Then for all $\beta > \frac{3}{4}$, we have $|G| < 4$.

By the classification of $\mathbb{Q}$-Gorenstein smoothable surface singularities \cite[Proposition 3.10]{kollar-shepherd}, $G$ must be a cyclic group acting in $\operatorname{SU}(2)$. Hence, the singularities of $S_{\infty}$ are canonical and, by the classification of del Pezzo surfaces with canonical singularities, %W is a cubic surface with at worst \mathbf{A}_1 or \mathbf{A}_2 singularities.
%This implies that $G$ is a cyclic group and $S_{\infty}$ has only canonical singularities by \cite[Proposition 4.22, Theorem 4.23]{kollar-shepherd}. Hence,
$S_{\infty}$ is a complete intersection of two quadrics in $\mathbb{P}^4$ with at worse $\mathbf{A}_1$ or $\mathbf{A}_2$ singularities. In particular, $D_{\infty}$ is a hyperplane section, as $D_{\infty} \sim -K_{S_{\infty}}\sim \mathcal{O}_{S_{\infty}}(1)$ by the adjunction formula.
\end{proof}

%(Possible way to circumvent code (idea): By Theorem \ref{adl thm 2.22} since the log CM line bundle is ample K- (poly/semi)stability of $(S,(1-\beta)D)$ implies GIT$_{t}$ (poly/semi)stability with $t = \frac{b}{a} = \frac{6(1-\beta)}{6-\beta}$. We consider $\mathcal{KM}_{\chi, r, (1-\beta)}$ as in Definition \ref{k-moduli stack} where we set $(X,D)\cong (S,H\cap S)$ with $S$ a del Pezzo surface of degree $4$. By a result of Prokhorov \cite{Prokhorov_2019} log canonical degenerations of del Pezzo surfaces of degree $4$ have $0,2,4$$\mathbf{A}_1$singularities. We know by the volume bound that this happens to the GH limit for $\beta > \frac{3}{\sqrt{10}}$. Thus for every $\beta > \frac{3}{\sqrt{10}}$ we have that $\mathcal{KM}_{\chi, r, (1-\beta)}$ is either isomorphic to a smooth, or to a singular complete intersection with $\mathbf{A}_1$ singularities. Since we have that for such $\beta$ and $t$ K- (poly/semi)stability of $(S,(1-\beta)D)$ implies GIT$_{t}$ (poly/semi)stability I think we can define the map $\phi \colon \overline{M}^{GH}_{4,\beta}\rightarrow \overline{M}^{GIT}_{t(\beta)}$ and use moduli continuity method to prove the result.)

The above Lemma allows us to prove the following:

\begin{theorem}\label{iso of stacks}
Let $\beta > \frac{3}{4}$. Then there exists an isomorphism of moduli stacks between the K-moduli stack $\mathcal{M}^K_{4,2,2}(\beta)$ of K-semistable families of $\mathbb{Q}$-Gorenstein smoothable log Fano pairs $(S, (1-\beta)D)$, where $S$ is a complete intersection of two quadrics in $\mathbb{P}^4$ and $D$ is an anticanonical section, and the GIT$_t$-moduli stack $\mathcal{M}^{GIT}_{4,2,2}(t(\beta))$. In particular, for $\beta>\frac{3}{4}$ we also have an isomorphism $M_{4,2,2}(\beta)\cong M^{GIT}_{4,2,2}(t(\beta))$ on the restriction to moduli spaces.
\end{theorem}
\begin{proof}

Let $\boldsymbol{\mathcal X} = (\mathcal{X}, \overline{\mathcal{X}})$ be the Hilbert polynomials of $(S, D)$, with $S$ a smooth complete intersection of two quadrics in $\mathbb{P}^4$, and $D$ an anticanonical divisor, pluri-anticanonically embedded by $-mK_{\mathbf{\mathcal{X}}}$ in $\mathbb P^N$, and let $\mathbb H^{\boldsymbol{\mathcal X}; N}\coloneqq\mathbb H^{\mathcal X; N}\times \mathbb H^{\overline{\mathcal X}; N}\coloneqq \mathrm{Hilb}_\mathcal X(\mathbb P^N)\times \mathrm{Hilb}_{\overline{\mathcal X}}(\mathbb P^N)$. 

Given a closed subscheme $X\subset \mathbb P^N$ with Hilbert polynomial $\mathcal X(X, \mathcal O_{\mathbb P^N}(k)|_X)=\mathcal X(k)$, let $\mathrm{Hilb}(X)\in \mathbb H^{\mathcal X; N}$ denote its Hilbert point. Let, as in \cite{ascher2019wall},
$$\hat Z_m\coloneqq\left\{ \mathrm{Hilb}(X,D)\in \mathbb H^{\boldsymbol{\mathcal X}; N} \;\middle|\; 
\begin{aligned}
  & X \text{ is a Fano manifold which is the complete intersection of}\\
  &\text{two quadrics in }\mathbb{P }^4, D \sim_{\mathbb{Q}}-K_X \text{ a smooth divisor,}\\
  &\mathcal O_{P^N}(1)|_X \sim \mathcal O_X(-mK_X),\\
  &\text{ and }H^0(\mathbb P^N, \mathcal O_{\mathbb P^N}(1)\xrightarrow{\cong} H^0(X, \mathcal O_X(-mK_X).
  \end{aligned}
\right\}$$
which is a locally closed subscheme of $\mathbb H^{\boldsymbol{\mathcal X}; N}$. Let $\overline Z_m$ be its Zariski closure in  $\mathbb H^{\boldsymbol{\mathcal X}; N}$ and $Z_m$ be the subset of $\hat Z_m$ consisting of K-semistable varieties.

By Theorem \ref{k-stab implies git stab} we know that if a pair $(S, (1-\beta)D)$ is K-(poly/semi)stable then $(S,D)$ is GIT$_{t(\beta)}$ (poly/semi)stable, where $t(\beta) = \frac{6(1-\beta)}{6-\beta}$. By \cite{odaka-moduli}, the smooth K-stable loci is a Zariski open set of $ M^K_{4,2,2}(\beta)$, in the definition of moduli stack of $\mathcal M^K_{4,2,2}(\beta)=[Z_m/\operatorname{PGL}(N_m+1)]$ for appropriate $m>0$ and in fact $\mathcal M^{GIT}_{4,2,2}(t(\beta))\cong[\overline Z_m/\operatorname{PGL}(N_m+1)]$. Hence, by Theorems \ref{main thm in p4 vgit}, \ref{main thm in vgit p4_polystable}, Corollary \ref{cm of ci in p4}, Lemma \ref{gh convergence}, for $\beta >\frac{3}{4}$ we have an open immersion of representable morphism of stacks:

\begin{center}
    \begin{tikzcd}
    \mathcal{M}^{K}_{4,2,2}(\beta)\arrow[r, "\phi"] & \mathcal{M}^{GIT}_{4,2,2}(t(\beta))\\
    \left[(S,(1-\beta)D)\right]\arrow[r, mapsto, "\phi"]& \left[(S,D)\right]
    \end{tikzcd}
\end{center}
% \begin{equation*}
%     \begin{split}
%         \mathcal{M}^{K}_{4,2,2}(\beta) &\xrightarrow{\phi} \mathcal{M}^{GIT}_{4,2,2}(t(\beta))\\
%         [(S,(1-\beta)D)] &\mapsto [(S,D)]
%     \end{split}
% \end{equation*}
with an injective descent $\overline{\phi}$ on the moduli spaces such that we have a commutative diagram

 \begin{center}
 \begin{tikzcd}
 \mathcal{M}^{K}_{4,2,2}(\beta) \arrow[r, "\phi"] \arrow[d] & \mathcal{M}^{GIT}_{4,2,2}(t(\beta)) \arrow[d] \\
 M^{K}_{4,2,2} \arrow[r, " \overline{\phi}"] & M^{GIT}_{4,2,2}(t(\beta)).
 \end{tikzcd}
 \end{center}

Note that representability follows once we prove that the base-change of a scheme mapping to the K-moduli stack is itself a scheme. Such a scheme mapping to the K-moduli stack is the same as a $\mathrm{PGL}$-torsor over $\overline Z_m$, which produces a $\mathrm{PGL}$-torsor over $ Z_m$ after a $\mathrm{PGL}$-equivariant base change. This $\mathrm{PGL}$-torsor over $Z_m$ shows the desired pullback is a scheme. By \cite[Lemma 06MY]{stacks-project}, since $\phi$ is an open immersion of stacks, $\phi$ is separated and, since it is injective, it is also quasi-finite.

% \begin{equation*}
%   \begin{split}
%     \phi \colon \mathcal{M}^{GIT}_{3,2,2} &\rightarrow \mathcal{M}^K_{2.25}\\
%     [C_1\cap C_2] &\rightarrow [\operatorname{Bl}_{C_1\cap C_2} \mathbb{P}^3]
%   \end{split}
% \end{equation*}

% \begin{equation*}
%   \begin{split}
%     \overline{\phi} \colon M^{GIT}_{3,2,2} &\rightarrow M^{K}_{\operatorname{Bl}_{C_1\cap C_2} \mathbb{P}^3}\\
%     [C_1\cap C_2] &\rightarrow [\operatorname{Bl}_{C_1\cap C_2} \mathbb{P}^3].
%   \end{split}
% \end{equation*}

We now need to check that $\phi$ is an isomorphism that descends (as isomorphism of schemes) to the moduli spaces. Now, by \cite[Prop 6.4]{Alper}, since $\phi$ is representable, quasi-finite and separated, $\overline \phi$ is finite and $\phi$ maps closed points to closed points, we obtain that $ \phi$ is finite. Thus, by Zariski's Main Theorem, as $\overline \phi$ is a birational morphism with finite fibers to a normal variety, $\phi$ is an isomorphism to an open subset, but it is also an open immersion, thus it is an isomorphism.

\end{proof}

\begin{corollary}[First Wall Crossing]\label{Wall crossing}
The first wall crossing occurs at $t(\beta)=\frac{1}{6}$, $\beta = \frac{6}{7}$. In particular there exists an isomorphism of moduli stacks between the K-moduli stack $\mathcal{M}^{K}_{4,2,2}(\frac{6}{7})$ parametrising K-semistable families of $\mathbb{Q}$-Gorenstein smoothable log Fano pairs $(S, \frac{1}{7}D)$, where $S$ is a complete intersection of two quadrics in $\mathbb{P}^4$ and $D$ is an anticanonical section, and the GIT$_t$-moduli stack $\mathcal{M}^{GIT}_{4,2,2}(\frac{1}{6})$. In particular, a log Fano pair $(S,\frac{1}{7}D)$ is log K-polystable if $S$ is a complete intersection of two quadrics with at $1$ $\mathbf{A}_2$ and $2$ or $1$ $\mathbf{A}_1$ singularities, and $D$ a singular hyperplane section which is a double line and two lines meeting at two points, or $S$ is a complete intersection of two quadrics with $2$ or $4$ $1$ $\mathbf{A}_1$ singularities, and $D$ a singular hyperplane section with $2$ or $4$ $1$ $\mathbf{A}_1$ singularities, respectively.
\end{corollary}
\begin{proof}
The proof follows directly from Theorems \ref{main thm in p4 vgit}, \ref{main thm in vgit p4_polystable} and \ref{iso of stacks}.
\end{proof}
\begin{remark}
In particular, there exists such an isomorphism up to the second chamber. After the second chamber, recent work of ours (joint with J. Martinez--Garcia and J. Zhao \cite{martinezgarcia2024kmodulilogdelpezzo}) extends the isomorphism to all walls and chambers using the methods of this paper.
\end{remark}

\printbibliography
\end{document}